\documentclass[a4paper]{article}
\usepackage{amsfonts}
\usepackage{pgfplots}
\usepackage{amsmath}
\usepackage{calc}
\usepackage{amsthm}
\usepackage{amssymb}
\usepackage{hyperref}
\usepackage{authblk}
\usepackage{autonum}
\usepackage[utf8]{inputenc}
\usepackage{relsize}
\usepackage[T1]{fontenc}
\usepackage[normalem]{ulem}

\usepackage{tikz}
\usepackage{tikz-cd}
\usepackage[left=2.0cm,top=3cm,right=2.0cm,bottom=3.0cm,bindingoffset=0.5cm]{geometry}

\pgfplotsset{compat=newest}

\hypersetup{%
  colorlinks=false,
  linkbordercolor=red,
  pdfborderstyle={/S/U/W 0.3}
}

\newtheorem{theorem}{Theorem}
\newtheorem*{theorem*}{Theorem}
\newtheorem{corollary}[theorem]{Corollary}
\newtheorem{proposition}[theorem]{Proposition}
\newtheorem{lemma}[theorem]{Lemma}
\newtheorem{definition}[theorem]{Definition}

\newtheorem{hypothesis}[theorem]{Hypothesis}

\newtheorem{remark}[theorem]{Remark}

\newcommand{\stir[2]}{\genfrac{\{}{\}}{0pt}{#1}{#2}}
\newcommand{\calO}{\mathcal{O}}
\newcommand{\OK}{\mathcal{O}_K}
\newcommand{\IC}{\mathbb{C}}

\DeclareMathOperator{\Res}{Res}
\DeclareMathOperator{\vol}{vol}
\DeclareMathOperator{\ord}{ord}
\DeclareMathOperator{\rank}{rank}
\DeclareMathOperator{\spant}{span}
\DeclareMathOperator{\ind}{\mathlarger{\mathbf{1}}}

\DeclareMathOperator{\N}{N}
\DeclareMathOperator{\Tr}{tr}
\DeclareMathOperator{\Id}{Id}
\DeclareMathOperator{\SL}{SL}
\DeclareMathOperator{\fS}{\mathfrak{S}}
\DeclareMathOperator{\GL}{GL}
\DeclareMathOperator{\Gr}{\mathbf{Gr}}
\DeclareMathOperator{\Pro}{\mathbf{P}}

\DeclareMathOperator{\card}{\normalfont \texttt{\#}}

\begin{document}
\title{Moments of the number of points in a bounded set for number field lattices}

\author{Nihar Gargava, Vlad Serban and Maryna Viazovska}
\date{\vspace{-5ex}}
\maketitle

\begin{abstract}
We examine the moments of the number of lattice points in a fixed ball of volume $V$ for lattices in Euclidean space which are modules over the ring of integers of a number field $K$. 
In particular, denoting by $\omega_K$ the number of roots of unity in $K$, we show that for lattices of large enough dimension the moments of the number of $\omega_K$-tuples of lattice points converge to those of a Poisson distribution of mean $V/\omega_K$. This extends work of Rogers for $\mathbb{Z}$-lattices. What is more, we show that this convergence can also be achieved by increasing the degree of the number field $K$ as long as $K$ varies within a set of number fields with uniform lower bounds on the absolute Weil height of non-torsion elements.
\end{abstract}

\tableofcontents

\section{Introduction}
A classical result in the geometry of numbers due to C.L. Siegel \cite{Sie45} establishes a mean value theorem for lattice sum functions $F_f(\Lambda)=\sum_{\lambda\in\Lambda\setminus\{0\}}f(\lambda)$, where $f: \mathbb{R}^t\to\IC$ is integrable and decays sufficiently fast. More precisely, the space $\SL_t(\mathbb{R})/\SL_t(\mathbb{Z})$ of unimodular lattices in $\mathbb{R}^t$ carries a canonical Haar measure of total mass one.  Viewing $F_f(\Lambda)$ as a random variable on that space, Siegel \cite{Sie45} proved the mean value theorem $$\mathbb{E}[F_f(\Lambda)]=\int_{\mathbb{R}^t} f(x) dx.$$
In particular, when $f$ is the indicator function of a bounded convex body $F_f(\Lambda)$ counts non-trivial lattice points and the famous Minkowski--Hlawka theorem \cite{h43} can be deduced in this way. Various refinements of this approach imposing extra structure have since appeared, in particular for lattices coming from maximal orders in number fields or $\mathbb{Q}$-division rings \cite{va2011,AV,G21}. The additional structure can often be leveraged for suitable applications; for instance A. Venkatesh in \cite{AV} deduces the currently best asymptotic lower bounds on the sphere packing density in high dimensions by working with cyclotomic integers. 

In a series of papers \cite{Rogers55,R1955,R1956}, C.A. Rogers established roughly a decade after Siegel formulas for the higher moments of $\mathbb{Z}$-lattices and explicitly evaluated those formulas when the lattice sum function is counting non-trivial lattice points in a bounded convex set. More precisely, Rogers obtains in \cite[Theorem 3]{R1956}: 

\begin{theorem*}{\bf (Rogers, 1956) }
Let $\Lambda  \subseteq \mathbb{R}^{t}$ 
be a random unit covolume lattice
and let $S$ be a centrally symmetric Borel set of volume $V$. Consider the random variable
\begin{equation}
\rho(\Lambda) :=F_{\ind_S}(\Lambda)= \card \left( S \cap ( \Lambda \setminus \{0\} )\right).
\end{equation}

Then, provided the $\mathbb{Z}$-rank $t$ of the lattices satisfies 
$t\geq \lceil \tfrac{1}{4}n^{2}+3\rceil,$
it follows that the $n$-th moment of the number of non-zero lattice points in $S$ satisfies
$$ 2^n \cdot m_{n}( \tfrac{V}{2} ) \leq \mathbb{E}[\rho(\Lambda)^n]\leq 2^n\cdot m_{n}( \tfrac{V}{2} ) +E_{n,t}\cdot (V+1)^{n-1},$$
where
\begin{equation}
	m_{n}(\lambda) = e^{-\lambda}\sum_{r=0}^\infty\frac{\lambda^{r}}{r!} r^{n}  = \mathbb{E}_{X \sim \mathcal{P}(\lambda) }(X^{n})
\label{eq:def_of_poisson}
\end{equation}
is the $n$th moment of a Poisson distribution with parameter $\lambda$
    and where $E_{n,t}$ is an error term 
    decaying exponentially as $t$ increases:
    $$E_{n,t} \le 2\cdot 3^{ \lceil \tfrac{n^2}{4}\rceil}\cdot ( \tfrac{\sqrt{3}}{2})^t+21\cdot 5^{ \lceil \tfrac{n^2}{4}\rceil}\cdot (\tfrac{1}{2})^{t}.$$
\end{theorem*}

In other words, Rogers showed that the number of pairs of non-trivial lattice points in $S$ has a distribution which approaches a Poisson distribution with mean $\tfrac{1}{2}V$ for large values of $t$. 
In particular, we obtain for large rank $t$ essentially $2\sqrt{t}$ point count estimates 
\begin{align}
 \mathbb{E} \left(\tfrac{1}{2} \rho(\Lambda)  \right)& =  \tfrac{1}{2}\vol(S) , \\
 \mathbb{E} \left(( \tfrac{1}{2} \rho(\Lambda) )^{2}  \right)& =  \left( \tfrac{1}{2}\vol(S) \right)^{2} + \left( \tfrac{1}{2} \vol(S) \right) + o(1), \\
  \mathbb{E} \left(( \tfrac{1}{2} \rho(\Lambda) )^{3}  \right)& =  \left( \tfrac{1}{2}\vol(S)  \right)^{3} + 3\left( \tfrac{1}{2}\vol(S) \right)^{2} + \left( \tfrac{1}{2}\vol(S) \right) +  o(1), \\
	      & \mathrel{\makebox[\widthof{=}]{\vdots}}
\end{align}
which are valid independently of $\vol(S)$. 
Note that the polynomials appearing on the right hand side are Touchard polynomials\footnote{For the first moment, Siegel's theorem tells us that the error term is exactly zero.}
in $\tfrac{1}{2}\vol(S)$ and that the appearance of the fraction $\tfrac{1}{2}$ on either side of the estimates results from the symmetries of $\pm 1$ acting on all lattice vectors.

It seems natural to ask whether similar higher moment results hold for lattices with additional structure, or whether the behaviour is qualitatively different. For a number field $K$ the ring of integers $\OK$ can be seen via the Minkowski embedding as a lattice in $K\otimes_\mathbb{Q}\mathbb{R}\cong \mathbb{R}^{[K:\mathbb{Q}]}$. Thus, any free $\OK$-module of finite rank $t$ produces (after possible scaling) a unimodular lattice in 
the space $\SL_{t}(K\otimes_\mathbb{Q}\mathbb{R})/\SL_{t}(\OK)$, which also comes equipped with a canonical probability measure.  We study higher moments of the number of points in a bounded convex set for such $\OK$-lattices. \par

A first observation is that, assuming that $S$ is symmetric about the origin, the finite order units in $\OK$ act freely on the lattice points in $S$ and thus lattice points should come in $\omega_K$-tuples instead of pairs, where $\omega_K$ denotes the root number of $K$. As a consequence of our main theorems, we are indeed able to show that for balls $S$ the number of $\omega_K$-tuples of $\calO_K$-lattice points in $S$ have a distribution asymptotic to a Poisson distribution with mean $\tfrac{1}{\omega_K}V$:

\begin{theorem}
 Let $K$ be any number field and let $n$ be fixed. 
 Then there is an explicit constant $t_0(K,n)=O_K(n^3\log\log n)$ 
 such that the $n$-th moment $ \mathbb{E}[\rho(\Lambda)^n]$ of the number of nonzero lattice points lying 
 in an origin-centered ball of volume $V$ and a random unit covolume $\OK$-lattice of rank $t$ 
 satisfies
\begin{align}
       \omega_K^n \cdot   m_n( \tfrac{V}{\omega_K}  )
     &\leq \mathbb{E}[\rho(\Lambda)^n]
       \leq 
       \omega_K^n \cdot  m_n( \tfrac{V}{\omega_K}  )
       + E_{n,t,K} \cdot (V+1)^{n-1} 
       \\
    \end{align}
    with the error term 
    \begin{equation}
      E_{n,t,K} \le 
       C_K\cdot t^{(n-2)/2} \cdot e^{-\varepsilon_K\cdot (t-t_0)}
    \end{equation}
    provided that $t> t_0(K,n)$. 
    Here $m_n$ is as defined in Equation (\ref{eq:def_of_poisson}) and the ball of volume $V$ is with respect to the Euclidean norm given in Equation (\ref{eq:norm}). The constants $C_K,\varepsilon_K>0$ are uniform in the rank $t$ of the $\OK$-lattices and can also be explicitly described.
\end{theorem}
An expression for the explicit constants $t_0(K,n)$ and $\varepsilon_K$ in terms of $n$ and the geometry of the unit lattice of $K$ can be found in Corollary \ref{cor:PoissonforKfixed}. \par
Rogers' results rely on his integral formula for the $n$-th moments. Such a result is also available in the context of $\OK$-lattices and implicit in the literature. For instance, S. Kim \cite{K19} establishes an integral formula in the adelic language and deduces convergence of the second moment. See also, e.g., \cite{WeilSiegelformula} and \cite[Theorem 1]{hughes2023mean}. For the reader's convenience, we derive it explicitly in Appendix \ref{ap:integralformula}. However, one of the main challenges arising for general number fields is dealing with infinite unit groups in $\calO_K$ and bounding their contributions (see \ref{th:weil} for the integral formula). We remedy this by employing lower bounds on the Weil height of units $\OK^\times$. In fact, height considerations allow us to prove stronger asymptotic results by increasing not just the $\calO_K$-rank of the lattices, but also the degree of the number field. 

More precisely, we show: 

	\begin{theorem}\label{thm:mainintro}
		Let $\mathcal{S}$ denote any set of number fields $K$ such that the absolute Weil height of elements in $K^\times\setminus\mu_K$ has a strictly positive uniform lower bound on $\mathcal{S}$. There are then for a given $n$ explicit constants $t_0(n,\mathcal{S})=O_\mathcal{S}(n^3\log\log n)$ as well as explicit constants $C,\varepsilon>0$, all uniform in $\mathcal{S}$, 
		such that for any $t> t_0$ and for any $K\in\mathcal{S}$ of degree $d$ the $n$-th moment $ \mathbb{E}[\rho(\Lambda)^n]$ of the number of nonzero $\calO_K$-lattice points in an origin-centered ball of volume $V$ and $\Lambda$ in the space of unit covolume $\OK$-lattices of rank $t$ 
		satisfies:
		\begin{align}
			\label{eq:small_error_from_poisson}
			\omega_K^n  \cdot m_n( \tfrac{V}{\omega_{ K} } )
	       &\leq \mathbb{E}[\rho(\Lambda)^n]
	       \le 
	       \omega_K^n\cdot  m_n( \tfrac{V}{\omega_{ K }} )
	       + E_{n,t,K} \cdot (V+1)^{n-1}
		\end{align}
		with error term $E_{n,t,K}$ satisfying
		\begin{align}
			E_{n,t,K} \le 
	       C \cdot (td)^{(n-2)/2}\cdot \omega_K^{n^2/4}\cdot Z(K,t,n)\cdot e^{-\varepsilon \cdot d(t-t_0)} .
		\end{align}
		Here $\omega_{K} $ are the number of roots of unity in $K$, $ Z(K,t,n)$ denotes a finite product of Dedekind zeta values $\zeta_{K}$ at certain real values $>1$ and 
		$m_n$ is as in Equation (\ref{eq:def_of_poisson}).
	\end{theorem}

See the more detailed Theorem \ref{thm:maingeneralmomentsBogo} for explicit values of the constant  $t_0(n,\mathcal{S})$, of the zeta factor $Z(K,t,n)$ and of the constants $C,\varepsilon$. Note that the terms $(td)^{(n-2)/2}\cdot \omega_K^{n^2/4}$ grow polynomially in $t,d$ since $\omega_K=O(d\log\log d)$ and the error term indeed decays exponentially in the dimension of the lattices. \par 

The height bound assumption on $\bigcup_{K\in\mathcal{S}}K$ in Theorem \ref{thm:mainintro} is in the literature referred to as the \emph{Bogomolov property}. A prototypical example of an infinite tower satisfying the Bogomolov property are the cyclotomic numbers $\mathbb{Q}^{cyc}=\bigcup_{i\geq 2} \mathbb{Q}(\zeta_i)$, so that the limiting results of Theorem \ref{thm:mainintro} in particular apply to lattices over cyclotomic integers of increasing degree for fixed large enough rank. In this case we can also bound the zeta factor uniformly-see Corollary \ref{cor:cyclomoments}. For the reader's convenience and as an illustration, we record here an entirely explicit ensuing second moment result over cyclotomic fields: 
\begin{corollary}\label{cor:introcyclo}
Consider a sequence of cyclotomic number fields given by $K_i=\mathbb{Q}(\zeta_{k_i})$ of degree $d_i=\varphi(k_i)$ and let $t_0=\tfrac{267}{10}$. 
There then exist uniformly bounded constants $C,\varepsilon> 0$ such that for any $t\geq 27$ and for any degree $d_i$ the second moment $\mathbb{E}[\rho(\Lambda)^2]$ of the  number of nonzero $\calO_K$-lattice points in an origin-centered ball of volume $V$ over the space of $\OK$-lattices of rank $t$ and unit covolume satisfies:

$$ V^2+V\cdot \omega_{K_i}\leq\mathbb{E}[\rho(\Lambda)^2]\leq V^2+V\cdot \omega_{K_i}\cdot (1+C\cdot e^{-\varepsilon\cdot d_i(t-t_0)}).$$
Moreover, the inequality holds for $\varepsilon=\frac{1}{400}$ and $C=(3+\tfrac{3}{1-e^{-d_i(t-t_0)/1124}})\cdot \zeta_{K_i}(\tfrac{37t}{52})\cdot \zeta_{K_i}(\tfrac{t}{25})$ for any given $t\geq 27$ and $d_i\geq 2$. In particular, $C\leq 5625\cdot \max_i(\zeta_{K_i}(\tfrac{37t}{52})\cdot \zeta_{K_i}(\tfrac{t}{25}))$ holds for all such $t,d_i$.
    
\end{corollary}

We refer the reader unfamiliar with heights and the Bogomolov property to the discussion in Section \ref{se:Bogosection} for details and other examples of infinite extensions with this property. We also partially prove in Proposition \ref{prop:morethanPoisson} the necessity of the height bound assumption in Theorem \ref{thm:mainintro}, showing that for any fixed rank $t$ there exist number fields $K_i$ of arbitrarily large degree with moments strictly larger than Poisson of mean $V/\omega_{K_i}$. Finally, we remark that similar results apply to more general convex sets $S$ beyond balls (see \ref{thm:genconvex}), however in pinning down the asymptotic distribution one needs to take into account the symmetry properties of the body $S$. \par 

In conclusion, we observe not just a limiting Poisson behaviour for the finer moduli space of $\OK$-lattices of fixed covolume, but also uncover additional flexibility in choosing the parameters of the Poisson distribution by varying the number of roots of unity in $K$. We therefore expect applications to the geometry of numbers and in particular hope to address the lattice packing and covering problems in the vein of \cite{AV},\cite{Rogers55},\cite{Rogers58},\cite{ORW22}.
Beyond that, integral formulas and higher moments have been employed among others in dynamics in the context of logarithm laws for flows on homogeneous spaces and Diophantine approximation (see e.g., \cite{Kleinbock1998LogarithmLF,ArthreyaMargulis}, \cite[Section 5]{Gorodnik2022poisson} and \cite{AggarwalGhosh2023central}). \par

Furthermore, $\OK$-lattices 
have emerged as interesting candidates for lattice-based cryptography (see e.g., \cite{RingLWE,IdealLatticeSVP,LLLOK}). 
	The setup in these works often resembles our line of investigation, even considering lattices of fixed $\mathcal{O}_K$-rank and varying cyclotomic number field $K$. In analysing the hardness of problems such as the shortest vector problem (SVP) on these restricted lattices, our results indicate a Poisson-like behaviour similar to the full probability space of random lattices, albeit with a different Poisson parameter.

\subsection{Outline of paper and proof}
The paper is organized as follows: 

Section \ref{se:average_over_lattice} presents the Rogers integral formula for $\OK$-lattices. 
Section \ref{se:convergence} then establishes convergence of the higher moments and includes some preparatory lemmas. Convergence can be deduced by relating moments to values of height zeta functions on suitable Grassmannians.  These converge by work of W. Schmidt \cite{S1967} on asymptotic counts for points of bounded height in such varieties. Section \ref{se:poisson} then deals with the main Poisson terms and some first estimates. 
\par 
Section \ref{se:Bogosection} tackles the general term and contains the main results. In order to go beyond just convergence of the moments, asymptotic estimates for points of bounded height are not sufficient, and one needs to have good control of the error terms for small heights as well. In order to illustrate how the results were achieved, we sketch our proof for the simple case of the second moment. In this case, via the integral formula the second moment computation for a fixed ball $S$ in $t$ copies of Euclidean space $K\otimes_\mathbb{Q}\mathbb{R}$ amounts to: 
$$\vol(S)^2+\sum_{\alpha\in K^\times}[\OK:(\alpha)^{-1}\cap\OK]^{-t}\cdot\vol(S\cap\alpha S).$$
To arrive at a result as in Corollary \ref{cor:introcyclo} it suffices to prove exponential decay of the sum
\begin{equation}\label{eq:introoutline}
    \sum_{\alpha\in (K^\times\setminus\mu_K)/\mu_K}[\OK:(\alpha)^{-1}\cap\OK]^{-t}\cdot\frac{\vol(S\cap\alpha S)}{\vol(S)},
\end{equation}
where $\mu_K$ denotes the roots of unity in $K$. In order to do so, we first bound for a fixed $\beta\in K^\times$ the shifted sum over units: 
$S_\beta=\sum_{\alpha\in (\OK^\times\setminus\mu_K)/\mu_K}\tfrac{\vol(S\cap\alpha\beta S)}{\vol(S)}$. The full result for \eqref{eq:introoutline} is then deduced by summing over principal ideals $(\beta)$ the quantity $[\OK:(\beta)^{-1}\cap\OK]^{-t}\cdot S_\beta$ and relating its decay to the decay of $S_\beta$ up to some Dedekind zeta values of $K$ (see e.g., Proposition \ref{prop:projsumoverideals}). In order to bound $S_\beta$, we use a geometric convex combination Lemma to show that the volume ratio $\tfrac{\vol(S\cap\alpha\beta S)}{\vol(S)}$ decays exponentially with the Weil height of $\alpha\beta$, see Lemma \ref{lemma:genvolumeratio} as well as Lemmas \ref{lemma:convexgen}, \ref{lemma:grtoproj}, and \ref{lemma:grtoprojallcols} for the more general case. The final ingredient is then a count of the number of units $\alpha\in\OK^\times$ such that $\alpha\beta$ has bounded Weil height. This is achieved in Lemma \ref{lemma:unitcountonecoord} using properties of heights and the unit lattice. Note that here it is really the points of \emph{small height} which have the weightiest contributions to $S_\beta$ and therefore we need genuine upper bounds on such counts as opposed to the classical asymptotic formulae for increasing height. We hope this also illustrates for the reader why height lower bounds for algebraic integers play an important role in our work. \par

For the $n$-th moment when $n>2$, there are additional complications. We must evaluate the sum  
\begin{equation}
\sum_{m=1}^{{n-1}}\sum_{\substack{D \in M_{m \times n }(K) \\ 
\rank(D) = m \\ 
D \text{ is row reduced echelon} }}
\mathfrak{D}(D)^{-t} \int_{x \in K\otimes\mathbb{R}^{t \times m }} \ind_{S^m}(x D ) dx,
\end{equation}
 where $\mathfrak{D}(D)$ is a measure of the denominators in $D$ extending $[\OK:(\alpha)^{-1}\cap\OK]$ for $(n,m)=(2,1)$. The main Poisson terms come from matrices $D$ with a single non-zero entry in $\mu_K$ per column (we denote this set by $A_m$). 
 While our overarching approach in estimating the error terms generalizing \eqref{eq:introoutline} is similar to the second moment, we now needed to distinguish several cases depending on the shape of $D$ - see Section \ref{subse:a2m} for details. The trickiest case are matrices $D$ \emph{close} to $A_m$, in that $D$ having entries of Weil height larger than some threshold $h_0\approx \tfrac{1}{2}\log n$ or having many non-vanishing $m\times m$ minors makes estimates easier. However, the remaining cases then constitute a finite set of matrices $D$ with entries of height bounded by $h_0$ and having at least one column which differs from columns in the main terms $A_m$. This is just enough to push through our results (see Proposition \ref{prop:GrWeilunitsum}) and obtain suitable exponential decay of each error term. 

\par
Finally, Section \ref{se:oddsends} adds some concluding remarks on height assumptions and more general bodies. 

\section*{Acknowledgements}

We would like to thank Matthew DeCourcy-Ireland, Gabrielle De Micheli,  Gauthier Leterrier and Philippe Michel for helpful conversations on topics related to this paper. We also thank Alexander Gorodnik for his insights concerning equidistribution of Hecke points and Barak Weiss for discussions about lattice coverings. 
We are grateful for Andreas Strömbergsson's important remarks on a previous version of this paper which have greatly benefited this text.
Finally, we thank Danylo Radchenko for supplying a major part of the argument in Lemma \ref{lemma:cyclotomiczetabound}.

This research was funded by the Swiss National Science Foundation (SNSF), Project funding (Div. I-III), "Optimal configurations in multidimensional spaces",
184927. 
Part of this research was carried out by the first named author during their stay at Quantinuum Ltd, Cambridge, UK.

\section{Rogers integration formula for number fields}
\label{se:average_over_lattice}

Let $K$ be a number field and let $\mathcal{O}_K$ denote its ring of integers. Let $K_\mathbb{R} = K \otimes_{\mathbb{Q}} \mathbb{R}$ denote the $[K:\mathbb{Q}]$-dimensional Euclidean space associated to $K$ and let $\overline{( \  )}: K_{\mathbb{R}} \rightarrow K_{\mathbb{R}}$ be a positive-definite involution 
\footnote{ The standard choice is to consider the involution $\overline{(\ )}$ given by identifying $K_\mathbb{R} \simeq \mathbb{R}^{r_1}\times \mathbb{C}^{r_2}$ and defining $\overline{( \ )}$ to be the identity on the real places and complex conjugation otherwise.}
on $K_\mathbb{R}$ such that the following is a positive-definite real quadratic form on $K_{\mathbb{R}}$:
\begin{equation}
\langle x,y \rangle = \Delta_K^{-\frac{2}{[K:\mathbb{Q}]}} \Tr(x\overline{y}).
\label{eq:norm}
\end{equation}
Here $\Delta_K$ is the absolute value of the discriminant of the number field $K$. 
Note that the quadratic form makes $\mathcal{O}_K$ into a lattice in $K_\mathbb{R}$ and the normalization in Equation (\ref{eq:norm}) ensures it has unit covolume. 
When multiple copies $K_\mathbb{R}$ are considered, we will assume that the quadratic form is the sum of the quadratic forms from Equation (\ref{eq:norm}) on each copy. 
This quadratic form therefore defines a Lebesgue measure on any number of copies of $K_\mathbb{R}$. We shall study unimodular lattices in $K_\mathbb{R}^t$ which are $\mathcal{O}_K$-modules for some $t\geq 2$. More precisely, we consider the space of rank $t$ free unimodular $\mathcal{O}_K$-lattices $\SL_{t}(K_\mathbb{R})/\SL_{t}(\OK)$. When $K$ does not have class number one and $\mathcal{O}_K$ is not a principal ideal domain, one may wish to allow for more general $\mathcal{O}_K$-modules. The adaptation of our results to such a setting will be discussed in forthcoming work.  

As pointed out in the introduction, integral formulas for higher moments over number fields can be found in the literature. A. Weil vastly generalized Siegel's mean value theorem \cite{WeilSiegelformula}. One may recover a $n${th} moment formula from Weil's work by considering the algebraic group $G = \SL_{t}(K)$ acting on the left on the affine variety $M_{t \times n}(K)$ in Weil's setup as described in \S 5-12 of \cite{WeilSiegelformula}. Moreover, Appendix \ref{ap:integralformula} provides a self-contained derivation of the formula for the reader's convenience. We record the formulas which form the starting point for our work here. 
\begin{theorem} {\cite{K19,WeilSiegelformula,hughes2023mean} }
\label{th:weil}
For any number $t$ of copies of $K_\mathbb{R}$ let $g:K_{\mathbb{R}}^{t \times n} \rightarrow \mathbb{R}$ be an indicator function of a convex compact set and equip $K_\mathbb{R}^t$ with the measure as discussed around Equation (\ref{eq:norm}). Then, putting the Haar probability measure on $\SL_{t}(K_\mathbb{R})/\SL_{t}(\OK)$, we have that
\begin{equation}
\int_{\SL_{t}(K_\mathbb{R})/\SL_{t}(\OK)} \left(\sum_{v \in \gamma \mathcal{O}_K^{ t \times n}} g(v)  \right) d \gamma
=
	\sum_{m=0}^{{n}}\sum_{\substack{D \in M_{m \times n }(K) \\ 
\rank(D) = m \\ 
D \text{ is row reduced echelon}}}
\mathfrak{D}(D)^{-t} \int_{x \in K_\mathbb{R}^{t \times m }} g(x D ) dx,
\end{equation}
where $\mathfrak{D}(D)$ is the index of the sublattice 
$\{ C \in M_{1 \times m }(\mathcal{O}_K) \mid C \cdot D \in M_{1 \times n}(\mathcal{O}_K)\}$ in $M_{1 \times m}(\mathcal{O}_K)$. 
Here the right hand side could diverge (however, see Corollary \ref{co:convergence}), and the term at $m=0$ is understood to be $g(0)$.
\end{theorem}

What we will actually use as our higher moment formula in practice is the following version.
\begin{corollary}
\label{re:non-zero_terms}
In the same setting as Theorem \ref{th:weil}, we have 
\begin{equation}
\int_{\SL_{t}(K_\mathbb{R})/\SL_{t}(\OK)} \left(\sum_{v \in \gamma \mathcal{O}_K^{ t \times n}\setminus \{0\}} g(v)  \right) d \gamma
=
\sum_{m=1}^{{n}}\sum_{\substack{D \in M_{m \times n }(K) \\ 
\rank(D) = m \\ 
D \text{ is row reduced echelon} \\ D \text{ has no zero columns }}}
\mathfrak{D}(D)^{-t} \int_{x \in K_\mathbb{R}^{t \times m }} g(x D ) dx.
\end{equation}

\end{corollary}
\begin{remark}
\label{re:smooth_functions}
It is possible to replace the function $g$ above with any smooth function with a compact support or indicator functions of sets with ``nice'' boundary.
\end{remark}

\begin{remark}
\label{re:module_lattices}
 The proof given in Appendix \ref{ap:integralformula} easily generalizes to arbitrary module lattices in $K \otimes \mathbb{R}^t$ which are torsion-free (but not necessarily free) $\OK$-modules. The relevant moduli space is a union of $h_K$ homogeneous spaces $\SL_t(K_{\mathbb{R}})/\Gamma_i$ for some arithmetic groups $\Gamma_i$, where $h_K$ is the class number of $K$. The work of estimating the quantities in the integral formula carried through in the next sections therefore remains valid for non-principal components when $h_K\neq 1$. See \cite{DK22} for an overview of module lattices.
\end{remark}

\section{Convergence of the higher moment formula}
\label{se:convergence}

In this section, we explain how to establish convergence of the 
expression in the integral formula of Theorem \ref{th:weil} and thus convergence of the moments. 

For this purpose, it is sufficient to consider the case when $g$ is the indicator function of a unit ball in $K_\mathbb{R}^{t \times n}$, since if the integral is bounded in this case then it should be bounded for all $g \in \mathcal{C}_c(K_\mathbb{R}^{t})$. The formula can then be related to height zeta functions of Grassmannians and convergence follows from estimates of W. Schmidt on points of bounded height in Grassmannians \cite{S1967}. Of course, the more crucial case for us is when $g(x_1,x_2,\dots,x_n) = \ind_B(x_1)\ind_B(x_2) \dots $, where each $\ind_B$ is the indicator function of some ball $B \subseteq K_\mathbb{R}^{t}$, however we postpone this discussion for now.

\begin{lemma}
\label{le:volume_det}
Let $g$ be the indicator function as described in the preceding paragraph and let $V(d)$ denote the volume of a $d$-dimensional unit ball. If $D \in M_{m \times n}(K)$ is a full-rank matrix, then we have that 
\begin{equation}
\label{eq:ball_integral}
\int_{K_\mathbb{R}^{t \times m }} g(xD) dx = \det(D;M_{t \times m  }(\mathcal{O}_{K}))^{-1} V(m  t [K:\mathbb{Q}]).
\end{equation}
Here, we define $\det(D; M_{ t \times m}\left( \mathcal{O}_K\right))$ as the volume of the fundamental domain of the $mt[K:\mathbb{Q}]$-dimensional $\mathbb{Z}$-lattice $M_{t \times m}(\mathcal{O}_K) \cdot D$.
\end{lemma}
\begin{remark}
Equivalently, $\det(D;M_{t \times m}(\OK))$ is the $\left( mt [K:\mathbb{Q}] \right)$-dimensional volume of the image of a unit cube in $K_{\mathbb{R}}^{t \times m }$ via $x \mapsto xD$. This image is a parellelepiped in $K_\mathbb{R}^{t \times n } \simeq \mathbb{R}^{tn[K:\mathbb{Q}]}$. 
\end{remark}

\begin{proof} { \bf (of Lemma \ref{le:volume_det})}

Observe that by the definition of the Riemann integral
\begin{equation}
\int_{K_\mathbb{R}^{t \times m }}g(xD) dx = \lim_{\varepsilon \rightarrow 0 }  \varepsilon^{mt[K:\mathbb{Q}]} \left( \sum_{x \in M_{t \times m }(\OK)} g( \varepsilon \cdot x D) \right).
\end{equation}

The sum is now counting the number of lattice points of $\varepsilon M_{t \times m }(\OK)$ in the ball.
\end{proof}

\begin{lemma}
\label{le:height_relation}
Suppose $D \in M_{m \times n}(K)$. Then
\begin{equation}
\det(D ; M_{t \times m }(\mathcal{O}_K) )  =  \det(D; M_{1 \times m }(\mathcal{O}_K))^{t}.
\end{equation}
Here the left-hand side is the quantity described above and the right hand side is the analogous quantity computing the volume of the fundamental domain of $M_{1 \times m}(\mathcal{O}_K) \cdot D \subseteq M_{1 \times n}(K_{\mathbb{R}}) \simeq K_\mathbb{R}^{n}$.
\end{lemma}

\subsection{Heights of subspaces}
\label{ss:heights}

Let us recreate the height functions on $K$-subspaces of $K^{t}$ as given in \cite{S1967}. Consider the standard Pl\"ucker embedding
\begin{align}
\iota: \Gr(m,K^n) &  \rightarrow  \Pro\left({ \mathlarger \wedge}^{m} K^{n} \right)\\
\spant_{K}(w_1,w_2,\dots,w_m) & \mapsto  [ w_1 \wedge w_2 \wedge \dots \wedge w_m ] .
\end{align}
Here $\Pro({\mathlarger \wedge}^{m}K^{n})=\Gr(1,{\mathlarger \wedge}^{m}K^n)$ is the $m$th exterior product (over $K$) of the vector space $K^{n}$ and $w_1,\dots,w_m$ are some $K$-linearly independent vectors inside $K^{n}$. A constructive way to see this map is that if $S \in \Gr(m,K^{n})$ is generated by $w_1,\dots,w_m$ then $\iota(S)$ is the one-dimensional subspace generated by the $m \times m$ minors of the $n \times m $ matrix whose columns are $w_1,\dots,w_m$. We shall denote the norm of the fractional $\mathcal{O}_K$-ideal generated by $x_1,\dots,x_N \in K$ by: 
\begin{equation}\label{def:langlenorm}
    \N( \langle x_1,x_2,\dots,x_N \rangle):=\N(\OK x_1+\cdots +\OK x_N).
\end{equation}

Let $\sigma_1,\sigma_2,\dots,\sigma_{[K:\mathbb{Q}]} : K \rightarrow \mathbb{C}$ be all the complex embeddings of $K$. 
We can apply them coordinate-wise and lift them as $\sigma_1, \dots , \sigma_{N} : K^{N} \rightarrow \mathbb{C}^{N}$ for any $N \ge 1$. Now, for any projective space $\Pro(K^{N})$, we can define the $l^2$-height function as

\begin{align}
H: \Pro(K^{N}) \rightarrow & \  \mathbb{R}_{\ge 0}\\
[x_1,\dots,x_N ] \mapsto & \frac{1}{\N( \langle x_1,x_2,\dots,x_N \rangle) } \prod_{i=1}^{[K:\mathbb{Q}]} \sqrt{\sum_{j=1}^{N}  |\sigma_i(x_j)|^{2}}.
\label{eq:l2_heigh_on_pro}
\end{align}

We similarly define the $l^{\infty}$-height function: 
\begin{align}
H_W: \Pro(K^{N}) \rightarrow & \  \mathbb{R}_{\ge 0}\\
[x_1,\dots,x_N ] \mapsto & \frac{1}{\N( \langle x_1,x_2,\dots,x_N \rangle) } 
\prod_{i=1}^{[K:\mathbb{Q}]}\max_{j=1\dots N} |\sigma_{i}(x_j)|
\label{eq:linf_heigh_on_pro}
\end{align}
Observe that both the heights defined above are well-defined functions on $\Pro(K^{N})$.

Enumerating the size-$m$ subsets of $\{1,2,\dots,n\}$, we get an obvious identification 
$\Pro({\mathlarger \wedge}^{m} K^{n} ) \leftrightarrow \Pro(K^{\binom{n}{m}})$. Using this, we can define the height of a subspace in $\Gr(m,K^{n})$ to be
\begin{align}
H: \Gr(m,K^{n}) \rightarrow & \  \mathbb{R}_{\ge 0}\\
S  \mapsto & \ H(\iota(S)) \\
\label{eq:defi_of_height}
\end{align}
and similarly,
\begin{align}
H_W: \Gr(m,K^{n}) \rightarrow & \  \mathbb{R}_{\ge 0}\\
S  \mapsto & \ H_W(\iota(S)) \\
\label{eq:defi_of_height_inf}
\end{align}

Now, we are ready to state an important lemma, which is essentially Theorem 1 from \cite{S1967}.

\begin{lemma}
\label{le:equivalence_of_height}
Suppose $m \le n$. Let $D \in M_{m \times n}(K)$ be a full-rank row reduced matrix and let $S = D^T  K^{m}\in  \Gr(m,K^{n})$ be the $m$-dimensional subspace spanned by its rows. The height function $H$ from Equation (\ref{eq:defi_of_height}) satisfies
\begin{equation}
H(S) 
=  \det(D; M_{1 \times m }(\mathcal{O}_K))
\cdot \mathfrak{D}(D).
\end{equation}
Here $\det(D; M_{1 \times m }(\mathcal{O}_K))$ is as defined in Lemma \ref{le:height_relation} and $\mathfrak{D}(D)$ is as defined in Theorem \ref{th:weil}.
\end{lemma}
\begin{proof}
A proof is given for the reader's convenience in Appendix \ref{ap:three_page_proof}.
\end{proof}

\subsection{Relating the two types of heights}

The following gives a relationship between the two types of heights defined in this section. This will be useful for proving Poisson estimates later on in the paper.

\begin{lemma}
\label{le:lower_heigh_bound}
Let $x=[x_1,x_2,\dots,x_N] \in \Pro(K^{N})$. Then the following relation exists between heights defined in Equation (\ref{eq:l2_heigh_on_pro}) and (\ref{eq:linf_heigh_on_pro}):
\begin{equation}
H(x)^{2} \ge \left( H_W(x)^{\frac{2}{[K:\mathbb{Q}]}} + (N-1)\frac{M(x)^{ \frac{2 }{[K:\mathbb{Q}] (N-1)} }}{H_W(x)^{\frac{2}{[K:\mathbb{Q}](N-1)}} } \right)^{[K:\mathbb{Q}]} ,
\end{equation}
where 
\begin{equation}
\label{eq:defi_of_M}
M(x) = \frac{\N(x_1)\N(x_2) \dots \N(x_{N})}{\N(\langle x_1, x_2,\dots,x_N \rangle )^{N}}.
\end{equation}
Here $\N(x_i)$ denotes the norm of the ideal generated by $x_i$ and $N$ is any strictly positive integer.
\end{lemma}

\begin{proof}

Observe that the following is a convex function on $\mathbb{R}^{N}$:
\begin{equation}
(x_1,\dots,x_N) \rightarrow \log( e^{x_1} + e^{x_2} + \dots + e^{x_N}),
\end{equation} and hence we get that for $x_{ij} \ge 0$ \begin{equation} \prod_{j=1}^{r}\left(\sum_{i=1}^{N} x_{ij}\right) \ge \left( \sum_{i=1}^{N} \left( \prod_{j=1}^{r}x_{ij} \right)^{\frac{1}{r}} \right)^{r}.
\end{equation}

For maximum efficacy, before applying the above inequality, one should rearrange the inner sums in the decreasing order. So we add the assumption that for each $j$, $x_1j \ge x_{2j} \ge \dots \ge x_{rj}$. Now, using the arithmetic-mean-geometric-mean inequality on the last $N-1$ terms on each of the $r$ multiplicands, we get: 

\begin{align}
\prod_{j=1}^{r}\left(\sum_{i=1}^{N} x_{ij}\right) \ge \left( 
\sum_{i=1}^{N} \left( \prod_{j=1}^{r}x_{ij} \right)^{\frac{1}{r}} \right)^{r}
& \ge
\left(  \left( \prod_{j=1}^{r}x_{1j} \right)^{\frac{1}{r}} + (N-1) \left( \prod_{i=2}^{N}\prod_{j=1}^{r} x_{ij}\right)^{\frac{1}{r( N-1 )}} \right)^{r} \\
& = \left(  \left( \prod_{j=1}^{r}x_{1j} \right)^{\frac{1}{r}} + (N-1) \frac{\left( \prod_{i=1}^{N}\prod_{j=1}^{r} x_{ij}\right)^{\frac{1}{r( N-1 )}}}{ \left( \prod_{j=1}^{r} x_{1j}\right)^{\frac{1}{r(N-1)}} } \right)^{r}. \\
\end{align}

Now set $r=[K:\mathbb{Q}]$ and for each $r$ let $\{ x_{i1},x_{i2},\dots\}$ be the numbers $\{|\sigma(x_1)|^{2}, |\sigma(x_2)|^{2} , \dots \}$ written down in the decreasing order, with $\sigma:K \rightarrow \mathbb{C}$ being the $i$th embedding with respect to some enumeration. This way, we have that 
\begin{equation}
\prod_{j=1}^{r} x_{1j} = \prod_{\sigma:K \rightarrow \mathbb{C}} \max_{i=1\dots N} |\sigma(x_{i})|^{2} = H_W(x)^{2} N(\langle x_1, x_2, \dots, x_N \rangle )^{2}.
\end{equation}

So we reach the conclusion that 
\begin{equation}
H(x)^{2} \ge \left( H_W(x)^{\frac{2}{[K:\mathbb{Q}]}} + (N-1)\frac{M(x)^{ \frac{2 }{[K:\mathbb{Q}] (N-1)} }}{H_W(x)^{\frac{2}{[K:\mathbb{Q}](N-1)}} } \right)^{[K:\mathbb{Q}]} .
\end{equation}
\end{proof}


Concerning the quantity $M(x)$ we have: 

\begin{lemma}
\label{le:about_Mx}
Let $x = [x_1,\dots,x_N] \in \Pro(K^{N})$. The quantity $M(x)$ defined in Equation (\ref{eq:defi_of_M}) is an integer at least $1$ if and only if $x_1\cdots x_N\neq 0$ and zero otherwise. Moreover, $M(x)=1$ implies that $x_i\in \OK^{\times}$ for all $i$ (up to scaling, i.e. as an element of $\Pro(K^{N})$) and if $M(x)>1$ it equals the norm of a non-trivial ideal in $\OK$.
\end{lemma}
\begin{proof}
For a prime ideal $\mathcal{P} \subseteq \OK$, let $\nu_{\mathcal{P}}(x) \in \mathbb{Z}$ be the $\mathcal{P}$-adic valuation of $x \in K$. Then, we observe that 
\begin{equation}
\N(\langle x_1,\dots,x_N \rangle) = \prod_{\substack{\mathcal{P} \subseteq \OK \\ \mathcal{P} \text{ is prime} }} \N(\mathcal{P} )^{\min_{i =1 \dots N} \nu_{\mathcal{P}}(x_i)}.
\end{equation}
Note that the product is supported on finitely many primes. On the other hand
\begin{equation}
\N(x_1\dots x_N ) = \prod_{\substack{\mathcal{P} \subseteq \OK \\ \mathcal{P} \text{ is prime} }} \N(\mathcal{P} )^{\sum_{i =1 \dots N} \nu_{\mathcal{P}}(x_i)}.
\end{equation}
So we get that 
\begin{equation}
M(x) =  \prod_{\substack{\mathcal{P} \subseteq \OK \\ \mathcal{P} \text{ is prime} }} \N(\mathcal{P} )^{\sum_{i =1 \dots N} \nu_{\mathcal{P}}(x_i) - N \min_{i=1\dots N} \nu_{p}(x_i) }.
\end{equation}
All the exponents are positive integers. They are zero only when all the $\nu_{\mathcal{P}}(x_i)$ are equal to each other and this is only possible if they differ at most by units. 
\end{proof}

\subsection{Rational points of bounded height in Grassmannian varieties over number fields }

Lemma \ref{le:volume_det}, Lemma \ref{le:height_relation}, Lemma \ref{le:equivalence_of_height} and Corollary \ref{re:non-zero_terms} yield the following.

\begin{lemma}
\label{le:height_zeta}
Let $g$ be the indicator function $\ind_{B_R}$, where $B_R$ is a ball in $K_\mathbb{R}^{t \times n }$ of radius $R$. 
Then, we have that 
\begin{align}
\sum_{m=0}^{{n}}\sum_{\substack{D \in M_{m \times n }(K) \\ 
\rank(D) = m \\ 
D \text{ is row reduced echelon} }}
\mathfrak{D}(D)^{-t} \int_{x \in K_\mathbb{R}^{t \times m }} g(x D ) dx
	 \\
	 = 
1+ \sum_{m=1}^{{n}}
Z(t;\Gr(m,K^{n}),H)\cdot
V(m t [K:\mathbb{Q}]) R^{mt}.
\end{align}
Here $Z(t;\Gr(m,K^{n}),H)$ is the height zeta function defined as 
\begin{equation}
Z(t;\Gr(m,K^{n}),H) = \sum_{S \in \Gr(m,K^{n})} \frac{1}{H(S)^{t}}.
\end{equation}
\end{lemma}

To show the convergence of the right hand side in Theorem \ref{th:weil} (or Corollary \ref{re:non-zero_terms}), it is sufficient to show that all the height zeta functions in Lemma \ref{le:height_zeta} converge. The asymptotic growth of points of bounded height on these varieties has been established by Schmidt and we have from \cite[Theorem 3]{S1967}:
\begin{theorem}
{\bf (Schmidt, 1967)}
There exist constants $C_1, C_2 > 0$ depending only on $n,m,K$ such that 
\begin{equation}
C_1 T^{n}\le \card \{S \in \Gr(m,K^{n}) \mid H(S) \le T \} \le  C_2 T^{n}
\end{equation}
\label{th:schmidt}
\end{theorem}


\begin{corollary}
\label{co:abel_sum}
The height zeta functions $Z(t;\Gr(m,K^{n}),H)$ converge when $t\geq n+1$.
\end{corollary}
\begin{proof}
Define for $n \ge 1$
\begin{align}
a_l = \  & \card \{S \in \Gr(m,K^{n}) \mid H(S) \in [l-1,l)\}.
\end{align}
Then Theorem \ref{th:schmidt} and Abel's summation formula gives us that 
\begin{align}
\sum_{l=1}^{T} \frac{a_l}{l^{t}} & = \left( \sum_{l=1}^{T}a_l\right)T^{-t} + \int_{1}^{T} \left( \sum_{l=1}^{\lfloor x \rfloor} a_l\right) \frac{t}{x^{t+1}} dx\\
& \le C_2 T^{n-t} +  tC_2\int_{1}^{T} x^{n-t-1} dx
\end{align}
Here, the first term converges as $T \rightarrow \infty$ since $n-t \le -1$ and the second term also converges since $n-t-1 \le -2$
\end{proof}
\begin{corollary}
\label{co:convergence}
The higher moment formula as given in Theorem \ref{th:weil} converges for $t\geq n+1$. 
\end{corollary}

\section{Towards Poisson distribution}
\label{se:poisson}
Going beyond convergence, we now turn towards establishing the limiting Poisson distribution. 
Let $V(n)$ henceforth denote the volume of the unit ball in dimension $n\geq 1$.  The following identifies the main Poisson term and is an adaptation of \cite[Lemma 4]{R1956}:

\begin{lemma}
\label{le:poisson_term}
Let $\mu_K$ denote the cyclic group of roots of unity in $\OK$. Let $\omega_K = \card \mu_K$.
Consider the set $A_m$ for $m \in \{1,\dots,n\}$ given by
\begin{equation}
A_{m} = \left\{ D\in  M_{m \times n}( K) \ {\Big| } \substack{ \ D_{ij} \in  \mu_K \cup \{0 \},\\  D \text{ is in row-reduced echelon form of }\rank(D) = m \\ D \text{ has exactly one non-zero entry in each column}   } \right\}.
\end{equation}
Let $B \subseteq K_{\mathbb{R}}^{t}$ 
denote a ball with respect to the norm given in 
Equation \ref{eq:norm}. 
Let $g= \ind_B \otimes \dots \otimes \ind_B :K_\mathbb{R}^{t \times n} \rightarrow \mathbb{R}$ be the $n$-fold indicator function of the ball in each coordinate. Restricting the higher moment formula of Theorem \ref{th:weil} to matrices in $A_m$, we obtain that
\begin{equation}
\sum_{m=1}^{n} \sum_{D \in A_{m}} \mathfrak{D}(D)^{-t} \int_{x \in K_{\mathbb{R}}^{t \times m}} g(xD) dx = \omega_K^{n} \exp\left(-\frac{ 1  }{\omega_K} \cdot V(t[K:\mathbb{Q}])\right) 
\sum_{r=0}^{\infty} \frac{r^{n}}{r!} \left(\frac{1}{\omega_{K}} \cdot V(t [K:\mathbb{Q}]) \right)^{r}.
\end{equation}
\end{lemma}
\begin{proof}
Observe that for $D \in A_m$, we have $\mathfrak{D}(D)=1$. Next, observe that $B$ is invariant under the diagonal action of $\mu_K$ on $K_\mathbb{R}^{t}$ due to the choice of the quadratic form defining the ball. Therefore, for any units $\alpha_1,\alpha_2,\dots,\alpha_{n} \in \mu_K$ we have that
\begin{equation}
g(\alpha_1 x_1, \alpha_2 x_2 ,\dots, \alpha_n x_n) = g(x_1,x_2,\dots,x_n),
\end{equation}
where $\mu_K$ acts diagonally on $K_\mathbb{R}^{t}$ viewed as $t$ copies of $K_\mathbb{R}$. This implies that for any $D \in A_m$, we must have
\begin{equation}
\int_{x \in K_\mathbb{R}^{t \times m }} g(xD) dx = \vol( B  )^{m}.
\end{equation}

The combinatorial problem of counting $\card A_m$ is, up to multiplication by a power of $\omega_K$, the same as that of partitioning $n$ columns into $m$ sets. Therefore, we have that 
\begin{equation}
\card A_m  = \omega_{K}^{n-m} \stir{n}{m},
\end{equation}
where $\stir{n}{m}$ is the Stirling number of the second kind. Hence, setting $V= V(t[K:\mathbb{Q}])$ gives
\begin{equation}
\sum_{m=1}^{n}\sum_{D \in A_m} \int_{x \in K_\mathbb{R}^{t \times m }} g(xD) dx =   \sum_{m=1}^{n} V^{m} \omega_K^{n-m} \stir{n}{m} = \omega_K^{n}\sum_{m=1}^{n}  \stir{n}{m}\frac{V^{m}}{\omega_K^{m} } .
\end{equation}

Now we invoke the following identity about Touchard polynomials and we are done:
\begin{equation}
\sum_{m=1}^{n} \stir{n}{m} x^{m}  = e^{-x}\sum_{r=0} ^{\infty}\frac{r^{n}}{r!}x^{r}.
\end{equation}

\end{proof}

We now turn to studying and bounding the contributions of the rest of the terms. To that end, we introduce the following notations for the remainder of the paper: 
\begin{align}
A^1_{m} & = \left\{ D \in  M_{m \times n}( K) \ {\Big| } \substack{ \ D_{ij} \in  K  ,
\\  D \text{ is in row-reduced echelon form with }\rank(D) = m \\ 
\text{All the matrix entries are in $\mu_K\cup \{0\}$}   } \right\} \setminus 
 A_m . \\
A^2_{m} & = \left\{ D\in  M_{m \times n}( K) \ {\Big| } \substack{ \ D_{ij} \in  K  ,\\  D \text{ is in row-reduced echelon form of }\rank(D) = m \\ 
} 
\right\} 
\setminus\left( A_m^{1}\sqcup A_m\right).\\
\end{align}

Note that for $m=1$ we have that $A_m=A^1_m$. For $m\geq 2$ we record here a standard estimate on volume ratios: 

\begin{lemma}
\label{le:gamma_ratio}
We have the estimates on volume ratios: 
\begin{equation}
	\frac{V(mt[K:\mathbb{Q}])}{V(t[K:\mathbb{Q}])^{m}}
	= 
\frac{\Gamma\left(\frac{t[K:\mathbb{Q}]}{2}+1\right)^{m} }{ \Gamma\left(\frac{m t [K:\mathbb{Q}] }{2} + 1 \right)} 
<
\frac{(t[K:\mathbb{Q}]\pi )^{\frac{m-1}{2}}}
{m^{\frac{1}{2}mt[K:\mathbb{Q}]+\frac{1}{2}}}\cdot e^{\tfrac{m}{6t[K:\mathbb{Q}] }}.
\end{equation}
\end{lemma}
\begin{proof}
Straightforward by using honest upper and lower bounds in Stirling approximation.
\end{proof}

\subsection{Matrices of type \texorpdfstring{$A^{1}_m$}{A1m}}
In this subsection we obtain the following bound on the contribution of $A_{m}^{1}$-type terms.
These are the terms for which the geometrical methods utilized by Rogers \cite{R1956} generalize without much difficulty. The more delicate terms, involving contributions from unit entries of infinite order, will be dealt with in Section \ref{se:Bogosection}. 

\begin{theorem}
\label{th:poisson_dist}
Consider the setup of Lemma \ref{le:poisson_term}. 
Let $K$ be a number field and let $n<t$.
We then have that
\begin{align}
	V(t[K:\mathbb{Q}])^{-m} \sum_{m=1}^{n} \sum_{D\in A_m^1} \mathfrak{D}(D)^{-t} \int_{ K_\mathbb{R}^{t \times m}} g(xD) dx & \le 
	C ( \tfrac{\sqrt{3}}{2})^{t[K:\mathbb{Q}]},
\end{align}
where the constant does not depend on $n,m,K$. If the number field $K$ is also changing with $n,m$ fixed, the constant $C$ grows at most polynomially in $[K:\mathbb{Q}]$.

\end{theorem}
We record a trivial count which reduces the proof to bounding the contribution of each individual matrix:
\begin{lemma}
\label{le:counting_a1m}
We have that
\begin{equation}
	\sum_{m=1}^{n}\card A_{m}^{1} \le   \sum_{m=1}^{n} \stir{n}{m} (1+ \omega_{K} )^{(n-m)m}.
\end{equation}
\end{lemma}
The following result, an adaptation of \cite[Lemma 5]{R1956}, suffices for our purposes: 
\begin{lemma}
\label{le:lemma5_rogers}
Let $f:K_\mathbb{R}^{t}\rightarrow \mathbb{R}$ be the indicator function of a ball $B$ of radius $R > 0$. Then, for any $\alpha_1, \alpha_2 \in \mu_K$ and $a \in K_\mathbb{R}^{t}$, we have that 
\begin{align}
	\frac{1}{ V(t[K:\mathbb{Q}])^{2} R^{2t[K:\mathbb{Q}]} } 
 \int_{K_\mathbb{R}^{2 \times t }} f(x) f(y) f(\alpha_1 x + \alpha_2 y ) dxdy  
	\le  
2 (\tfrac{\sqrt{3} }{2})^{t[K:\mathbb{Q}]} .
\end{align}
\end{lemma}
\begin{proof}

Since $f$ is invariant under $\mu_K$, we can assume that the integral is 
\begin{equation}
	\int_{K_\mathbb{R}^{t \times 2} } f(x) f(y) f(\alpha y - x ) dxdy , \text{ for some $\alpha \in \mu_K$.} 
\end{equation}

We can rewrite the above as
\begin{equation}
  \int_{K_\mathbb{R}^{t}} f(y) \left( \int_{K_\mathbb{R}^{t}}  f(x) f(\alpha y -x )dx \right) dy.
\end{equation}
The inner term is the intersectional volume of two translates of $B$, one centered at the origin and the other at $\alpha y $. By doing some elementary geometry (see Figure \ref{fig:geometry}), one can see that 
\begin{equation}
	\int_{K_\mathbb{R}^{t}} f(x) f(\alpha y - x) dx =  2 V(N-1) R^{N}  \int_{\frac{1}{2}{\frac{ \|\alpha y \|}{R} }}^{1} (1- \rho^{2})^{\frac{N-1}{2}} d\rho,
\end{equation}
where $N = t[K:\mathbb{Q}]$ and $\rho$ is an integration parameter (see Figure \ref{fig:geometry}). We understand the right hand side to be $0$ if $\|\alpha y\| > 2 R$.

Substituting this in our expression gives 
\begin{align}
	& \left(2 V(N- 1) R^{t[K:\mathbb{Q}]}  \right) \int_{K_\mathbb{R}^{t}} f(y)\left(  \int_{\frac{1}{2} \frac{\|  y\|}{R}}  ( 1- \rho^{2})^{\frac{N-1}{2}} \right)dy \\
	& =  2V(N-1)V(N) R^{2N} N \int_{0}^{1} \xi^{N-1} \left( \int_{\tfrac{1}{2} \xi }^{1} (1-\rho^{2})^{\frac{N-1}{2}} d\rho  \right) d \xi.
\end{align}

Performing explicit computations as in \cite[Lemma 5]{R1956}, we find that this expression is bounded by 
\begin{align}
	& \le  2V(N)^{2} R^{2N}  (\tfrac{\sqrt{3}}{2})^{N}.
\end{align}

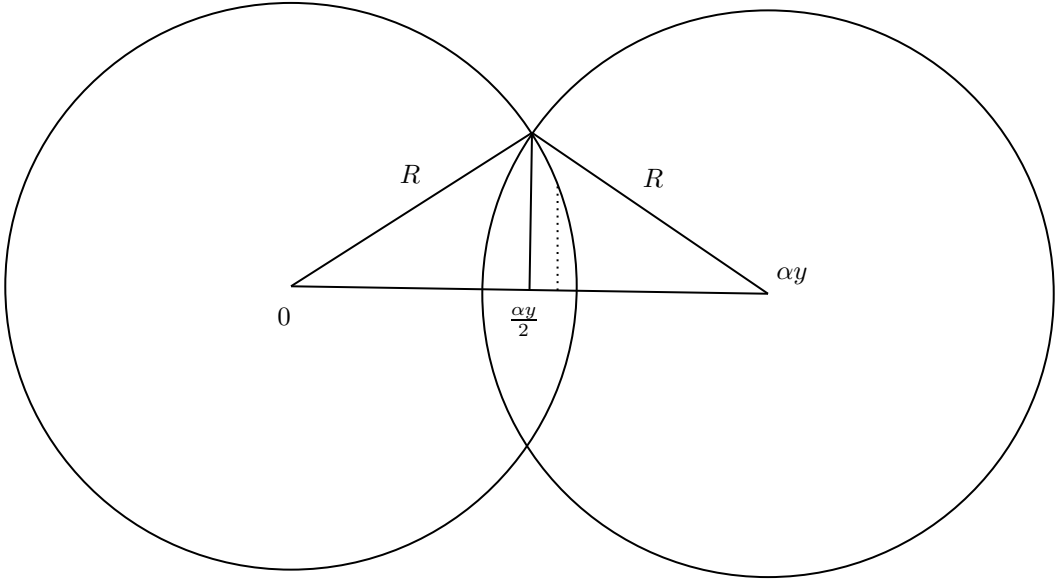
\begin{figure}[ht!]

\tikzset{every picture/.style={line width=0.75pt}} 
\centering
\begin{tikzpicture}[x=0.75pt,y=0.75pt,yscale=-1,xscale=1]

\draw   (327,166.1) .. controls (327,87.37) and (390.82,23.55) .. (469.55,23.55) .. controls (548.28,23.55) and (612.1,87.37) .. (612.1,166.1) .. controls (612.1,244.83) and (548.28,308.65) .. (469.55,308.65) .. controls (390.82,308.65) and (327,244.83) .. (327,166.1) -- cycle ;

\draw[dotted]    (364.55,112.11) -- (364.55,164.23) ;

\draw    (231.55,162.35) -- (469.55,166.1) ;
\draw    (231.55,162.35) -- (351.8,85.11) ;
\draw    (469.55,166.1) -- (351.8,85.11) ;
\draw    (351.8,85.11) -- (350.55,164.23) ;
\draw   (89,162.35) .. controls (89,83.62) and (152.82,19.8) .. (231.55,19.8) .. controls (310.28,19.8) and (374.1,83.62) .. (374.1,162.35) .. controls (374.1,241.08) and (310.28,304.9) .. (231.55,304.9) .. controls (152.82,304.9) and (89,241.08) .. (89,162.35) -- cycle ;

\draw (472.27,151.23) node [anchor=north west][inner sep=0.75pt]    {$\alpha y$};
\draw (223.27,171.67) node [anchor=north west][inner sep=0.75pt]    {$0$};
\draw (284.4,100) node [anchor=north west][inner sep=0.75pt]    {$R$};
\draw (338.27,171.23) node [anchor=north west][inner sep=0.75pt]    {$\frac{\alpha y}{2}$};
\draw (405.6,102) node [anchor=north west][inner sep=0.75pt]    {$R$};

\end{tikzpicture}
\caption{Intersection of two balls. The base of the dotted line is at a distance of $R\rho$ from the origin. Cutting the intersection along the dotted line gives a ball in one dimension less and has radius $R\sqrt{1-\rho^2}$. We integrate on the parameter $\rho$.}
\label{fig:geometry}
\end{figure}

\end{proof}

\begin{lemma}
\label{le:potato_slice}
For any $z \in K_\mathbb{R}^{t}$ and with the same setting as Lemma \ref{le:lemma5_rogers}, we have 
\begin{equation}
 \int_{K_\mathbb{R}^{ t \times 2}} f(x) f(y) f(\alpha_1 x + \alpha_2 y + z ) dxdy  \le 
 \int_{K_\mathbb{R}^{ t \times 2 }} f(x) f(y) f(\alpha_1 x + \alpha_2 y ) dxdy  .
\end{equation}
\end{lemma}
\begin{proof}
Let $z = (z_1,z_2,\dots,z_t) \in K_\mathbb{R}^{t}$ and let $z' = (z_1',z_2,\dots,z_t)$, where $z_1' \in K_\mathbb{R}\simeq \mathbb{R}^{\oplus r_1} \oplus \mathbb{C}^{\oplus r_2}$ is equal to $z_1$ at all embeddings except one embedding $\sigma: K \rightarrow \mathbb{R}$ or $\sigma: K \rightarrow \mathbb{C}$ where it is equal to $0$.

Thus we can write $z = \pi(z) e_{1, \sigma } + z'$ where $\pi: K_\mathbb{R}^{t} \rightarrow \mathbb{R}$ is the $\mathbb{R}$-coordinate of $z_1$ along $\sigma$ and $e_{1,\sigma}$ is an appropriate vector. 

For the statement, we will prove the following inequality and the rest will follow suit:
\begin{equation}
 \int_{K_\mathbb{R}^{t \times 2 }} f(x) f(y) f(\alpha_1 x + \alpha_2 y + z ) dxdy  \le 
 \int_{K_\mathbb{R}^{t \times 2 }} f(x) f(y) f(\alpha_1 x + \alpha_2 y + z'  ) dxdy  .
 \label{eq:general_claim}
\end{equation}

For $x,y$ we analogously define $x' = x- \pi(x) e_{1,\sigma}$ and $y' = y - \pi(y) e_{1,\sigma}$. Then the claim above will follow from the claim 
that for any $x,y \in K_\mathbb{R}^{t}$
\begin{align}
 & \int_{\mathbb{R}^{2}} f(x' + s e_{1,\sigma}) f(y' + t e_{1,\sigma}) f\left(\alpha_1 (x'+s e_{1,\sigma}) + \alpha_2 (y' + t e_{1,\sigma}) + z \right) ds  dt \\
	\le  &   \int_{\mathbb{R}^{2}} f(x' + s e_{1,\sigma}) f(y' + t e_{1,\sigma}) f\left(\alpha_1 (x'+s e_{1,\sigma}) + \alpha_2 (y' + t e_{1,\sigma}) + z' \right) ds  dt .
	\label{eq:2d_claim}
\end{align}
Indeed, we can obtain (\ref{eq:general_claim}) from (\ref{eq:2d_claim}) by integrating along $x',y'$.

To prove the last inequality, observe that if $B \subseteq K_\mathbb{R}^{t}$ is the ball whose indicator function is $f$ and if $P\subseteq K_\mathbb{R}^{t \times 2}$ is $2$-dimensional plane spanned by $(e_{1,\sigma},0)$ and $(0,e_{1,\sigma})$ then 
\begin{equation}
	(B \times B) \cap \left( (x',y') + P \right) \subseteq K_\mathbb{R}^{t \times 2}
\end{equation}
is the area within a square centered at the point $x',y'$ since 
\begin{equation}
\|x' + s e_{1,\sigma}\|^{2} =  \|x'\|^{2} + s^{2} \text{ and }
  \|y' + t e_{1,\sigma}\|^{2} =  \|y'\|^{2} + t^{2}.
\end{equation}
Furthermore, if $E_z \subseteq K_\mathbb{R}^{t \times 2}$ is the set whose indicator function is $(x,y) \mapsto f(\alpha_1 x + \alpha_2 y + z)$, then
\begin{equation}
	E_{z} \cap ( (x',y') + P)  \text{ and }
	E_{z'} \cap ( (x',y') + P)  
\end{equation}
are both $2$-dimensional areas between two parallel lines and one is a translate of the other. Since the latter area is symmetrical around $(x',y')$ and the former may not be, we can conclude geometrically 
\begin{equation}
\vol( ( B \times B ) \cap	E_{z} \cap ( (x',y') + P)   )
\le 
\vol( ( B \times B ) \cap	E_{z'} \cap ( (x',y') + P)   )
\end{equation}
This shows 
that (\ref{eq:2d_claim}) must hold.

\end{proof}

\begin{proof}[\bf Proof of Theorem \ref{th:poisson_dist}]
Set $N=t[K:\mathbb{Q}]$ as before. By Lemma \ref{le:counting_a1m} it is enough to consider the contribution of any $D \in A^{1}_m$.

For such matrices $D$, we claim that 
\begin{equation}
V(N)^{-m} \frac{1}{\mathfrak{D}(D)^{t}} \int_{x \in M_{t \times m}(K_\mathbb{R})} g(xD)dx \le 2( \tfrac{\sqrt{3}}{2} )^{N}.
\end{equation}
Indeed, recall that $D$ is an $m \times n$ matrix with entries in $\mu_K \cup \{0\}$ such that it has at least one column with more than one entry. 
Hence, without loss of generality we can assume that $D$ looks like 
\begin{equation}
 \label{eq:matrix_in_a1m}
  \begin{bmatrix}
	  1     &     &   &    &              & \mu_{1} & \dots & *\\
		&   1   &    &   &            & \mu_{2} & \dots & *\\
		& &        1 &   &              &  *   & \dots & * \\
		&    &    &\ddots&	        & \vdots &   &  \vdots\\
	            &    &    &    &      1       & * & \dots & * \\
  \end{bmatrix}.
\end{equation}
So if $f$ is the indicator function of the ball as in the statement of Lemma \ref{le:lemma5_rogers}, then we can write that for $x \in K_\mathbb{R}^{t \times n}$
\begin{align}
	f(x D )  & = f(x_1)f(x_2)\dots f(x_m) f( \mu_1x_1 + \mu_2 x_2 + \cdots) \cdots \\
		 & \le f(x_1)\dots f(x_m) f(\mu_1 x+ \mu_2 x + \cdots )
\end{align}
Then, we can invoke Lemma \ref{le:potato_slice} to get that
\begin{align}
  \int_{x \in K_\mathbb{R}^{t \times m}} f(xD) 
		 & \le  \int_{x \in K_\mathbb{R}^{t \times m }} f(x_1)\dots f(x_m) f(\mu_1 x_{1}+ \mu_2 x_{2} + \cdots ) dx\\
		 & \le  \int_{x \in K_\mathbb{R}^{t \times m }} f(x_1)\dots f(x_m) f(\mu_1 x_{1}+ \mu_2 x_{2} ) dx
\end{align}
The claim is therefore a consequence of Lemma \ref{le:lemma5_rogers}. The contribution from these terms decays exponentially as a result. The statement of the theorem then follows, given that $\card A^{1}_{m}$ grows at most polynomially in the degree $d=[K:\mathbb{Q}]$, whereas the term $(\tfrac{\sqrt{3}}{2})^{N}$ decays exponentially with $d$.

\end{proof}

\section{Upper bounds on moments using Weil heights}
\label{se:Bogosection}

We now turn to bounding the remaining terms in order to establish explicit formulas for moments of $\mathcal{O}_K$-lattices. Our aim is to establish such formulas while allowing the degree $d=[K:\mathbb{Q}]$ to vary as well. Our results will involve lower bounds on the Weil height for algebraic numbers. Such bounds are closely related to Lehmer's problem-which is still open-however suitable bounds for our purposes are known in many interesting and important cases. 
\subsection{Mahler measures and the Bogomolov property}

For an algebraic number $\alpha \in K^\times$, recall that the Mahler measure (or unnormalised exponential Weil height) is given by the product over the set of places $M_K$ of $K$:
\begin{equation}
H_W(\alpha) := \prod_{v\in M_K} \max\{1, |\alpha |_v\}.
\end{equation}
We also define, keeping only the infinite places, the closely related 
$$H_\infty(\alpha) = \prod_{\sigma: K \rightarrow \mathbb{C}} \max\{1, |\sigma(\alpha) | \}$$
which will be more directly relevant for estimates in the Euclidean space associated to $K$. The two coincide for algebraic integers and in general differ by a denominator. 
We also recall that the \textbf{absolute} Mahler measure (or exponential Weil height) of an algebraic number $\alpha$ is given by $H_W(\alpha)^{1/\deg(\alpha)}$, where $\deg(\alpha)=[\mathbb{Q}(\alpha):\mathbb{Q}]$ and the underlying product is taken over the places of $\mathbb{Q}(\alpha)$. We shall denote by
$$h(\alpha)=\log(H_W(\alpha)^{1/\deg(\alpha)})$$
the \textbf{Weil height} of an algebraic number. For non-integers, we shall also write $h_\infty(\alpha)$ for $\log(H_\infty(\alpha)^{1/\deg(\alpha)})$.

\begin{remark}
Note that the absolute Mahler measure and Weil heights are independent of the particular subfield of an algebraic closure over which one is considering an algebraic integer. That is, if $\beta \in K$ we have  $\deg \beta =\card \{\sigma: \mathbb{Q}(\beta) \rightarrow \mathbb{C}\}$ and
\begin{equation}
\frac{\log \left( \prod_{\sigma : K \rightarrow \mathbb{C}} \max\{1,|\sigma(\beta)|\} \right)}{[K: \mathbb{Q}]} 
= 
\frac{\log \left( \prod_{\sigma : \mathbb{Q}(\beta) \rightarrow \mathbb{C}} \max\{1,|\sigma(\beta)|\} \right)}{[\mathbb{Q}(\beta): \mathbb{Q}]}.
\end{equation}
Therefore, for any number field $K$, the Weil height of $\alpha\in K^\times$ may also be computed as
$$h(\alpha)=\tfrac{1}{[K: \mathbb{Q}]}\cdot\sum_{v\in M_K}\log( \max\{1, |\alpha |_v\}).$$

\end{remark}

Lehmer's famous problem asks for a uniform lower bound for $h(\alpha)\deg(\alpha)$. We shall consider algebraic numbers related to the stronger property: 
\begin{definition}
A subset $S\subset \overline{\mathbb{Q}}$ is said to satisfy the \textbf{Bogomolov property} if there exists a constant $C>0$ such that 
$$h(\alpha)\geq C$$
provided $\alpha\in S$ has infinite multiplicative order. 
\end{definition}

Throughout this section, we will therefore consider $\mathcal{O}_K$-lattices for towers of number fields inside a subset of $\overline{\mathbb{Q}}$ satisfying the Bogomolov property. In other words, we formulate the assumption:

\begin{hypothesis}\label{hyp:Lehmer}
    As $K$ varies among the number fields considered, there exist uniform constants $c_0\geq c_1>0$ such that the absolute (logarithmic) Weil heights satisfy 
    $$h(\alpha)>c_1\text{ for }\alpha\in K^\times\setminus \mu_K$$ 
    and $$h_\infty(\alpha)=h(\alpha)>c_0\text{ for }\alpha\in \mathcal{O}_K\setminus \{\mu_K, 0\},$$
    where $\mu_K$ denotes the group of roots of unity contained in $K$.
\end{hypothesis}

We now recall some important examples from the literature when the Bogomolov property is satisfied. The first result is a bound due to Schinzel \cite{Schinzel1973}: 
\begin{theorem}\label{thm:cycloheightbound}
    Assume that an algebraic number $\alpha$ of infinite multiplicative order is contained in a totally real field. Then, denoting by $\varphi=\frac{1+\sqrt{5}}{2}$ the golden ratio, we have
     $$h(\alpha)\geq \frac{1}{2}\log{\varphi}\approx 0.2406\ldots.$$
     Moreover, the same is true for $\alpha$ in a CM field provided one (and equivalently, all) of its Archimedean embeddings satisfy $\vert \alpha\vert\neq 1$. 
\end{theorem}
We therefore get that Theorem \ref{thm:cycloheightbound} also applies to algebraic integers in CM fields, however there exist algebraic numbers which are not roots of unity but all of whose conjugates lie on the unit circle-in fact the bound is violated for such numbers. 
We do, however, have for abelian extensions the bound due to Amoroso--Dvornicich \cite{AMOROSO2000260}:
\begin{theorem}\label{thm:cycloheightboundnumbers}
    Assume that an algebraic number $\alpha$ of infinite multiplicative order is contained in an abelian extension of $\mathbb{Q}$. Then we have
     $$h(\alpha)\geq \frac{\log{5}}{12}\approx 0.1341\ldots.$$
\end{theorem}
We may therefore record as a special case:
\begin{corollary}\label{cor:cyclohypothesis}
    Any tower of cyclotomic fields satisfies Hypothesis \ref{hyp:Lehmer} with constants $c_0=\frac{1}{2}\log{\varphi}\approx 0.2406$ and $c_1=\frac{\log{5}}{12}\approx 0.1341$. 
\end{corollary}
Even in the special case of cyclotomic fields these bounds are reasonably sharp, for instance in the field $\mathbb{Q}(\zeta_{21})$ there is an algebraic number of height $\log(7)/12$. Concerning Schinzel's result, we already have exceptions in the following range (see \cite[Theorem 5.39]{MS21}): 
\begin{theorem}
Suppose that $\beta$ is a cyclotomic integer. Then the only values for $h_\infty(\beta)$ inside the interval $(0,0.27132]$ occur for $\beta =2\cos(2\pi/5)$, $2\cos(2\pi/7)$, $2 \cos(2\pi/60)$.
\label{th:lower_bound_mahler}
\end{theorem}

Beyond these results, the Bogomolov property is itself well-studied and we list a number of subsets of $\overline{\mathbb{Q}}$ satisfying it and leading to towers of number fields verifying Hypothesis \ref{hyp:Lehmer}. We refer the reader to \cite[Chapter 11]{MS21} and \cite{AmorosoDavidZannierB} for more details. 
\begin {itemize}
\item Generalizing the totally real case, Langevin \cite{Langevin} showed that the property holds for closed subsets of $\mathbb{C}$ which do not contain the unit circle. 
\item Totally $p$-adic numbers or (infinite) Galois extensions with bounded local degree at some rational prime $p$ satisfy Bogomolov's property (see \cite{Pottmeyer} and \cite[Theorem 2]{BombieriZannier01}).
\item Generalizing the abelian case, Habegger \cite{HabeggerEtorsion} shows that fields obtained adjoining torsion points of elliptic curves over $\mathbb{Q}$ have the Bogomolov property. Amoroso--David--Zannier show \cite[Theorem 1.5.]{AmorosoDavidZannierB} among others that infinite Galois extensions of a fixed number field with Galois group $G$ have the Bogomolov property provided that $G$ has finite exponent modulo center. 
\end{itemize}
We end our discussion with some height bounds that work for every number field, in particular we state E. Dobrowolski's asymptotic result \cite[Theorem 1]{D1979}: 
\begin{theorem}
    Let $\alpha$ be an algebraic integer of degree $d$, not zero or a root of unity, and let $\varepsilon>0$. Then for $d\geq d(\varepsilon)$ we have that 
    $$h(\alpha)\geq \frac{1-\varepsilon}{d}\cdot \left(\frac{\log\log d}{\log  d}\right)^3.$$
    Moreover, P.Voutier \cite{Voutier96} showed that for any $d\geq 2$ we may take
    $$h(\alpha)\geq \frac{1}{4d}\cdot \left(\frac{\log\log d}{\log  d}\right)^3.$$
\end{theorem}
We therefore record the obvious but important remark: 

\begin{remark}
    Any fixed number field $K$ satisfies Hypothesis \ref{hyp:Lehmer} for suitable constants. 
\end{remark}
In particular, this will imply our limiting moment formulas established in this section are valid for any fixed number field and large enough rank.

\subsection{Bounds for contributions from projective space}
\label{subsec:projBogo}

Throughout this section, $M\geq 1$ is a fixed integer and we write $\alpha=(\alpha_1,\ldots,\alpha_M)\in K^M\setminus\{0\}$ as well as the height: 
$$H_\infty(\alpha)=\prod_{\sigma:K\to\mathbb{C}}\max_{1\leq j\leq M}\max(1,\vert\sigma(\alpha_j)\vert)$$
which specializes to the (exponential) Weil height when $M=1$ and $\alpha\in\OK^M$. We also write 
$$h_\infty(\alpha)= \tfrac{1}{d}\log H_\infty(\alpha)$$
and denote the norm \footnote{We slightly abuse notations by decreeing our norms of algebraic numbers are positive, ergo the norms of the ideal they generate.} of the denominator ideal generated by $\alpha$ by 
 \begin{equation}\label{def:Dalpha}
     D(\alpha):=\N(\calO_K+\alpha_1\calO_K+\cdots+\alpha_M\calO_K)^{-1}.
 \end{equation}
 Observe that the inequalities:
 $$D(\alpha)^{-1}\leq \N(\alpha_1\cdots\alpha_M)^{\frac{1}{M}}\leq H_\infty(\alpha)$$
 follow from the definitions when the $\alpha_i\neq 0$ and that we have the relation 
 \begin{equation}
	 D(\alpha ) \cdot\N(\langle 1,\alpha_1,\ldots, \alpha_M \rangle) = 1
 \end{equation}
 with the norm defined under \eqref{def:langlenorm}. 
 
 Our main goal in this subsection is to examine for $t>M\geq 1$ the sum 
 $$S_{M,t}:=\sum_{\alpha\in (K^\times)^M} D(\alpha)^{-t}\vol(B \cap \alpha_1^{-1} B\cap \dots \cap \alpha_{M}^{-1}B).$$
 This will yield upper bounds on the $A^2_m$-terms when $m=1$ or can be viewed as bounding height zeta functions for projective spaces instead of the full Grassmannian variety $\Gr(m,K^n)$. 
\begin{lemma}\label{lemma:assumenorm}
    The quantity $\N(\langle \alpha_0,\ldots,\alpha_M \rangle)^{t}\vol(\alpha_0^{-1} B \cap \alpha_1^{-1} B\cap \dots \cap \alpha_{M}^{-1}B)$ only depends on the class $[\alpha_0:\cdots:\alpha_M]$ in projective space $\mathbf{P}^M(K)$ modulo permutation of coordinates.
\end{lemma}
\begin{proof}
  Multiplying by a scalar $\lambda\in K^\times$ scales the volume by $\N(\lambda)^{-t}$ whereas the index is scaled by $\N(\lambda)$. 
\end{proof}
In particular, scaling by $\alpha_i$ of maximal norm this implies that we may restrict our computations for $S_{M,t}$ to the case where $N(\alpha_i^{-1})\geq 1\forall i$. We shall use the following convex combination lemma to bound volumes of intersections of scaled balls:

\begin{lemma}
\label{le:intersection_of_ellipsoids}
Let $M \ge 1$ and suppose $\alpha_0,\alpha_1,\dots,\alpha_{M} \in K^{*}$. Let $B$ be an origin-centered ball of radius $R$ in the space $K_\mathbb{R}^{t}$ with respect to the norm in Equation (\ref{eq:norm}). Given that $K^{*}$ acts on $K_\mathbb{R}^{t}$ diagonally, we have
\begin{equation}
\vol(\alpha_0 B \cap \alpha_1 B\cap \dots \cap \alpha_{M}B) \le  \vol(B) \cdot \min_{\substack{c_i \ge 0 \\ \sum_{i} c_i = 1}} \left\{ \prod_{\sigma : K \rightarrow \mathbb{C}}\left( { \sum_{i=0}^{M}{c_i| \sigma(\alpha_i )|^{2}} }{} \right)^{-\frac{t}{2}}\right\} ,
\end{equation}
where the minimum is over any real convex combination of the $\alpha_i$.
\end{lemma}

\begin{proof}
We are calculating the volume of the intersections of the following ellipsoids (see Equation (\ref{eq:norm}):
\begin{equation}
\left\{ x \in K_\mathbb{R}^{t} \mid  \Delta_K^{-\frac{2}{[K:\mathbb{Q}]}} \sum_{j=1}^{t}\Tr\left( \alpha_i x_j \overline{ \alpha_i x_j}  \right) \le R^{2}\right\}, \text{ as } i \in \{0,1,\dots,M\}
\end{equation}
Observe that for any $\{c_0,\dots,c_{M}\} \in \mathbb{R}_{\ge 0}$ such that $\sum_{i} c_i = 1$, we have that 
\begin{equation}
\bigcap_{i =0}^{M}
\left\{ x \in K_\mathbb{R}^{t} \mid  \Delta_K^{-\frac{2}{[K:\mathbb{Q}]}} \sum_{j=1}^{t}\Tr\left( \alpha_i x_j \overline{ \alpha_i x_j}  \right) \le R^{2}\right\}
\subseteq 
\left\{ x \in K_\mathbb{R}^{t} \mid  \Delta_K^{-\frac{2}{[K:\mathbb{Q}]}} \sum_{j=1}^{t}\Tr\left( \left( \sum_{i=0}^{M}c_i \alpha_i \overline{ \alpha_i } \right) x_j \overline{  x_j}  \right) \le R^{2}\right\}
\end{equation}

The ellipsoid defined on the right side has volume given by 
$$\vol(B) \prod_{\sigma: K \rightarrow \mathbb{C}}\left( \sum_{i} c_i | \sigma ( \alpha_i) |^{2} \right)^{-t/2}.$$
\end{proof}
We are now in a position to connect the intersection volumes to Weil heights:

\begin{lemma}\label{lemma:genvolumeratio}
Let $\alpha_1,\ldots, \alpha_M\in K^\times$. We have the bound:
$$
\frac{\vol(B \cap \alpha_1 B\cap \dots \cap \alpha_{M}B)}{\vol(B)}\le\left(\frac{H_\infty(\alpha)^{\frac{2}{d}}+M\cdot 
H_\infty(\alpha)^{-\frac{2}{dM}}\cdot \N(\alpha_1\cdots\alpha_M)^{\frac{2}{dM}}}{M+1}\right)^{-dt/2}.$$
Moreover, under the assumption that $\N(\alpha_1\cdots\alpha_M)\geq 1$ we have for any $k\geq 2$:
$$ \frac{\vol(B \cap \alpha_1 B\cap \dots \cap \alpha_{M}B)}{\vol(B)}\le\N(\alpha_1\cdots\alpha_M)^{\frac{-t}{kM}}\cdot \left(\frac{H_\infty(\alpha)^\frac{2(k-1)}{kd}+M\cdot H_\infty(\alpha)^\frac{-2(k-1)}{kMd}}{M+1}\right)^{-dt/2}$$

\end{lemma}
\begin{proof}
    Lemma \ref{le:intersection_of_ellipsoids} together with Lemma \ref{le:lower_heigh_bound} comparing heights yields the inequality
$$
 \frac{\vol(B \cap \alpha_1 B\cap \dots \cap \alpha_{M}B)}{\vol(B)} \leq \left(\frac{H_\infty(\alpha)^{\frac{2}{d}}+
 M\cdot H_\infty(\alpha)^{-\frac{2}{dM}}\cdot \N(\alpha_1\cdots\alpha_M)^{\frac{2}{dM}}}{M+1}\right)^{-\frac{1}{2}dt},
$$
 where we also applied the bound in Lemma \ref{le:about_Mx}. To obtain the second formulation we now factor out $\N(\alpha_1\cdots\alpha_M)^{\frac{2}{kdM}}$ and  writing  $$g_M(x)=\tfrac{x+Mx^{-\frac{1}{M}}}{M+1}$$ we have the bound 
     \begin{align}
     \N(\alpha_1\cdots\alpha_M)^{\frac{t}{kM}}\cdot\left(\frac{H_\infty(\alpha)^{\frac{2}{d}}+M\cdot H_\infty(\alpha)^{-\frac{2}{dM}}\cdot \N(\alpha_1\cdots\alpha_M)^{\frac{2}{d(M+1)}}}{M+1}\right)^{-dt/2}
     \\
     \leq 
     g_M\left(\frac{H_\infty(\alpha)^{\frac{2}{d}}}{\N(\alpha_1\cdots\alpha_M)^{\frac{2}{kdM}}}\right)^{-dt/2},
     \end{align}
    using that $\N(\alpha_1\cdots\alpha_M)\geq 1$. Now observe that $g_M(x)$ is increasing for $x,M\geq 1$ so that using the inequality $\N(\alpha_1\cdots\alpha_M)^{1/M}\leq H_\infty(\alpha)$ we can bound 
    $$g_M\left(\frac{H_\infty(\alpha)^{\frac{2}{d}}}{\N(\alpha_1\cdots\alpha_M)^{\frac{2}{kdM}}}\right)\geq g_M(H_\infty(\alpha)^\frac{2(k-1)}{kd}).$$
    The claim follows.

\end{proof}
\begin{remark}
    The role of $k$ in Lemma \ref{lemma:genvolumeratio} and ensuing results is slightly artificial, but it allows us in later results to take $k$ large enough so that we can control the sum of volume ratios over units for small $t$ while the additional factor $\N(\alpha)^{\frac{-t}{kM}}$ allows us to relate the sum $S_{M,t}$ to a Dedekind zeta value. This leads to slightly better results for small moments.  
\end{remark}
The following lemmas provide upper bounds for point counts in the unit lattice: 
\begin{lemma}\label{lemma:unitcountonecoord}
    Assume Hypothesis \ref{hyp:Lehmer} and its notations. Consider the canonical log embedding: 
$L:K^\times\to \mathbb{R}^{r_1+r_2}$ defined by mapping
$$\alpha\mapsto (\log \vert\sigma_1(\alpha)\vert,\ldots,2\log \vert\sigma_{r_1+r_2}(\alpha)\vert),$$
as well as the function
\begin{align}
    h: \mathbb{R}^{r_1+r_2}&\to \mathbb{R}_{\geq 0} \\
    x&\mapsto \tfrac{1}{[K:\mathbb{Q}]}\cdot\sum_{j=1}^{r_1+r_2}\max(0,x_{j}).
\end{align}
    Then for any $\eta\in \mathbb{R}^{r_1+r_2}$ with $\sum_{j=1}^{r_1+r_2}\eta_j=Y$ and any $B\geq 0$ we have that
    $$\card \{\beta\in\mathcal{O}_K^\times~\vert~ h(\eta+L(\beta))\leq B\}\leq \omega_K\cdot \left(\frac{B+c_0/2+\max(0,\tfrac{-Y}{d})}{c_0/2}\right)^{r_1+r_2-1}.$$
\end{lemma}
\begin{proof}
Note that the factor of $2$ at complex places in the definition of $L$ ensures that $L(\mathcal{O}_K^\times)$ is contained in the hyperplane $H:=\{x\in \mathbb{R}^{r_1+r_2}: \sum_{j=1}^{r_1+r_2}x_{j}=0\}$. 
Observe that $h$ satisfies the triangle inequality and in fact satisfies the properties of a semi-norm on $H$. 
Now by Hypothesis \ref{hyp:Lehmer} we obtain for any $\beta\in \mathcal{O}_K^\times\setminus \mu_K$ that 
$$h(L(\beta))=h_\infty(\beta)\geq c_0.$$
Let now $P=\{\xi\in H: h(\xi)\leq c_0/2\}$. We claim that for $\eta\in \mathbb{R}^{r_1+r_2}$ and $\beta_1,\beta_2\in \mathcal{O}_K^\times$: 
$$(L(\beta_1)+\eta+P)\cap (L(\beta_2)+\eta+P)=\begin{cases}
    L(\beta_1)+\eta+P& \text{ if }\beta_1\beta_2^{-1}\in \mu_K \\
    \emptyset & \text{ else.}
\end{cases}$$

To prove the claim, let $y$ be in the intersection. Then by the triangle inequality we have that
$$h(L(\beta_1^{-1}\beta_2))\leq h(y-L(\beta_1)-\eta)+h(L(\beta_2)+\eta-y)\leq c_0$$
and therefore $\beta_1\beta_2^{-1}\in \mu_K$. Since $L$ is a homomorphism to the additive group whose kernel is $\mu_K$ the claim follows.

Moreover, if for $\beta\in \mathcal{O}_K^\times$ we have that $h(\eta+L(\beta))\leq B$, then $L(\beta)+\eta+P$ is contained in the set 
$$Q=\{\xi \in \mathbb{R}^{r_1+r_2}:  \sum_{j=1}^{r_1+r_2}\xi_{j}=Y, h(\xi)\leq B+c_0/2 \}.$$
For any fixed $\eta$, we thus obtain by the claim that 
$$\card \left\{\beta\in\mathcal{O}_K^\times\vert h(\eta+L(\beta))\leq B\right\}\leq \omega_K\cdot\frac{\vol(Q)}{\vol(P)},$$
where the volumes are computed with respect to the natural measure identifying the hyperspaces $P$ and $Q$ are in with $\mathbb{R}^{r_1+r_2-1}$. \par 
For $\eta$ such that $Y=\sum_{j=1}^{r_1+r_2}\eta_j\geq 0 $, it is easy to see that the volume of $Q$ is bounded by the volume of $P_B= \{\xi\in H: h(\xi)\leq c_0/2+B\}$. 
Thus we bound the desired unit count by 
$$\omega_K\cdot\frac{\vol(P_B)}{\vol(P)}=\omega_K^M\cdot\frac{\vol(\{\xi\in H: h(\xi)\leq c_0/2+B\})}{\vol(\{\xi\in H: h(\xi)\leq c_0/2\})}.$$
Since $H$ is an $\mathbb{R}$-vector space of dimension $r_1+r_2-1$ and $h$ a semi-norm on that vector space, the result follows for $Y\geq 0$. \par
When $Y<0$, observe that $\tilde{\eta}$ defined by $\tilde{\eta}_j=\eta_j-\tfrac{Y}{r_1+r_2}$ satisfies 
$$\sum_{j=1}^{r_1+r_2}\tilde{\eta}_j=0 \text{ and }h(\tilde{\eta})\leq h(\eta)+h(-\tfrac{Y}{r_1+r_2}).$$
Therefore, given that $h(\eta+L(\beta))\leq B$ implies $h((\tilde{\eta})+L(\beta))\leq B+h(-\tfrac{Y}{r_1+r_2})$,
we may obtain an upper bound by running the same argument as in the first part of the proof with $\tilde{\eta}$ instead of $\eta$ and $B+h(-\tfrac{Y}{r_1+r_2})$ instead of $B$. This settles the case of $Y<0$ since $h(-\tfrac{Y}{r_1+r_2})=-Y/d$. 
\end{proof}
Note that we also have the following inequality by definition: 
\begin{lemma}\label{le:heightineqcoord}
    Let $\alpha\in (\overline{K}^\times)^M$ be an $M$-tuple of algebraic numbers. Then $h_\infty(\alpha)\geq \max_{1\leq i\leq M}h_\infty(\alpha_i)$. 
\end{lemma}

\begin{lemma}\label{lemma:unitcount}
    Assume Hypothesis \ref{hyp:Lehmer} and its notations. 
    Let $\alpha\in ({K}^\times)^M$ and $B\geq 0$. Then 
    $$\card \{\beta\in(\mathcal{O}_K^\times)^M~\vert~ h_\infty(\alpha\beta)\leq B\}\leq \omega_K^M\cdot \prod_{1\leq i\leq M}\left(\frac{B+c_0/2+\max(0, \tfrac{-\log \N(\alpha_i)}{d})}{c_0/2}\right)^{(r_1+r_2-1)}.$$
\end{lemma}
\begin{proof}
    By Lemma \ref{le:heightineqcoord}, we have that 
     $$\card \{\beta\in(\mathcal{O}_K^\times)^M~\vert~ h_\infty(\alpha\beta)\leq B\}\leq\prod_{1\leq i\leq M}\card \{\beta\in\mathcal{O}_K^\times~\vert~ h_\infty(\alpha_i\beta)\leq B\}$$
    We conclude by Lemma \ref{lemma:unitcountonecoord}. 
\end{proof}

\begin{proposition}\label{prop:mintersections}
    Assume Hypothesis \ref{hyp:Lehmer} and its notations and fix $k\geq 2$. There exist positive constants $C,\varepsilon_1>0$ uniformly bounded in $d,t$ such that for all $\alpha\in (K^\times)^M\setminus \mu_K^M$ with $N(\alpha_i)\geq 1$ for $i\in [1,\ldots, M]$ the following holds:  write
     $$t_0=\frac{2r_K\cdot M}{d}\cdot\frac{ \log(2+\tfrac{1}{2k})}{\log(f_M(c_0(1-\frac{1}{k}))} ,$$
     where $f_M(x):=\frac{\exp(x)+M\exp(-\frac{x}{M})}{M+1}$ and $r_K$ is the rank of the unit group.
Then we have for any $t> t_0$ and any $d\geq 1$ that
\begin{align}
 \sum_{\substack{\beta\in(\mathcal{O}_K^\times)^M\\ \alpha\beta\notin \mu_K^M}}\frac{\vol(B\cap(\alpha_1\beta_1)^{-1}B\cap\cdots \cap(\alpha_M\beta_M)^{-1}B)}{\vol(B)} 
 \leq 
 C\cdot\omega_K^M\cdot \N(\alpha)^\frac{-t}{kM}\cdot D(\alpha)^{\frac{t}{4}}\cdot  e^{-\varepsilon_1\cdot d\cdot (t-t_0)}.
\end{align}
Moreover, the constants can be made explicit. We may for instance take 
$$\varepsilon_1=\tfrac{1}{2}\min\left\{\tfrac{c_1}{8}, \log(f_M(\tfrac{3}{4}c_1)), \alpha_M\cdot \tfrac{c_0(k-1)}{k}))\right\}$$ and $C=1+\tfrac{1}{1-e^{-\alpha_M\cdot c_0\cdot d(t-t_0)(k-1)/(4k^2)}}$, where $\alpha_M>0$ is small enough so that $f_M(x)\geq e^{\alpha_M\cdot x}$ for $x\geq c_0/2$. 
 
\end{proposition}
\begin{proof}
    We consider the function $f_M(x):=\frac{\exp(x)+M\exp(-x/M)}{M+1}$
    satisfying $f_1(x)=\cosh(x)$, $f_M(0)=1$ and increasing exponentially for $x>0$. We also abbreviate $\N(\alpha)=\prod_i \N(\alpha_i)$. Then by Lemma \ref{lemma:genvolumeratio} our task is reduced to bounding for suitably chosen $k\geq 2$: 
    \begin{equation}\label{eq:projtobound}
        \sum_{\substack{\beta\in(\mathcal{O}_K^\times)^M \\ \alpha\beta \notin (\mu_K)^{M}}}\N(\alpha)^\frac{-t}{kM}\cdot f_M\left( h_\infty(\alpha\beta)\left( 2(1-\tfrac{1}{k})\right)\right)^{-dt/2},
    \end{equation}
    where $h_\infty(\alpha\beta)=\tfrac{1}{d}\cdot \log(H_\infty(\alpha\beta))$ reduces to the log Weil height for $M=1$. 
    Recall the $c_0$ defined in Hypothesis \ref{hyp:Lehmer}. Since $f_M$ is increasing it suffices to bound, for any $S \in \mathbb{Z}_{>0}$, 
    the sum
$$\Sigma_{M,k}^\infty:=
\sum_{n=S}^{\infty}\card\left\{\beta\in(\mathcal{O}_K^\times)^M\mid h_\infty(\alpha\beta)\in \Big[\tfrac{nc_0}{2S}, \tfrac{(n+1)c_0}{2S}\Big[~\right\} \cdot 
f_M\left(\frac{nc_0(1-\tfrac{1}{k})}{S}\right)^{-dt/2}$$
    together with the term with the contribution of the remaining units satisfying $h_\infty(\alpha\beta)< c_0/2$:
   
	    \begin{equation}
	      \Sigma_{M,k}^{c_0} := 
\sum_{\substack{  \beta\in(\mathcal{O}_K^\times)^M \\ h_\infty(\alpha\beta)< \tfrac{c_0}{2}~ }} 
f_M\left(  h_{\infty}(\alpha \beta)  \cdot  2 (1-\tfrac{1}{k})\right)^{-dt/2}.
	    \end{equation}
So our goal is to show for appropriate constants $C,\varepsilon_1$ that 
\begin{equation}
  \Sigma_{M,k}^{c_0} + \Sigma_{M,k}^{\infty} \leq C \cdot D(\alpha)^{\frac{t}{4}} \cdot e^{-\varepsilon_1 \cdot d \cdot (t-t_0)}.
\end{equation}

	    Let us examine the term $\Sigma_{M,k}^{c_0}$ first.
   We have the bound on the number of units 
    $$\card \{\beta\in(\mathcal{O}_K^\times)^M\vert h_\infty(\alpha\beta)< c_0/2\}\leq \omega_K^M.$$
	    Indeed, let $\beta_1 = (\beta_{11},\dots,\beta_{1M})$ and $\beta_2 = (\beta_{21},\dots,\beta_{2M})$ be in $( \mathcal{O}_K^{\times} )^{M}$. Then, for we know that for each $j=1,2$ and each $i=1,\dots,M$ we have $h_{\infty}(\alpha_{i}\beta_{ji}) < \tfrac{c_0}{2}$. Then by the triangle inequality (c.f. Lemma \ref{lemma:unitcountonecoord}), we can show that $h_{\infty}(\beta_{1j}^{-1}\beta_{2j}) < c_0 $ so that $ \beta_{1} \beta_2^{-1} \in (\mu_K)^{M}$.\par
    By assumption, since $\alpha\beta\notin ( \mu_K )^M$ there exists a constant $c_1$ such that 
    $$h(\alpha\beta)=h_\infty(\alpha\beta)+ \tfrac{1}{d}\cdot \log D(\alpha)\geq c_1>0.$$
    
	    Let us first assume that $D(\alpha)<\exp(dc_1/4)$. Then
$\Sigma_{M,k}^{c_0}$ is bounded by 
    
	    $$\omega_K^M \cdot f_M\left(  2(1-\tfrac{1}{k}) \cdot ( c_1-\tfrac{1}{d}\cdot \log D(\alpha) )\right)^{-dt/2}
	    \leq 
    \omega_K^M \cdot f_M\left( 2 \left(1-\tfrac{1}{k}\right)\tfrac{3}{4}c_1\right)^{-dt/2}
	    \leq 
    \omega_K^M \cdot f_M\left(  \tfrac{3}{4}c_1\right)^{-dt/2}.$$

    In the case when $D(\alpha)\geq \exp(dc_1/4)$, 
    we simply bound the contribution
    $\Sigma_{M,k}^{c_0}$ by observing that 
    $D(\alpha)^{-\frac{t}{4}}\leq e^{-\tfrac{dtc_1}{16}}$, so that these terms satisfy the bound claimed in the proposition. 
    
    \par 
    We may now therefore turn to the remaining terms $\Sigma_{M,k}^{\infty}$. 
    Since, we asummed that $\N(\alpha_i) \geq 0$, we know that $\max(0,-\tfrac{1}{d} \log \N(\alpha_i)) = 0$.
    Lemma \ref{lemma:unitcount} therefore yields:    
	    \begin{align}
		    \Sigma_{M,k}^\infty
		    &\leq \omega_K^M\cdot \sum_{n=S}^{\infty}\left(\tfrac{n+S+1}{S}\right)^{(r_1+r_2-1)M}\cdot f_M(\tfrac{n}{S}\cdot c_0(1-\tfrac{1}{k}))^{-dt/2}\\
		    &= \omega_K^M\cdot \sum_{n=1}^{\infty}\left(\tfrac{n}{S}+2\right)^{(r_1+r_2-1)M}\cdot f_M(\tfrac{n+S-1}{S}\cdot c_0(1-\tfrac{1}{k}))^{-dt/2}\\
				       &\leq \omega_K^M\cdot \sum_{n=1}^{\infty} f_M(\tfrac{S+n-1}{S}\cdot c_0(1-\tfrac{1}{k}))^{\frac{-d(t-t_0)}{2}}\\
				       &\leq \omega_K^M\cdot \sum_{n=S}^{\infty} \exp\left(-n\cdot{\frac{\alpha_M\cdot c_0(k-1)\cdot d(t-t_0)}{2kS}}\right)
				       \label{eq:geom_series}
	    \end{align}
    where $\alpha_M>0$ is a constant small enough so that $f_M(x)\geq e^{\alpha_M\cdot x}$ if $x\geq c_0/2$ and where
    $$t_0\geq\frac{2(r_1+r_2-1)M}{d}\sup_{n\in\mathbb{N}_{\geq 1}}\frac{ \log(\frac{n}{S}+2)}{\log \left(f_M(\frac{S+n-1}{S}\cdot c_0 \cdot (1-\frac{1}{k}) \right)}$$
    for suitable $k$. The logarithm ratio decays as $n$ increases and therefore it suffices to take 
     $$t_0\geq\frac{2(r_1+r_2-1)M}{d}\cdot \frac{ \log(\frac{1}{S}+2)}{\log(f_M(c_0(1-\frac{1}{k}))}.$$
	     Summing up the geometric series in (\ref{eq:geom_series}) gives us 
\begin{align}
	\Sigma_{M,k}^{\infty}
				       &\leq \omega_K^M\cdot \frac{ \exp\left(-{\frac{\alpha_M\cdot c_0(k-1)\cdot d(t-t_0)}{2k}}\right)}{ 1-
 \exp\left(-{\frac{\alpha_M\cdot c_0(k-1)\cdot d(t-t_0)}{2kS}}\right)
				       }.
\end{align}

    We chose to present the results for the choice of $S=2k$.
   
\end{proof}
\subsection{Summing over  ideals}

It remains to sum the contributions in Proposition \ref{prop:mintersections} over principal ideals. To that end, we have: 
\begin{lemma} \label{lemma:partialDedekindzeta}
Let $\mathcal{J}\subset\mathcal{O}_K$ be an integral ideal. Then for any real $t> 1$ we have: 
$$\sum_{\substack{\mathcal{I}\subset \mathcal{J}\\ I \text{integral ideal}}}\N(\mathcal{I})^{-t}=\N(\mathcal{J})^{-t}\cdot\zeta_K(t)$$
\end{lemma}
\begin{proof}
The proof follows from the definitions since we are in a Dedekind domain: for instance writing the prime decomposition $\mathcal{J}=\prod_i\mathcal{P}_i$, the left hand side becomes 
$$\sum_{\substack{\mathcal{I}\subset \calO_K\\ \mathcal{I} \text{integral ideal}}}\N(\mathcal{I}\cdot \prod_i\mathcal{P}_i)^{-t}= \N(\prod_i\mathcal{P}_i)^{-t}\cdot \sum_{\substack{I\subset \calO_K\\ \mathcal{I} \text{ integral ideal}}}\N(\mathcal{I}).$$
\end{proof}
We can now reformulate the $m=1$ term for the $n$-th moment: 
\begin{proposition}\label{prop:reformulation}
Let $\alpha_1,\ldots, \alpha_M \in K^\times$. Then the $m=1$ term in Theorem \ref{th:weil} is given for indicator functions of balls by 
$$\sum_{\alpha_1,\ldots\alpha_M\in K^\times}D(\alpha)^{-t}\cdot \vol(B\cap \alpha_1B\cap\cdots\cap \alpha_MB),$$
where $D(\alpha)$ is as defined in \eqref{def:Dalpha}. Moreover, for any function 
$f_M:K^{\times,M}\to \mathbb{R}$ and any $T\in \mathbb{R}_{> 1}$, the sum 
$$\sum_{\alpha_1,\ldots\alpha_M\in K^\times}D(\alpha)^{-T}\cdot \vol(B\cap \alpha_1B\cap\cdots\cap \alpha_MB)\cdot f_M(\alpha_1,\ldots,\alpha_M)$$equals: 
$$\zeta_K(T)^{-1}\cdot \sum_{\substack{\mathcal{I}\subset \calO_K\\ \mathcal{I} \text{integral ideal}}}\N(\mathcal{I})^{-T}\sum_{\alpha_1,\ldots,\alpha_M\in \mathcal{I}^{-1}\setminus\{0\}} \vol(B\cap \alpha_1B\cap\cdots\cap \alpha_MB)\cdot f_M(\alpha_1,\ldots,\alpha_M).$$
\end{proposition}
\begin{proof}
For the first expression, it suffices to see that the index of $\{c\in\calO_K: c\cdot\alpha_i\in\calO_K\forall i\}$ in $\calO_K$ is equivalent to the index of $(\alpha_1,\ldots,\alpha_M)^{-1}\cap \calO_K$ in $\calO_K$, where $(\alpha_1,\ldots,\alpha_M)$ denotes the fractional ideal generated by the $\alpha_i$.  
Let now $\mathcal{J}$ denote the integral ideal $(\alpha_1,\ldots,\alpha_M)^{-1}\cap \calO_K$. To establish the equivalence of the second expression, observe that for an integral ideal $\mathcal{I}\subset\calO_K$ we have
$$\alpha_1,\ldots,\alpha_M\in \mathcal{I}^{-1} \Leftrightarrow\mathcal{I}\subset (\alpha_1,\ldots,\alpha_M)^{-1}\cap \calO_K=\mathcal{J}.$$
Thus in the second expression every tuple $\alpha_1,\ldots, \alpha_M$ contributes 
$$\zeta_K(T)^{-1}\cdot\sum_{\substack{\mathcal{I}\subset \mathcal{J}\\ \mathcal{I} \text{ integral ideal}}}\N(\mathcal{I})^{-T}\cdot \vol(B\cap \alpha_1B\cap\cdots\cap \alpha_MB)\cdot f_M(\alpha_1,\ldots,\alpha_M).$$
In the first expression the contribution is $\N(\mathcal{J})^{-T}\cdot\vol(B\cap \alpha_1B\cap\cdots\cap \alpha_MB)\cdot f_M(\alpha_1,\ldots,\alpha_M)$. 
We conclude by Lemma \ref{lemma:partialDedekindzeta} that the two expressions are equal.
\end{proof}
We can now put everything together: 

\begin{proposition}\label{prop:projsumoverideals}
Assume Hypothesis \ref{hyp:Lehmer} and its notations and fix $k\geq 2$. There exist positive constants $C_M,\varepsilon_M>0$ uniformly bounded in $d,t$ such that the following holds:  write
$$t_0=\sup_{K\in \mathcal{S}}\left\{kM+\frac{1}{2},\frac{2r_K\cdot M}{d}\cdot \frac{ \log(2+\tfrac{1}{2k})}{\log(f_M(c_0(1-\frac{1}{k}))}\right\},$$
where $f_M(x):=\frac{\exp(x)+M\exp(-\frac{x}{M})}{M+1}$ and $r_K$ is the rank of the unit group. For any $t> t_0$ we then have:
$$\sum_{\alpha\in (K^\times)^M\setminus\mu_K^M} D(\alpha)^{-t}\vol(B \cap \alpha_1^{-1} B\cap \dots \cap \alpha_{M}^{-1}B)\leq C_M\cdot\omega_K^M\cdot\frac{\zeta_K(t(\frac{3}{4}-\frac{1}{k}))\cdot \zeta_K(\frac{t}{kM})^M}{\zeta_K(\frac{3t}{4})}\cdot e^{-\varepsilon_M\cdot d\cdot(t-t_0)}\cdot \vol(B). $$
 We may moreover take 
$$\varepsilon_M=\tfrac{1}{2}\min\left\{\tfrac{c_1}{8}, \log(f_M(\tfrac{3}{4}c_1)), \alpha_M\cdot \tfrac{c_0(k-1)}{k}))\right\}$$ and $C_M=(2M+1)(1+\tfrac{1}{1-e^{-\alpha_M\cdot c_0\cdot d(t-t_0)(k-1)/(4k^2)}})$, where $\alpha_M>0$ is small enough so that $f_M(x)\geq e^{\alpha_M\cdot x}$ for $x\geq c_0/2$. 
 
\end{proposition}
\begin{proof}

   After multiplying by a constant $2M$, we may in addition assume that $\N(\alpha_i)\geq 1$. Indeed we may cover  $(K^\times)^M\setminus\mu_K^M$ by $M$ sets on which $N(\alpha_i)$ is smallest for some fixed $1\leq i\leq M$. By Lemma \ref{lemma:assumenorm} and the ensuing remark, for each such set the contribution is bounded by twice the contribution of $\{\alpha\in(K^\times)^M\setminus\mu_K^M: \N(\alpha_i)\geq 1 \forall i\}$. 
	   Using Proposition \ref{prop:mintersections}, we immediately obtain (constants $C,\varepsilon_1$ as in \ref{prop:mintersections}):
  $$\sum_{\substack{\alpha\in (\OK^\times)^M\setminus(\mu_K^\times)^M}} D(\alpha)^{-t}\vol(B \cap \alpha_1^{-1} B\cap \dots \cap \alpha_{M}^{-1}B)\cdot \vol(B)^{-1} \leq C \cdot \omega_K^{M} \cdot e^{-\varepsilon_1 \cdot d \cdot (t-t_{0})} .$$
   
   It remains to bound:
  $$\sum_{\substack{\alpha\in (K^\times)^M\setminus(\calO_K^\times)^M\\ \N(\alpha_i)\geq 1}} D(\alpha)^{-t}\vol(B \cap \alpha_1^{-1} B\cap \dots \cap \alpha_{M}^{-1}B)\cdot \vol(B)^{-1}.$$
  We 
  apply Proposition \ref{prop:mintersections} and deal with bounding (constants $C,\varepsilon_1$ as in \ref{prop:mintersections}) the sum
  $$ \sum_{\substack{\alpha\in (K^\times)^M\setminus(\calO_K^\times)^M\\ \N(\alpha_i)\geq 1}} C\cdot D(\alpha)^{-t}\N(\alpha)^\frac{-t}{kM}\cdot D(\alpha)^{\frac{t}{4}}\cdot e^{-\varepsilon_1\cdot d\cdot(t-t_0)}.$$
   Using the Proposition \ref{prop:reformulation} with $T=\tfrac{3}{4}t$, it therefore suffices to bound 

$$\zeta_K(\tfrac{3t}{4})^{-1}\cdot \sum_{\substack{\mathcal{I}\subset \calO_K\\ \mathcal{I} \text{ integral ideal}}}\N(\mathcal{I})^{\frac{3t}{4}}\sum_{\substack{\alpha_i\in (\mathcal{I}^{-1}\setminus\{0\})/\calO_K^\times\\\\ \N(\alpha_i)\geq 1}}\N(\alpha)^\frac{-t}{kM}\cdot e^{-\varepsilon_1\cdot d\cdot(t-t_0)} .$$
Now observe that the map $\alpha_i\mapsto(\alpha_i)\cdot \mathcal{I}$ gives a bijection between $(\mathcal{I}^{-1}\setminus\{0\})/\calO_K^\times$ and integral ideals $\mathcal{J}\subset \calO_K$ in the ideal class of $\mathcal{I}$. We may therefore bound this expression by: 

$$\zeta_K(\tfrac{3t}{4})^{-1}\cdot\sum_{\substack{\mathcal{I} \subset \calO_K\\ \mathcal{I} \text{ integral ideal}}}\N(\mathcal{I})^{\frac{3t}{4}}\prod_{1\leq i\leq M}\sum_{\substack{\mathcal{J}\subset \calO_K\\ \N(\mathcal{I})\leq \N(\mathcal{J})}}\N(\mathcal{J}\mathcal{I}^{-1})^{-\frac{t}{kM}}\cdot e^{-\varepsilon_1\cdot d\cdot(t-t_0+1)}$$
and therefore as claimed by 
$$ \frac{\zeta_K(t(\frac{3}{4}-\frac{1}{k}))\cdot \zeta_K(\frac{t}{kM})^M}{\zeta_K(\frac{3t}{4})}\cdot e^{-\varepsilon_1\cdot d\cdot(t-t_0)}.$$
We see that in particular taking $k\geq 2$ and $t\geq kM+1/2$ suffices for convergence of the zeta factors for any given $d$ and we obtain the explicit constants by setting $C_M=(2M+1)\cdot C$ and $\varepsilon_M=\varepsilon_1$.
  
\end{proof}
We therefore find: 
\begin{theorem}\label{thm:mainballintersection}
Let $\mathcal{S}$ denote any set of number fields satisfying Hypothesis \ref{hyp:Lehmer} and let $c_0, c_1$ denote the resulting uniform constants. 
For any choice of $k\geq 2$ there exist positive constants $C_M,\varepsilon_M>0$ uniformly bounded in $d,t$ such that the following holds: write
     $$t_0=\sup_{K\in \mathcal{S}}\left(kM+\frac{1}{2},\frac{2r_K\cdot M}{d}\cdot \frac{ \log(2+\tfrac{1}{2k})}{\log(f_M(c_0(1-\frac{1}{k}))}\right),$$
where $$f_M(x)=\tfrac{\exp(x)+M\exp(-x/M)}{M+1}$$ and $r_K$ is the rank of the unit group. 
We then have for any $t> t_0$ and for any $K\in \mathcal{S}$ of degree $d$:
    $$\sum_{\alpha\in (K^\times)^M} D(\alpha)^{-t}\vol(B \cap \alpha_1^{-1} B\cap \dots \cap \alpha_{M}^{-1}B)=\vol(B)\cdot \omega_K^M\left( 1+ C_M \cdot Z(K,t,M,k)\cdot e^{-\varepsilon_M\cdot d\cdot (t-t_0)}\right),$$
    where $$0\leq Z(K,t,M,k)\leq\zeta_K\big( t(\tfrac{3}{4}-\tfrac{1}{k})\big)\cdot \zeta_K(\tfrac{t}{kM})^M\cdot\zeta_K(\tfrac{3t}{4})^{-1}.$$ 
    We may moreover take $\varepsilon_M$ and $C_M$ as in Proposition \ref{prop:projsumoverideals}.
\end{theorem}
\begin{proof}
    This follows from the previous proposition and the fact that $\vol(B \cap \alpha^{-1} B)=\vol(B)$ if $\alpha\in \mu_K$.
\end{proof}
\begin{remark}
    If we consider cyclotomic fields of increasing degree, we may take $c_0=0.24$ and for $M=1$ the condition on $t$ is satisfied for $t_0<27,k=26$. For $M=2,3,4,5$ we get $t_0<97,213,372,576$ and $k=48,70,92,115$. 
    As a function of $M$, a calculation shows that we have $t_0\leq CM^2$ for $C\approx 22.18\ldots$ as $M$ grows.
\end{remark}

Note that $\omega_K^M=o(d^{M+1})$ so that we indeed obtain exponential decay of the error term and in particular deduce a result for the second moment: 
\begin{corollary}\label{cor:mainsecondmoment}
    Let $\mathcal{S}$ denote any set of number fields satisfying Hypothesis \ref{hyp:Lehmer} and let $c_0,c_1$ denote the resulting uniform constant. Then for any choice of $k\geq 2$ there exist positive uniformly bounded constants $C,\varepsilon>0$ such that the following holds: write
    $$t_0=\sup_{K\in \mathcal{S}}\left\{k+\frac{1}{2},\frac{2r_K}{d}\cdot \frac{ \log(2+\tfrac{1}{2k})}{\log(\cosh(c_0(1-\tfrac{1}{k}))} \right\},$$
where $r_K$ is the rank of the unit group. We then have for any $t> t_0$ and for any $K\in \mathcal{S}$ of degree $d$ that the second moment $ \mathbb{E}[\rho(\Lambda)^2]$ of the number of nonzero $\calO_K$-lattice points in a fixed origin-centered ball of volume $V$ in $K_\mathbb{R}^t$ satisfies: 
\begin{align}
       V^2+\omega_K\cdot V&\leq \mathbb{E}[\rho(\Lambda)^2]\\
       &\leq   V^2+\omega_K\cdot V+\omega_K^2\cdot C \cdot Z(K,t,k)\cdot e^{-\varepsilon\cdot d\cdot (t-t_0)}\cdot V,\\
    \end{align}
    where $0\leq Z(K,t,k)\leq\zeta_K\big(t(\frac{3}{4}-\frac{1}{k})\big)\cdot \zeta_K(\frac{t}{k})\cdot\zeta_K(\frac{3t}{4})^{-1}$. \par
    We may moreover take $\varepsilon=\tfrac{1}{2}\min(\tfrac{c_1}{8}, \log(\cosh(3c_1/4))), \tfrac{2c_0(k-1)}{5k}))$ and $C=3+\tfrac{3}{1-e^{-c_0\cdot d(t-t_0)(k-1)/(10k^2)}}$.
\end{corollary}
\begin{proof}
    This follows from Theorem \ref{thm:mainballintersection} for $M=1$. The explicit constants can be obtained by bounding $\cosh(x)>e^{\tfrac{2}{5}\min(x,x^2)}$. 
\end{proof}
See Corollary \ref{cor:introcyclo} for the ensuing second moment result for cyclotomic fields. 
To go beyond the second and third moments we shall extend this approach in the next section. 


\subsection{General error estimates for  \texorpdfstring{$A^2_m$}{A2m}-type terms}
\label{subse:a2m}

In this section, we estimate the contributions of more general subspaces of dimension $m$ to the integral formula by reducing to our previous considerations for projective space. Recall from Section \ref{se:poisson} the set of matrices $$A^2_{m} = \left\{ D\in  M_{m \times n}( K) \ {\Big| } \substack{ \ D_{ij} \in  K  ,\\  D \text{ is in row-reduced echelon form of }\rank(D) = m \\ D \text{ has at least one entry }\not\in \mu_K\cup\{0\}
} 
\right\}. $$

The main result of this section is the following: 
\begin{theorem}\label{thm:generalBogobound}
   Let $\mathcal{S}$ denote any set of number fields satisfying Hypothesis \ref{hyp:Lehmer} and let $c_0, c_1$ denote the resulting uniform constants. Fix $n$ and $2\leq m<n$. There exist explicit positive uniform constants $C_\mathcal{S},\varepsilon_\mathcal{S}>0$ such that the following holds: write $t_0$ for
    $$2(n-m)\cdot \sup_{K\in\mathcal{S}}
    \left\{ m^2+m,\frac{r_K(m^2+m)}{d}\cdot
    \frac{\log(2+12c_0^{-1}+2\log(n-m)\cdot c_0^{-1})}{\log \left( \min\{ \frac{64}{27},e^{\frac{1}{3}c_1},\cosh^3(c_1)\} \right) }, \frac{r_K}{d}\cdot\log_2^{-1}(f_{n-m}(\tfrac{3}{4}c_0))\right\},$$
where $r_K$ is the rank of the unit group and $f_{n-m}(x)=\tfrac{e^{x}+(n-m)e^{-\frac{x}{n-m}}}{1+n-m}$. We then have for any $t> t_0$ and for any $K\in \mathcal{S}$:
$$  \frac{1}{V(td)^m R^{mtd}} \sum_{D\in A^2_m} \frac{1}{\mathfrak{D}(D)^{t}}\int_{K_\mathbb{R}^{m \times t }}   f(xD) dx  \leq  C_\mathcal{S}\cdot \omega_K^{m(n-m)}(td)^{\tfrac{m-1}{2}} \cdot Z(K,t,n,m)\cdot e^{-\varepsilon_\mathcal{S}\cdot d\cdot (t-t_0)},$$
    with the zeta factor 
$$Z(K,t,n,m)=\frac{\zeta_K(\frac{1}{2(m+1)}t-\frac{1}{e}m(n-m))\cdot\zeta_K(\frac{t}{4m(n-m)})^{m(n-m)}}{\zeta_K(t-1)}.$$
Moreover, we may take 
$$\varepsilon_\mathcal{S}=\tfrac{1}{2}\log(\min\{ \tfrac{4}{3}, e^{\frac{c_1}{3(m+1)}},f_{n-m}(3c_1/4)\}).$$ The constant $C_\mathcal{S}$ may also be chased down as a function depending only $m,n,c_0,c_1$. 
\end{theorem}
\begin{remark}
    Note that despite the relatively ugly expression for the minimal rank $t_0$, we have that $t_0(n)=O(n^3\log\log n)$ as the moment $n$ increases with a constant only depending on the choice of number fields $\mathcal{S}$. 
\end{remark}
In order to prove the theorem, we will actually subdivide the $A^2_m$-terms as follows: write 
\begin{align}
A^{2,0}_{m}& = \left\{ D\in  M_{m \times n}( K) \ {\Big| } \substack{ \ D_{ij} \in  K  ,\\  D \text{ is in row-reduced echelon form of }\rank(D) = m \\ 
D \text{ has exactly one non-zero entry per column}\\D \text{ has at least one entry }\not\in \mu_K\cup\{0\}.} 
\right\},\\
A^{2,h_0}_{m} & = \left\{ D\in  M_{m \times n}( K) \ {\Big| } \substack{ \ D_{ij} \in  K  ,\\  D \text{ is in row-reduced echelon form of }\rank(D) = m \\ 
D \text{ has all entries of Weil height less than } h_0\\D \text{ has at least one entry }\not\in \mu_K\cup\{0\}.} 
\right\} 
\setminus A^{2,0}_{m},\\
A^{2,\infty}_{m} & = \left\{ D\in  M_{m \times n}( K) \ {\Big| } \substack{ \ D_{ij} \in  K  ,\\  D \text{ is in row-reduced echelon form of }\rank(D) = m \\ 
D \text{ has at least one entry of Weil height larger than } h_0} 
\right\}\setminus A^{2,0}_{m}
\end{align}
for a suitable choice of threshold height $h_0>0$. These sets clearly cover $A^2_m$ and we show that the contribution of each term decays exponentially. \par
Consider first the $A^{2,0}_{m}$-type terms. In this case, the contributions can via a separation of variables be reduced to products of intersections of shifted balls as in subsection \ref{subsec:projBogo}. We have thus already done all the work and the results follow from Theorem \ref{thm:mainballintersection}. This in turn allows us to assume that $D$ in $A^{2,h_0}_{m}$ has at least one column with multiple entries. We prove the contributions of such terms decay similarly to Lemma \ref{le:lemma5_rogers} even for relatively small height by cherry-picking a particular column of $D$ to which to apply estimates (see Lemma \ref{lemma:grtoproj}). Finally, $h_0$ is chosen large enough so that the terms in $A_m^{2,\infty}$ have exponentially decaying contributions purely for height reasons. \par
The following convex combination lemmas will allows us to handle the $A^{2,h_0}_{m}$ and $A_m^{2,\infty}$-type terms.

\begin{lemma}
\label{lemma:convexgen}
Let $f:K_\mathbb{R}^{t}\rightarrow \mathbb{R}$ be the indicator function of a ball of radius $R > 0$ and assume $n>  m\geq 2$. Then, for any $(\alpha_{i,j})\in M_{(n-m)\times m}(K)$, we have that   
\begin{align}
&\frac{ \int_{K_\mathbb{R}^{m \times t }} f(x_1)\cdots f(x_m)\prod_{j=1}^{n-m} f(\sum_{i=1}^m \alpha_{i,j}x_i) dx_1\cdots dx_m  }{ V(mt[K:\mathbb{Q}]) R^{mt[K:\mathbb{Q}]} } \\
& \le (m+1)^{mtd/2}\cdot \min_{1\leq k\leq n-m} \min_{J \in {\binom{[n-m]}{k}}}   \prod_{\sigma: K \rightarrow \mathbb{C} } \left(1 + \frac{1}{k}\sum_{j\in  J}\sum_{i=1}^m|\sigma(\alpha_{i,j})| ^{2}\right)^{-\frac{t}{2}}.
\end{align}
\end{lemma}
\begin{proof}
We will again use the idea of convex combinations, see Lemma \ref{le:convex_combinations} of the appendix for a slightly more general result and an alternative derivation. 
Let $c_k \in [0,1]$ for $1\leq k\leq n$ be any coefficients satisfying $\sum_{k=1}^n c_k= 1$. 
Then for $(x_1,\ldots,x_m)\in K_\mathbb{R}^{t \times m}$ the conditions $\|x_1\| \le R,\cdots, \| x_m\| \le R$ and $\|\sum_{i=1}^m \alpha_{i,j}x_i \| \le R$ for $1\leq j\leq n-m$ imply that 
\begin{equation}
 c_1 \|x_1\|^{2} +\cdots +c_m \| x_m\|^2 + \sum_{j=1}^{n-m}c_{j+m}\cdot  \|\sum_{i=1}^m \alpha_{i,j}x_i  \|^{2} \le R^{2}.
 \label{eq:defines_an_ellipsoid}
 \end{equation}
Equation (\ref{eq:defines_an_ellipsoid}) then defines an ellipsoid in $K_\mathbb{R}^{t \times m }$. The relevant quadratic form is scaled by a symmetric matrix that in each copy of $\mathbb{R}^m$ looks like (after fixing one of the $t$ copies and an embedding $\sigma:K\to \IC$): 
\begin{equation}
A_\sigma:=\begin{bmatrix} c_1 + \sum_{j=1}^{n-m}c_{j+m}\sigma(\alpha_{1,j})\overline{\sigma(\alpha_{1,j})} &\cdots&\sum_{j=1}^{n-m}c_{j+m}\sigma(\alpha_{1,j})\overline{\sigma(\alpha_{m,j})}  \\
\vdots & \ddots &\vdots \\
\sum_{j=1}^{n-m}c_{i+m}\sigma(\alpha_{i,m})\overline{\sigma(\alpha_{i,1})}  & & c_m + \sum_{j=1}^{n-m}c_{j+m}\sigma(\alpha_{m,j})\overline{\sigma(\alpha_{m,j})}
\end{bmatrix}
\end{equation}
 It therefore suffices to give a lower bound on $\det(A_\sigma)$, since the volume ratio equals 
 $$ \prod_{\sigma: K \rightarrow \mathbb{C}}\sqrt{ \det(A_\sigma)}^{-t}.$$
We now make the choice of $c_1=\cdots=c_m=\tfrac{1}{m+1}$ and may for $j\in J$ set $ c_{m+j}=\tfrac{1}{k(m+1)}$ and take the remaining convex coefficients to be zero. A bound on $\det(A_\sigma)$ may now be deduced from combinatorics. For instance, it is known that the coefficient of $z^{m-k}$ in $\det (z\cdot\Id +X)$ is the sum of the $k\times k$ principal minors of a square matrix $X$. Writing $A_\sigma= z\cdot \Id +X$ for $z=1/(m+1)$ we see that $X$ is a positive semidefinite Hermitian matrix and therefore its principal minors are nonnegative. We thus obtain a lower bound by keeping the terms in $z^m$ and $z^{m-1}$ resulting in the bound 
$$\det(A_\sigma)\geq (m+1)^{-m} (1+\frac{1}{k}\sum_{j\in  J}\sum_{i=1}^m|\sigma(\alpha_{i,j})| ^{2}).$$
The result follows since this is valid for any choice of $k$ non-pivot columns $J$. 
 
\end{proof}
Recall now that we write for nonzero $\alpha\in (K^\times)^M$ and for some integer $M>0$ the height: 
$$H_\infty(\alpha)=
\prod_{\sigma:K\to\mathbb{C}}\max_{1\leq j\leq M}\max(1,\vert\sigma(\alpha_j)\vert).
$$

\begin{lemma}\label{lemma:grtoproj}
Let $D\in M_{m\times n}(K)$ be a row-echelon matrix of rank $m$ written as $D=(\Id_m\mid \alpha)$ for entries $\alpha_{i,j}\in K$. Let $f:K_\mathbb{R}^{t}\rightarrow \mathbb{R}$ be the indicator function of a ball of unit radius and assume $n> m\geq 1$. For any fixed column of $(\alpha)_{ij}$ with $\alpha_1,\ldots, \alpha_M\in K^\times$ denoting its non-zero entries we have the bound:
\begin{align}
&\frac{ \int_{K_\mathbb{R}^{m \times t }} f(x_1)\cdots f(x_m)\prod_{j=1}^{n-m} f(\sum_{i=1}^m \alpha_{i,j}x_i) dx_1\cdots dx_m  }{ V(t[K:\mathbb{Q}])^m } \\
&\le  (M+1)^{Mtd/2}\cdot \frac{V(t[K:\mathbb{Q}]M)}{V(t[K:\mathbb{Q}])^M} \left(H_\infty(\alpha)^{\frac{2}{d}}+M \N(\alpha)^{2/(dM)}\cdot H_\infty(\alpha)^{\frac{-2}{dM}}\right)^{-\frac{dt}{2}},
\end{align}
where we abbreviate $\N(\alpha)$ for $\N(\alpha_1\cdots\alpha_{M})$. 
\end{lemma}
\begin{proof}
We induct on $m$. For any column $j$ of $(\alpha)_{ij}$, first observe that we have the trivial bound 
$$\int_{K_\mathbb{R}^{m \times t }} f(x_1)\cdots f(x_m)\prod_{j=1}^{n-m} f(\sum_{i=1}^m \alpha_{i,j}x_i) dx_1\cdots dx_m\leq \int_{K_\mathbb{R}^{m \times t }} f(x_1)\cdots f(x_m)f(\sum_{i=1}^m \alpha_{i,j}x_i) dx_1\cdots dx_m.$$
We shall prove by induction that the right hand side is bounded. When $m=M=1$, the claimed bound is a special case of Lemma \ref{lemma:genvolumeratio}. Let now $m\geq 2$ arbitrary. If $M=m$
we apply Lemma \ref{lemma:convexgen} and reduce to a term that looks like the height of the class $(1:\alpha_1:\cdots:\alpha_m)$ in projective space. 
Comparing heights as in Lemma \ref{le:lower_heigh_bound} we then obtain the claimed bound. Finally, if $M<m$,  writing $x_1,\ldots, x_M$ for the variables corresponding to rows with non-zero entries in the $j$-th column we have: 
$$\frac{\int_{K_\mathbb{R}^{m \times t }} f(x_1)\cdots f(x_m)f(\sum_{i=1}^m \alpha_{i,j}x_i) dx_1\cdots dx_m}{V(t[K:\mathbb{Q}])^m}=\frac{\int_{K_\mathbb{R}^{m \times t }} f(x_1)\cdots f(x_M)f(\sum_{i=1}^M \alpha_{i}x_i) dx_1\cdots dx_M}{V(t[K:\mathbb{Q}])^M}$$
by separating variables. But the latter is bounded by exactly the desired term by induction. 
\end{proof}
We also record the result taking into account all of the columns: 
\begin{lemma}\label{lemma:grtoprojallcols}
Let $D\in M_{m\times n}(K)$ be a row-echelon matrix of rank $m$ written as $D=(\Id_m\mid \alpha)$ for entries $\alpha_{i,j}\in K$, exactly $M$ entries $\alpha_1,\ldots, \alpha_M$ of them non-zero. Let $f:K_\mathbb{R}^{t}\rightarrow \mathbb{R}$ be the indicator function of a ball of unit radius and assume $n> m\geq 1$. Then 
\begin{align}
&\frac{ \int_{K_\mathbb{R}^{m \times t }} f(x_1)\cdots f(x_m)\prod_{j=1}^{n-m} f(\sum_{i=1}^m \alpha_{i,j}x_i) dx_1\cdots dx_m  }{ V(t[K:\mathbb{Q}])^m } \\
&\le  (m+1)^{mtd/2}\cdot \frac{V(t[K:\mathbb{Q}]m)}{V(t[K:\mathbb{Q}])^m} \left(e^{2\cdot h_\infty(\sqrt{\tfrac{1}{n-m}}\alpha)}+\tfrac{M}{n-m} \N(\sqrt{\tfrac{1}{n-m}}\alpha)^{2/(dM)}\cdot e^{\tfrac{-2}{M}\cdot h_\infty(\sqrt{\tfrac{1}{n-m}}\alpha)}\right)^{-\frac{dt}{2}},
\end{align}
where we abbreviate $\N(\alpha)$ for $\N(\alpha_1\cdots\alpha_{M})$. 
\end{lemma}
Note that we use absolute heights in the statement to obtain the right result independently of whether $\sqrt{n-m}\in K$.
\begin{proof}
    We apply Lemma  \ref{lemma:convexgen} for the full number of columns. This yields a term that looks like the height of the class $(1:\sqrt{\tfrac{1}{n-m}}\alpha_1:\cdots:\sqrt{\tfrac{1}{n-m}}\alpha_M)$ in projective space. Comparing heights as in Lemma \ref{le:lower_heigh_bound} we then obtain the claimed bound. 
\end{proof}

We may now sum up these contributions over units to obtain similarly to Proposition \ref{prop:mintersections}:

\begin{proposition}\label{prop:GrWeilunitsum}
    Assume Hypothesis \ref{hyp:Lehmer} and its notations. Assume $n>m\geq 2$. There exist explicit positive constants $C,\varepsilon_1>0$ uniform in $d,t$ such that for all $D=(\Id_m\mid \alpha_{ij})$ in $A_m^2\setminus A_m^{2,0}$ the following holds:  write
     $$t_0=\frac{2r_K \cdot m(m+1)(n-m)\cdot\log(2+12c_0^{-1}+2\log
   (n-m)\cdot c_0^{-1})}{d\log (s)},$$
where $r_K$ is the rank of the unit group and $s= \min(\tfrac{64}{27}, e^{\tfrac{c_1}{3}}, \cosh^3(c_1))>1$ is a constant depending only on the choice of number fields. Let $f$ denote the indicator function of a ball of radius $R$ and let $\alpha_1,\ldots,\alpha_M$ for $n-m+1\leq M\leq m(n-m)$ denote the nonzero entries of $(\alpha)_{ij}$. 
Write $D_\beta=(\Id_m\mid \beta\alpha)$ for $\beta\in (\calO_K^\times)^M$, where we scale the nonzero entries $\alpha_i\mapsto\beta_i\alpha_i$ and $D(\alpha)$ is as defined in \ref{def:Dalpha}. Then for any $t> t_0$ we have the bound:

\begin{align}
&\frac{1}{ V(td)^m R^{mtd} }\sum_{\beta\in(\calO_K^\times)^M}  \int_{K_\mathbb{R}^{m \times t }} f(xD_\beta) dx_1\cdots dx_m  \\
&\le  C\cdot \omega_K^M \cdot(td\pi)^{m/2}\cdot \max (N(\alpha)^\frac{-t}{4M}, N(\alpha)^\frac{-tm}{(m+1)M})\cdot D(\alpha)^{\tfrac{t}{2(m+1)}+\tfrac{r_kM}{ed}}\cdot  e^{-\varepsilon_1\cdot d\cdot (t-t_0)}.
\end{align}
Moreover, we may e.g. choose $\varepsilon_1= \tfrac{1}{2}\log(\min\{ \tfrac{4}{3}, e^{\frac{c_1}{3(m+1)}}, \cosh\left(c_1\right)\})$ and $C=\frac{4}{1-e^{d(t_0-t)/2}}$.

\end{proposition}

\begin{proof}
The proof proceeds similar to Proposition \ref{prop:mintersections} and uses Lemmas \ref{lemma:grtoproj} and \ref{lemma:grtoprojallcols}. 
We first record a count of unit $M$-tuples $\beta$ with bounded height after scaling by $\alpha$. 
Note that $D(\alpha)\in\mathbb{Z}_{\geq 1}$ by definition and moreover for any $1\leq i\leq M$ we have that 
$$\max(1, \N(\alpha_i)^{-1})\leq D(\alpha_i)\leq D(\alpha).$$
We may therefore apply Lemma \ref{lemma:unitcount} and bound
\begin{align}
    &\card \{\beta\in(\mathcal{O}_K^\times)^M\ \mid \ \tfrac{1}{d}\log H_\infty(\alpha\beta)\leq B\}\\
    \leq~&  \omega_K^M\cdot\prod_{i=1}^M \left(\tfrac{B+\max(0, \log(\N\left(\alpha_i\right)^{\frac{-1}{d}}))
+\tfrac{c_0}{2}}{\tfrac{c_0}{2}}\right)^{r_K}\\
    \leq~& \omega_K^M\cdot \left(\tfrac{B+\log
   (D\left(\alpha\right)^{\frac{1}{d}}))
+\tfrac{c_0}{2}}{\tfrac{c_0}{2}}\right)^{r_KM}.\\
\end{align}

 Note also that we will systematically use the bounds on the volume ratios involving unit balls in Lemma \ref{le:gamma_ratio} when applying Lemma \ref{lemma:grtoproj}. The approximation $\tfrac{V(ktd)}{V(td)^k}\approx k^{-ktd/2}$ will be factored into our estimates for $1\leq k\leq m$ whereas the error term bound in the Stirling approximation $p_k(t,d):=\frac{(td\pi)^{(k-1)/2}}{\sqrt{k}}\cdot e^{k/(6td)}$ ultimately yields the factor $(td\pi)^{m/2}$ in the statement of the proposition. \par 
\textbf{Type }$\mathbf{A_m^{2,h_0}}$ terms. We first estimate the sum for terms $D_\beta\in A_m^{2,h_0}$. We claim that since $D\in A_m^2\setminus A_m^{2,0}$, there exists a column of $D$ with $k$ non-zero entries $\alpha_j=(\alpha_{j1},\ldots, \alpha_{jk})$ satisfying 
\begin{equation}
    \N(\alpha_j)\geq \N(\alpha)^{\frac{k}{M}}\text{ and }\left(\alpha_j\notin( \OK^{\times} )^{k}\text{ or }k\geq 2.\right)
\end{equation}
Indeed, consider the nonempty set $J\subset \{1,\ldots,n\}$ of columns with multiple non-zero entries. If all non-zero entries of $D$ outside of $J$ are units, then we are done since $\OK^\times$-entries have unit norm and thus the norm condition is also satisfied for one of the columns of $J$. It remains to deal with the case when all the columns in $J$ fail the norm condition. 
But then the set $J'\subseteq \{1,\cdots...,n\} \setminus J$ of columns of $D$ with exactly one non-unit entry must be non-empty. Since at least one column in all of $\{1,\cdots,n\}$ must have the norm condition, 
in this case it will be a column in $J'$. Hence we get a column with the desired property. \par
Let $\alpha_j=(\alpha_{j1},\ldots, \alpha_{jk})$ henceforth denote such a column with its $k$ non-zero entries. Among the $M$ nonzero entries of $D$, the indices $\{j1,\ldots, j_k\}$ pick out a $k$-element subset of $\{1,\ldots,M\}$. Given $\beta\in(\mathcal{O}_K^\times)^M$, we shall therefore in what follows write $\alpha_j\beta$ for the $k$-tuple of algebraic numbers $\alpha_j=(\alpha_{j1}\beta_{j1},\ldots, \alpha_{jk}\beta_{jk})$. We apply Lemma \ref{lemma:grtoproj} to the column $\alpha_j$ in order to establish the proposition for terms in $\mathbf{A_m^{2,h_0}}$.

This yields, incorporating the unit counts and Stirling approximation terms above: 
\begin{align}
& 
    \sum_{\beta \in (\OK^{\times})^{M}}
\frac{ \int_{K_\mathbb{R}^{m \times t }} f(xD_\beta) dx_1\cdots dx_m  }{ V(td)^m R^{mtd} 
 }\\
    &\le (1+\tfrac{1}{k})^{ktd/2}\cdot p_k(t,d)\cdot \sum_{\substack{\beta\in(\calO_K^\times)^M\\ D_\beta\in  A_m^{2,h_0}}}\left(H_\infty(\alpha_j\beta)^{\frac{2}{d}}+k \N(\alpha_j)^{\frac{2}{dk}}\cdot 
    H_\infty(\alpha_j\beta)^{\frac{-2}{dk}}\right)^{-\frac{dt}{2}}\\
    &\leq (1+\tfrac{1}{k})^{ktd/2}\cdot p_k(t,d)\cdot \card \{\beta\in(\mathcal{O}_K^\times)^M\mid \tfrac{1}{d}\log H_\infty(\alpha_j\beta)\leq h_0\}\cdot f_k(\alpha_j)^{-\frac{dt}{2}}\\
    &\leq (1+\tfrac{1}{k})^{ktd/2}\cdot p_k(t,d)\cdot \omega_K^M\cdot 
    \left(\frac{h_0+\log(D\left(\alpha\right)^{\frac{1}{d}}))+\frac{c_0}{2}}{\tfrac{c_0}{2}}\right)^{r_KM}
    \cdot f_k(\alpha_j)^{-\frac{dt}{2}},
\end{align}
writing $p_k(t,d)=\frac{(td\pi)^{(k-1)/2}}{\sqrt{k}}\cdot e^{k/(6td)}$ and setting 
$$f_k(\alpha_j):= \min_{\substack{\beta\in(\calO_K^\times)^M\\D_\beta\in  A_m^{2,h_0}}} H_\infty(\alpha_j\beta)^{\frac{2}{d}}+k \N(\alpha_j)^{\frac{2}{dk}}\cdot H_\infty(\alpha_j\beta)^{\frac{-2}{dk}}.$$
We wish to give a lower bound on $f_k(\alpha_j)$. To that end, recall that by Hypothesis \ref{hyp:Lehmer} there is a lower bound 
\begin{equation}\label{eq:HDbound}
    H_\infty(\alpha_j\beta)\cdot D(\alpha_j)=H_\infty(\alpha_j\beta)\cdot D(\alpha_j\beta)\geq e^{dc_1}
\end{equation}
for some $c_1>0$ as long as $\alpha_j\beta\notin\mu_K^k$. Moreover we remark that by definition $D(\alpha_j)\leq D(\alpha)$. We distinguish two cases: 
\begin{description}
    \item[Case 1: ]The denominators are large so that $D(\alpha_j)\geq e^{\frac{1}{3}dc_1}$ or we have at least $k\geq 2$ non-zero entries $(\alpha_{j1},\dots,\alpha_{jk})$.
	    We then simply bound $f_k(\alpha_j)$ by taking its minimum as a function of the Weil height. 
	    It occurs when the equality $$H_\infty(\alpha_j\beta)^{\frac{2}{d}}=\N(\alpha_j)^{\frac{2}{d(k+1)}}$$ is satisfied and we obtain that 
    \begin{equation}
	    \label{eq:bigdenom}
   f_k(\alpha_j)^{-\frac{dt}{2}}\leq \N(\alpha_j)^{-\frac{t}{k+1}}\cdot (1+k)^{-\frac{1}{2}dt}\text{ together with }\left(D(\alpha_j)\geq e^{\frac{1}{3}dc_1}\text{ or }k\geq 2\right) . \end{equation}

We therefore have in the case where $D(\alpha_j)\geq e^{\frac{1}{3}dc_1}$ that
$$(1+\tfrac{1}{k})^{ktd/2} \cdot f_k(\alpha_j)^{-\frac{dt}{2}}\cdot D(\alpha)^{-\tfrac{t}{2(m+1)}}\leq \N(\alpha_j)^{-\frac{t}{k+1}} \cdot e^{-\tfrac{td}{2}\cdot\left(\tfrac{c_1}{3(m+1)}\right)} \cdot \left(
\frac{(k+1)^{(k-1)}}{k^k}\right)^{\frac{td}{2}}.
$$
If $k=1$, this gives us $$(1+\tfrac{1}{k})^{ktd/2} \cdot f_k(\alpha_j)^{-\frac{dt}{2}}\cdot D(\alpha)^{-\tfrac{t}{2(m+1)}}\leq \N(\alpha_j)^{-\frac{t}{k+1}} \cdot e^{-\tfrac{td}{2}\cdot\left(\tfrac{c_1}{3(m+1)}\right)} ,
$$
otherwise we know that 
$$\frac{(k+1)^{(k-1)}}{k^k} \le 
\frac{3}{4} \text{ for }k \ge 2
$$
and therefore under the assumption that either $D(\alpha_j)\geq e^{\frac{1}{3}dc_1}\text{ or }k\geq 2$, we can conclude
   $$(1+\tfrac{1}{k})^{ktd/2} \cdot f_k(\alpha_j)^{-\frac{dt}{2}}\cdot D(\alpha)^{-\tfrac{t}{2(m+1)}}\leq \N(\alpha_j)^{-\frac{t}{k+1}} \cdot e^{-\tfrac{td}{2}\cdot\min\left(\log\left(\tfrac{4}{3}\right),\tfrac{c_1}{3(m+1)}\right)}.$$
Thus, taking into account that $\N(\alpha_j )\ge \N(\alpha)^{\frac{k}{M}}$, we get that for any $k \ge 1$
$$\N(\alpha_j )^{\frac{1}{k+1}}\ge \N(\alpha)^{\frac{k}{M(k+1)}} \ge\min(\N(\alpha)^{\frac{m}{(m+1)M}},\N(\alpha)^{\frac{1}{2M}}) ,$$ 
where we upper or lower bound the exponent depending on whether $\N(\alpha)\leq 1$ or not. So we have 
 $$(1+\tfrac{1}{k})^{ktd/2} \cdot f_k(\alpha_j)^{-\frac{dt}{2}}\cdot D(\alpha)^{-\tfrac{t}{2(m+1)}}\leq \max (N(\alpha)^\frac{-t}{2M}, N(\alpha)^\frac{-tm}{(m+1)M}) \cdot e^{-\tfrac{td}{2}\cdot\min\left(\log\left(\tfrac{4}{3}\right),\tfrac{c_1}{3(m+1)}\right)}.$$

    \item[Case 2: ] 
	    The denominators satisfy $D(\alpha_j)< e^{\frac{1}{3}dc_1}$ and $k=1$. 
	    Then note that by our assumptions $\alpha_j\notin \OK^\times$ and 
	    therefore for any $\beta\in\calO_K^\times$ we have 
	    that $\alpha_j\beta\notin\mu_K$. Hence we deduce via \eqref{eq:HDbound} that $H_\infty(\alpha_j\beta)\geq e^{2dc_1/3}$. Moreover, we may rewrite
    $$f_k(\alpha_j)\geq \min_{\substack{\beta\in(\calO_K^\times)^M\\D_\beta\in  A_m^{2,h_0}}} \left(\N(\alpha_j)^{ \frac{1}{2d}}\cdot g\left(\frac{H_\infty(\alpha_j\beta)}{\N(\alpha_j\beta)^{1/4}}\right)\right)\text{ for }g(x)=x^{\frac{2}{d}}+x^{- \frac{2}{d}}$$
    and given that $g$ is increasing in the range $[1,\infty[$ 
    and $H_\infty(\alpha_j\beta)\geq \N(\alpha_j\beta)$ we get
\begin{equation}
g\left(\frac{H_\infty(\alpha_j\beta)}{\N(\alpha_j\beta)^{\frac{3}{4}}}\right) \geq 
g\left(H_\infty(\alpha_j\beta)^{\frac{3}{4}} \right) 
\end{equation}
   so that we can bound 
    \begin{align}\label{eq:smalldenomone}
	    (1+\tfrac{1}{k})^{ktd/2} \cdot f_k(\alpha_j)^{-\frac{dt}{2}}
     &\leq \N(\alpha_j)^{-\frac{t}{4}}\cdot(1+\tfrac{1}{k})^{ktd/2} \cdot g(e^{\frac{1}{2}dc_1})^{-dt/2}\\
     &\leq \N(\alpha_j)^{-\frac{t}{4}}\cdot \cosh(c_1)^{-dt/2}.
     \end{align}

Taking into account that 
$\N(\alpha_j )\ge \N(\alpha)^{\frac{k}{M}}$ and $k=1$, we can write\begin{align}
	    (1+\tfrac{1}{k})^{ktd/2} \cdot f_k(\alpha_j)^{-\frac{dt}{2}}
     &\leq \N(\alpha)^{-\frac{t}{4M}}\cdot \cosh(c_1)^{-dt/2}.
     \end{align} 

\end{description}

Putting all of these cases and bounds together, we obtain the upper bound on the volume ratio
\begin{align}
&\sum_{\substack{\beta\in(\calO_K^\times)^M\\D_\beta\in  A_m^{2,h_0}}} 
\frac{ \int_{K_\mathbb{R}^{m \times t }} f(xD_\beta) dx_1\cdots dx_m  }{ V(td)^m R^{mtd} 
 }\\
 &\le 3\cdot p_m(t,d)\cdot \max (N(\alpha)^\frac{-t}{4M}, N(\alpha)^\frac{-tm}{(m+1)M})\cdot D(\alpha)^{\tfrac{t}{2(m+1)}}\cdot \omega_K^M \cdot\left(\tfrac{h_0+\log (D\left(\alpha\right)^{\frac{1}{d}}))+\tfrac{c_0}{2}}{\tfrac{c_0}{2}}\right)^{r_KM}\cdot S^{-dt/2},
\end{align}

where  $p_m(t,d)= \tfrac{(td\pi)^{(m-1)/2}}{\sqrt{m}}\cdot e^{m/(6td)}$ and $S>1$ is given by
$$S=\min\{ \tfrac{4}{3}, e^{\frac{c_1}{3(m+1)}}, \cosh\left(c_1\right)\}.$$ 
Note that the various values of $S$ correspond to the cases when $k\geq 2$, $D(\alpha_j)\geq e^{\frac{1}{3}dc_1}$ or the remaining case. 

\par
It now suffices to find $t_0$ large enough so that
$$\left(\frac{h_0+\log(D\left(\alpha\right)^{\frac{1}{d}}))+\frac{c_0}{2}}{\tfrac{c_0}{2}}\right)^{r_KM}\cdot S^{-dt/2}\leq e^{-\varepsilon_1\cdot d\cdot (t-t_0)} \cdot
D(\alpha)^\frac{r_kM}{ed},$$

with $\varepsilon_1$ as in the statement of the proposition. By Jensen's inequality we may bound  
\begin{align}
     \left(\frac{h_0+\log(D\left(\alpha\right)^{\frac{1}{d}}))+\frac{c_0}{2}}{\tfrac{c_0}{2}}\right)^{r_KM}
    \le
    \  2^{r_kM-1}\left( \left(1+\tfrac{2h_0}{c_0}\right)^{r_kM}+\left(\tfrac{2}{dc_0}\log D(\alpha) \right)^{r_KM}\right)
\end{align}
and we examine how each individual term behaves with growing $d,t$. Viewed as a function in $d\geq 1$, we may estimate $\left(\tfrac{\log D(\alpha_j)}{d}\right)^d\leq e^{\tfrac{\log D(\alpha_j)}{e}}$ and therefore obtain
\begin{equation}\label{eq:Dcontribution}
	\left(\frac{ 4\log D(\alpha_j)}{dc_0}\right)^{r_KM}\leq (\tfrac{4}{c_0})^{r_kM}\cdot D(\alpha_j)^\frac{r_kM}{ed}.
\end{equation}
The upper bound on $\sum_{\substack{\beta\in(\calO_K^\times)^M\\D_\beta\in  A_m^{2,h_0}}} 
\frac{ \int_{K_\mathbb{R}^{m \times t }} f(xD_\beta) dx_1\cdots dx_m  }{ V(td)^m R^{mtd} 
 }$ therefore holds as claimed in the proposition provided that 
$$t_0\geq \frac{2r_K\cdot M }{d\log S}\cdot \max\left\{\log\left(\tfrac{4}{c_0}\right), \log\left(2+4\tfrac{h_0}{c_0}\right)\right\}.$$
Note that for $t\geq t_0$ we also have $3\cdot p_m(t,d)\leq (td\pi)^{m/2}$. This concludes our dealings with the $A_m^{2,h_0}$-type terms.\par 
\textbf{Type }$\mathbf{A_m^{2,\infty}}$ terms.
By Lemma \ref{lemma:grtoprojallcols}, the sum
\begin{align}
\Sigma_{n,m}^\infty & := \sum_{\substack{\beta\in(\calO_K^\times)^M\\ D_\beta\in  A_m^{2,\infty}}}(\tfrac{m+1}{m})^{mtd/2}\left(e^{2\cdot h_\infty(\sqrt{\tfrac{1}{n-m}}\alpha\beta)}+\tfrac{M}{(n-m)^2} \N(\alpha)^{2/(dM)}\cdot e^{\tfrac{-2}{M}\cdot h_\infty(\sqrt{\tfrac{1}{n-m}}\alpha\beta)}\right)^{-\frac{dt}{2}}\\
    &\geq \sum_{\substack{\beta\in(\calO_K^\times)^M\\ D_\beta\in  A_m^{2,\infty}}}\frac{ \int_{K_\mathbb{R}^{m \times t }} f(xD_\beta) dx_1\cdots dx_m  }{ V(td)^m R^{mtd} 
 }
\end{align}
provides an upper bound and it suffices to estimate $\Sigma_{n,m}^\infty$.
For $D_\beta \in\mathbf{A_m^{2,\infty}}$ we have by assumption that the heights are bounded below by 
\begin{equation}
  h_\infty(\alpha\beta)\geq \max_{1\leq i\leq M}h_\infty(\alpha_i\beta_i)\geq h_0.
  \label{eq:assumption_h_0}
\end{equation}

Abbreviating $f_m(x)= e^{2x}+\tfrac{M}{(n-m)^2} \N(\alpha)^{\frac{2}{dM}}e^{-2\frac{x}{M}}$,
we may rewrite  
$$
\Sigma_{n,m}^\infty  = \sum_{\substack{\beta\in(\calO_K^\times)^M\\ D_\beta\in  A_m^{2,\infty}}}(\tfrac{m+1}{m})^{mtd/2}
f_m\left(h_\infty(\sqrt{\tfrac{1}{n-m}}\alpha\beta)\right).
$$
Now observe that 
\begin{equation}
	\log(\sqrt{n-m})+h_{\infty}\left(\tfrac{1}{\sqrt{n-m}} \alpha \beta\right) = \tfrac{1}{[K(\sqrt{n-m}):\mathbb{Q}]}
	\cdot\sum_{\sigma:K(\sqrt{n-m})\to\mathbb{C}}\max_{1\leq j\leq M}\max(\log\sqrt{n-m},\log(\vert\sigma(\alpha_j\beta_j)\vert)) \geq h_{\infty}(\alpha \beta).
\end{equation}
Hence, we know that for any $B\ge 1$
\begin{equation}
	h_{\infty}\left(\tfrac{1}{\sqrt{n-m}} \alpha \beta\right) \leq B
	\Rightarrow  h_{\infty}(\alpha\beta)\leq B+\tfrac{1}{2}\log(n-m)
\end{equation}
and therefore we have the inclusion of sets
\begin{align}
	& \left\{\beta \in ( \OK^{\times} )^{M} \mid D_\beta \in A_m^{2, \infty}, h_{\infty}\left( \frac{1}{\sqrt{n-m}} \alpha \beta\right) \leq B \right\} \\
	 \subseteq & \left\{\beta \in ( \OK^{\times} )^{M} \mid  h_0 \leq \tfrac{1}{d} \log\left( H_{\infty}\left( \alpha \beta\right)\right) \leq B + \tfrac{1}{2} \log(n-m) \right\}.
\end{align}

We may therefore bound the sum $\Sigma_{n,m}^\infty$ by 

\begin{align}
    & ((1+\tfrac{1}{m})^m)^{td/2}\cdot \sum_{i=1}^\infty\card \{\beta\in(\mathcal{O}_K^\times)^M\mid h_0\leq\tfrac{1}{d}\log H_\infty(\alpha\beta)\leq h_0+i\}\cdot f_m(h_0+i-1-\tfrac{1}{2}\log(n-m))^{-\frac{dt}{2}}\\
    &\leq e^{td/2}\cdot \omega_K^M\cdot \sum_{i=1}^\infty \left(\tfrac{h_0+i+\log
   (D\left(\alpha\right)^{\frac{1}{d}}))
+\tfrac{c_0}{2}}{\tfrac{c_0}{2}}\right)^{r_KM}\cdot f_m(h_0+i-1-\tfrac{1}{2}
\log(n-m))^{-\frac{dt}{2}},
\end{align}

where for the second inequality we count units via Lemma \ref{lemma:unitcount} as before. The inequality
$$H_\infty\Big(\sqrt{\tfrac{1}{n-m}}\alpha\beta\Big) \geq \N\left(\sqrt{\tfrac{1}{n-m}}\alpha\right)^{\frac{1}{M}}$$ 
allows us to simply estimate
	$$f_m(x) \ge \N(\alpha)^{\frac{1}{dM}}\cdot e^x \text{ for } x= h_0+i-1 -\tfrac{1}{2}\log(n-m)\text{ and }i\in\mathbb{N}.$$
Using Jensen's inequality and bounding the contribution of $D(\alpha)$ to unit counts as in Equation \eqref{eq:Dcontribution}, we therefore obtain 
\begin{align}
    \Sigma_{n,m}^\infty\leq \omega_K^M\cdot \N(\alpha)^{\frac{-t}{2M}}\cdot \sum_{i=1}^\infty 
   \exp \left(-\tfrac{d(t-t_0)}{2}\cdot (h_0+i-2-
\tfrac{1}{2}\log(n-m))\right)
\end{align}
where we have chosen
$$t_0\geq \frac{2r_K\cdot M }{d}\cdot\sup_{i\in \mathbb{Z}_{\ge 1}}\frac{\max\left\{\log\left(2+\frac{4}{c_0}\left(  h_0+i \right)\right),\log(\tfrac{4}{c_0})\right\}}{h_0+i-2-\tfrac{1}{2}\log(n-m)}.$$
Observe that the term in $i=1$ attains the maximum. We may now make a choice of threshold height  $$h_0=2+\tfrac{1}{2}\log(n-m).$$
This yields the condition 
$$t_0\geq \frac{2r_K\cdot M }{d}\cdot\log\left(2+12c_0^{-1}+2\log
   (n-m)\cdot c_0^{-1}\right).$$
The result follows by bookkeeping of all the bounds obtained, noting that we may bound $\log(S)^{-1}\leq (m+1)\log(s)^{-1}$ for any $m\geq 2$. Similarly the explicit constants can be chased through the arguments.
\end{proof}
With this in hand, we are ready to tackle: 
\begin{proof}[Proof of Theorem \ref{thm:generalBogobound}]
First, note that it suffices to prove the statement fixing pivot columns and some number $M$ of nonzero entries in the last $n-m$ columns of $D$. We have $M\geq n-m$ and the contributions for matrices in $A_m^{2,0}$ when $M=n-m$ are dealt with in Proposition \ref{prop:projsumoverideals} (we apply it with $M=n-m$ and $k=4$). All of these contributions must be taken into account when expliciting the constants in the asymptotic. \par

Second, we claim that $$\mathfrak{D}(D)\geq D(\alpha)=\N\left( \langle1,\alpha_1,\ldots, \alpha_M\rangle\right)^{-1},$$ 
where $(\alpha_1,\ldots, \alpha_M)$ are the non-zero entries of $D$ in the non-pivot columns. Note that a sharper result can be obtained by taking all of the Pl\"ucker coordinates of $D$ into account, see Part 2. of Proposition \ref{pr:result_about_volumes}, but the claim suffices for our purposes. 
To prove the claim, observe that $\mathfrak{D}(D)$ is by definition the index as a sub-lattice of $\calO_K^m$ of the set of $(c_1,\ldots,c_m)\in \calO_K^m$ such that for each column $1\leq i\leq n$ we have $\sum_{j=1}^m c_jD_{ij}\in \calO_K$. This amounts to a linear condition modulo the integral ideal $I_i=(\alpha_{i1},\ldots,\alpha_{is})^{-1}$ where $i1,\ldots ,is$ are the indices of the subset of $(\alpha_1,\ldots, \alpha_M)$ in the $i$-th column of $D$. Considering multiple columns, $(c_1,\ldots,c_m)$ must lie in an intersection of hyperplanes modulo $J=\sum_{i=m+1}^{n-m}I_i\subset\calO_K$. For every prime $\mathfrak{p}\mid J$, we then get that by construction the $c_i$ satisfy at least one linear equation 
$$\sum_{j=1}^m c_jD_j\equiv 0 \mod \mathfrak{p}^{\ord_\mathfrak{p}(J)}$$
with at least one $c_j\neq  0 \mod \mathfrak{p}^{\ord_\mathfrak{p}(J)}$. In other words, this forces $(c_1,\ldots,c_m)$ into the pre-image of a hyperplane under the reduction map $\calO_K^m\to (\calO_K/\mathfrak{p}^{\ord_\mathfrak{p}(J)})^m$, which then has index $\mathfrak{p}^{\ord_\mathfrak{p}(J)}$. By the Chinese remainder theorem, we therefore get that $\mathfrak{D}(D)\geq \N(J)$. But we have that $D(\alpha)=\N(J)$ and the claim follows. \par 

By the claim and  Proposition \ref{prop:GrWeilunitsum}, the proof of the theorem thus reduces to establishing convergence of the sum  
 $$\sum_{\alpha\in (K^\times)^M\setminus(\calO_K^\times)^M} \max (N(\alpha)^\frac{-t}{4M}, N(\alpha)^\frac{-tm}{(m+1)M})\cdot D(\alpha)^{\tfrac{t}{2(m+1)}+\tfrac{r_kM}{ed}}\cdot D(\alpha)^{-t}.$$
 Summing over ideals as in Proposition \ref{prop:projsumoverideals}, it therefore suffices to bound 
 $$\zeta_K(t-\tfrac{t}{2(m+1)}-\tfrac{r_kM}{ed})^{-1}\cdot\sum_{\substack{I\subset \calO_K\\ I \text{ integral ideal}}}\N(I)^{-t+\tfrac{t}{2(m+1)}+\frac{r_kM}{ed}}\prod_{1\leq i\leq M}\sum_{J\subset \calO_K}\N(J)^{-\frac{t}{4M}}\N(I)^\frac{tm}{(m+1)M}.$$
 
 We see that this is bounded by $Z(K,t,n,m)$ as claimed, and we record the additional condition on $t$ to ensure convergence of the zeta values. The rest is keeping track of bounds on $t$ and explicit exponents for the various cases. 

\end{proof}
\subsection{General moments using the Bogomolov property}
Summarizing the results of this section, we obtain the following main theorem: 

\begin{theorem}\label{thm:maingeneralmomentsBogo}
Let $\mathcal{S}$ denote any set of number fields satisfying Hypothesis \ref{hyp:Lehmer} and let $c_0, c_1$ denote the resulting uniform constants. Fix a moment $n\geq 2$.
There exist constants $0<C_\mathcal{S},\varepsilon_\mathcal{S}<\infty$ uniform in $d,t$ such that the following holds: let $t_0$ denote
     $$ \sup_{K\in\mathcal{S}}\left(\frac{r_K n(n+1)^2}{d}\cdot\frac{\log(2+12c_0^{-1}+2\log
	     (n-1)\cdot c_0^{-1})}{\log \left( \min\{ \frac{64}{27},e^{\frac{1}{3}c_1},\cosh^3(c_1)\} \right)},
     \frac{2r_K(n-1)}{d}\cdot\frac{\log(17/8)}{\log\left(f_{n-1}(\tfrac{3}{4}c_0)\right)}\right),$$
where $f_m(x):=\frac{\exp(x)+m\exp(-\frac{x}{m})}{m+1}$ and $r_K$ is the rank of the unit group. 
We then have for any $t>t_0$ and for any $K\in \mathcal{S}$ of degree $d$ that the $n$-th moment $\mathbb{E}[\rho(\Lambda)^n]$ of the  number of nonzero $\calO_K$-lattice points in an origin-centered ball of volume $V$ in $K_\mathbb{R}^t$ satisfies: 
\begin{align}
       \omega_K^ne^{-V/\omega_K}\sum_{r=0}^\infty\frac{r^n}{r!}(\tfrac{V}{\omega_K})^r &\leq \mathbb{E}[\rho(\Lambda)^n]\\
       &\leq \omega_K^ne^{-V/\omega_K}\sum_{r=0}^\infty\frac{r^n}{r!}(\tfrac{V}{\omega_K})^r+C_\mathcal{S}\cdot \omega_K^{\tfrac{n^2}{4}}(td)^{\tfrac{n-2}{2}}\cdot e^{-\varepsilon_\mathcal{S}d(t-t_0)}\cdot (V+1)^{n-1}\cdot Z(K,t,n),\\
    \end{align}
    where $0\leq Z(K,t,n)=\frac{\zeta_K(\frac{1}{2(m+1)}t-\frac{1}{e}m(n-m))\cdot\zeta_K(\frac{t}{n^2})^{\frac{1}{4}n^2}}{\zeta_K(t-1)}$. Moreover, it suffices to take 
    $$\varepsilon_\mathcal{S}= \tfrac{1}{2}\log(\min\{ \tfrac{4}{3}, e^{\frac{c_1}{3n+2}}, f_{n-1}(\tfrac{3}{4}c_1)\}).$$
    The constant $C_\mathcal{S}$ may as well be chased down explicitly in terms of $n,\varepsilon_\mathcal{S}$.
\end{theorem}
\begin{proof}
    This follows from our previous results, namely the terms with $m=1$ are dealt with in Theorem \ref{thm:mainballintersection} (we simply put $k=4$). The error terms in $A_m^2$ for $m\geq 2$ are bounded via Theorem \ref{thm:generalBogobound}, keeping the values of $2\leq m\leq n$ that give the worst bound on $t_0$. The zeta factors from $A_m^2$ are the larger ones. The contributions of terms in $A_m^1$ for $m\geq 2$ decay exponentially by Theorem \ref{th:poisson_dist} and they are thus easily handled error terms. 
    Finally, the main term contributions for $2\leq m\leq n-1$ are computed in Lemma \ref{le:poisson_term}. The explicit exponent $\varepsilon_\mathcal{S}$ is found by taking the smallest over all the different terms and the constant $C_\mathcal{S}$ can be chased down similarly as an enumeration of cases as well as geometric sums bounded in terms of $\varepsilon_\mathcal{S}$ and several counts, such as Stirling numbers, which depend only on the moment.
\end{proof} 
A few comments on Theorem \ref{thm:maingeneralmomentsBogo} are in order. First, the bound on $t$ is $t_0=O(n^3\log\log n)$ as $n$ increases with an implicit constant only depending on the number fields. For specific setups, especially for small moments where the contributions are covered in Theorem \ref{thm:mainballintersection}, the bound as well as the zeta factor can be sharpened slightly. Similarly, the explicit exponent may be optimised; Theorem \ref{thm:maingeneralmomentsBogo} emphasizes a general result for reasonable and explicit bounds, and we make no claim as to optimality of these. Recall also that $\omega_K=O(d\log\log d)$ so that the theorem indeed exhibits exponential decay in $d,t$ of the non-Poisson terms provided the zeta factors do not grow exponentially in $d$. \par
Second, one may trivially take $\mathcal{S}$ to be a constant number field $K$. Hypothesis \ref{hyp:Lehmer} is then satisfied and we obtain convergence of the moments of the number of $\omega_K$-tuples of lattice points inside a ball of volume $V$ towards the moments of a Poisson distribution of mean $V/\omega_K$ for any number field $K$ and large enough number of copies $t$. We record a version of this statement: 
\begin{corollary}\label{cor:PoissonforKfixed}
    Let $K$ be any number field of fixed degree $d$. Let $c_0, c_1$ denote the constants bounding the Weil height on $\OK$ and $K$ as in \ref{hyp:Lehmer} and fix a moment $n\geq 2$.
    Let $t_0$ be as in Theorem \ref{thm:maingeneralmomentsBogo}.
    We then have for any $t>t_0$ that the $n$-th moment $\mathbb{E}[\rho(\Lambda)^n]$ of the  number of nonzero $\calO_K$-lattice points in an origin-centered ball of volume $V$ in $K_\mathbb{R}^t$ satisfies: 
\begin{align}
       \omega_K^ne^{-V/\omega_K}\sum_{r=0}^\infty\frac{r^n}{r!}(V/\omega_K)^r &\leq \mathbb{E}[\rho(\Lambda)^n]\\
       &\leq \omega_K^ne^{-V/\omega_K}\sum_{r=0}^\infty\frac{r^n}{r!}(V/\omega_K)^r+C_K\cdot t^{(n-2)/2}\cdot e^{-\varepsilon_K(t-t_0)}\cdot (V+1)^{n-1},\\
    \end{align}
    for constants $C_K,\varepsilon_K>0$ uniform in $t$. Moreover, we may take 
    $$\varepsilon_K= \tfrac{1}{2}\log(\min\{ \tfrac{4}{3}, e^{\frac{c_1}{3n+2}}, f_{n-1}(\tfrac{3}{4}c_1)\}).$$
    The constant $C_K$ may as well be chased down explicitly in terms of $n,\varepsilon_K,\omega_K$ and Dedekind zeta values of $K$.
\end{corollary}

Third, it is not entirely trivial that for appropriately large fixed $t,k$ the error term in Theorem \ref{thm:maingeneralmomentsBogo} decays exponentially in $d$ due to the dependence on $K$ of the zeta factor error terms $Z(K,t,n,k)$ in Theorem \ref{thm:maingeneralmomentsBogo}, which a priori could grow exponentially in $d$. Proving bounds in $d$ for the growth does not appear trivial for general number fields. For instance, using lattice-point estimate based methods such as the Dedekind-Weber theorem to count ideals of bounded norm does not appear like a promising approach due to the fact that the best known bounds on the error term for counts of ideals of bounded norm grows exponentially in $d$ (see, e.g., \cite[Corollaire 1.3.]{Angeheights}). Nevertheless, for specific towers of number fields one should be able to prove the desired boundedness (or at least subexponential growth in $d$) for Dedekind zeta values. For instance we have: 
\begin{lemma}\label{lemma:cyclotomiczetabound}
Let $K=\mathbb{Q}(\zeta_n)$ be a cyclotomic field of degree $d=\varphi(n)$. Let $s>1$ be a real number. Then we have that
$$\zeta_K(s)\leq C(s)$$
for some constants $C(s)>0$ uniform in $d$. 
\end{lemma}
\begin{proof}
    We first claim that the Dedekind zeta function of cyclotomic fields $\mathbb{Q}(\zeta_n)$ may be written as
    $$\zeta_K(s)=\prod_{p\in\mathbb{P}}\frac{1}{\left(1-\frac{1}{p^{s\cdot \ord_{n_p}p}}\right)^{\frac{\varphi(n_p)}{\ord_{n_p}p}}},$$
    where $n_p=n\cdot p^{-v_p(n)}$ denotes the prime-to-$p$ part of $n$.\par
    The claim follows from examining for each Euler factor the splitting behaviour of primes above $p$ based on the factorization of the cyclotomic polynomial $\Phi_n(x)$ modulo $p$. For instance, if $p\nmid n$, the number of roots of a factor of $\Phi_n(x)$ modulo $p$ coincides with the size of the orbit of Frobenius acting via multiplication-by-$p$ on $(\mathbb{Z}/n\mathbb{Z})^\times$, and hence the result follows in this case. When $p\mid n$, the same applies to the subextension $\mathbb{Q}(\zeta_{n_p})$ unramified at $p$, and then the remaining extension $K/\mathbb{Q}(\zeta_{n_p})$ is totally ramified at $p$, and the claim follows. \par
    Using the claim, we have the following argument due to Danylo Radchenko: 
    write $\zeta_{K}(s) = T_1 T_2$ where 
\begin{equation}
  T_1 = \prod_{p \mid n} \frac{1}{\left( 1- \frac{1}{p^{s \ord_{n_{p}} p }}\right) ^\frac{\varphi(n_{p}) }{\ord_{  n_{p} } p }}.
\end{equation}
and 
\begin{equation}
	T_2 = \prod_{p \nmid n} \frac{1}{\left( 1- \frac{1}{p^{s \ord_{n} p }}\right) ^\frac{\varphi(n) }{ \ord_{n}p }}  . 
\end{equation}

For $T_2$, we have that 
\begin{equation}
\log T_2 \ll  \sum_{p \nmid n } \frac{\varphi(n) }{ ( \ord_{n} p  ) p^{s \ord_{n} p }}.
\end{equation}
We write $\varphi(n) \le n , \ord_{n} p \ge 1$ and observe that the set $\{p^{\ord_{n} p}\}_{p \in \mathbb{P}}$ lies in $\{n+1,2n+1,\dots\}$. Since they are simply different prime powers, there are no repetitions.

So we have 
\begin{equation}
  \log T_2 \ll n \left(\frac{1}{(n+1)^{s}} + \frac{1}{(2n+1)^{s}}+ \cdots \right) \ll \frac{1}{n^{s-1}} \ll 1.
\end{equation}

Now for $T_1$,
\begin{equation}
\log T_1 \ll  \sum_{p \mid n } \frac{\varphi(n_{p}) }{ ( \ord_{n_p} p  ) p^{s \ord_{n_{p}} p }},
\end{equation}
we again write 
$\varphi(n_{p}) \le n_{p} , \ord_{n_p} p \ge 1$ and use that $p^{\ord_{n_{p}} p} \ge n_{p}$ so this gives us 
\begin{equation}
\log T_1 \ll  \sum_{p \mid n } \frac{1}{n_{p}^{s-1}} = \frac{1}{n^{s-1}}\sum_{p \mid n} p^{\nu_{p}(n) \cdot (s-1)}
\end{equation}
Let $k$ be the number of primes in $n$. Then the largest prime factor of $n$ can be at most $\tfrac{n}{p_{1}p_{2} \cdots p_{k-1}}^{} \le \frac{n}{(k-1)!}$. So we write
\begin{equation}
	\log T_1 \ll  \frac{k}{n^{s-1}} \left( \frac{n}{(k-1)!}\right)^{s-1} .
\end{equation}
This tends to $0$ as $k \rightarrow \infty$ so it must be bounded.
\end{proof}

Finally, we make Theorem \ref{thm:maingeneralmomentsBogo} more explicit for towers of cyclotomic fields: 

\begin{corollary}\label{cor:cyclomoments}
    Consider a sequence of cyclotomic number fields given by $K_i=\mathbb{Q}(\zeta_{k_i})$ of degree $d_i=\varphi(k_i)$ and let $n\geq 2$. Moreover let 
    $$t_0=\max\left\{ 19 n(n+1)^2 \log(52+\tfrac{25}{3}\log(n-1)), \frac{(n-1)\log(\tfrac{17}{8})}{\log(f_{n-1}(\frac{9}{50}))}\right\}.$$
    where $f_{n-1}(x):=\frac{\exp(x)+(n-1)\exp(-\frac{x}{n-1})}{n}$. 
There exists constants $C_n,\varepsilon_n> 0$ uniform in $d_i,t$ such that for any $t>t_0$ and any degree $d_i$ the $n$-th moment $\mathbb{E}[\rho(\Lambda)^n]$ of the  number of nonzero $\calO_K$-lattice points in an origin-centered ball of volume $V$ in $K_\mathbb{R}^t$ satisfies 

\begin{align}
       \omega_{K_i}^ne^{-V/\omega_{K_i}}\sum_{r=0}^\infty\frac{r^n}{r!}(\tfrac{V}{\omega_{K_i}})^r &\leq \mathbb{E}[\rho(\Lambda)^n]\\
											  &\leq \omega_{K_i}^ne^{-V/\omega_{K_i}}\sum_{r=0}^\infty\frac{r^n}{r!}(\tfrac{V}{\omega_{K_i}})^r+C_n\cdot \omega_{K_i}^{\tfrac{n^2}{4}}(td_i)^{\tfrac{n-2}{2}}\cdot e^{-\varepsilon_n\cdot d_i(t-t_0)}\cdot (V+1)^{n-1},\\
    \end{align}
    ergo the moments of the number of $\omega_{K_i}$-tuples of nonzero lattice points approach the moments of a Poisson distribution of mean $V/\omega_{K_i}$ as $d_it\to \infty$. Moreover, we may take 
    $$\varepsilon_n= \tfrac{1}{2}\log(\min\{5^{\frac{1}{36n+24}}, f_{n-1}(\tfrac{\log 5}{16})\}).$$
    The constant $C_n$ may as well be chased down explicitly in terms of $n,\varepsilon_K$ and Dedekind zeta values of $K$.
\end{corollary}
\begin{proof}
    The result follows from Theorem \ref{thm:maingeneralmomentsBogo} using Corollary \ref{cor:cyclohypothesis} and the resulting constants. With these choices, some of the conditions on $t_0$ in Theorem \ref{thm:maingeneralmomentsBogo} simplify. We bound the zeta factors in Theorem \ref{thm:maingeneralmomentsBogo} uniformly in $\varphi(k_i)$ by Lemma \ref{lemma:cyclotomiczetabound}. 
\end{proof}
Again we note that for low moments such as $n=2,3$ better bounds can be achieved, in particular because we can reduce to contributions from projective space heights and Theorem \ref{thm:mainballintersection}. Moreover, as $n$ becomes large, $\frac{(n-1)\log(17/8)}{\log(f_{n-1}(9/50))}\le 45 n^2$ and it suffices to take $t_0\geq  19 n(n+1)^2 \log(52+\tfrac{25}{3}\log(n-1))$ in Corollary \ref{cor:cyclomoments}. 
\section{Odds and ends}
This section is devoted to a couple results and remarks complementing the results established in Section \ref{se:Bogosection}. 
\subsection{No limiting Poisson moments}
Section \ref{se:Bogosection} establishes that under some assumptions on height lower bounds, the $n$-th moments approach the moments of a Poisson distribution even when varying the number fields for a fixed large enough number of copies. We now show that in order to obtain such a behaviour, some assumption on the heights is necessary by exhibiting sequences of number fields $K$ for which moments do not converge to Poisson moments of the expected parameters $V/\omega_K$.

\begin{lemma}\label{le:intlowerbound}
    Let $B$ denote the unit ball in $K_\mathbb{R}^t$ and for $\alpha\in K^\times$ denote by $H_W(\alpha)$ its Mahler measure. Then 
    $$D(\alpha)^{-t}\cdot \sum_{\alpha\in\calO_K^\times}\frac{\vol(B\cap\alpha^{-1}B)}{\vol (B)}\geq \sum_{\alpha\in K^\times}H_W(\alpha)^{-t}.$$
\end{lemma}
\begin{proof}
    We may without loss of generality assume $\alpha\in \OK$. In each of the $t$ copies of $K_\mathbb{R}$ and for each nonzero $\alpha\in\calO_K$, the origin-centered ellipsoid of lengths 
    $$\min(1,\vert\sigma_1(\alpha)\vert),\ldots,\min(1,\vert\sigma_d(\alpha)\vert),$$
    where $\sigma_1,\ldots,\sigma_d$ are the embeddings $K\to\mathbb{C}$, is contained inside of $B\cap\alpha^{-1}B$.
\end{proof}
Note that this provides a lower bound for the second moment in termst of the height zeta function of $\mathbb{P}^1(K)$. We deduce: 
\begin{proposition}\label{prop:morethanPoisson}
    Let $f_n$ be a sequence of irreducible polynomials of degree $n$ such that their Mahler measures are uniformly bounded $\forall n$. Let $K_n:=\mathbb{Q}(f_n)$ be the resulting sequence of number fields. Then for any fixed number of copies $t>2$ there exists $C_t>0$ such that over any number field $K_n$ the second moment satisfies
    $$  \mathbb{E}[\rho(\Lambda)^2]\geq C_t+\omega_{K_n}^2e^{-V/\omega_{K_n}}\sum_{r=0}^\infty\frac{r^2}{r!}(V/\omega_{K_n})^r.$$
\end{proposition}
\begin{proof}
    From the integral formula, it suffices to show that $\sum_{\alpha\in\calO_K^\times}\frac{\vol(B\cap\alpha^{-1}B)}{\vol (B)}\geq C_t+\omega_{K_n}$
    for some $C_t>0$ not depending on $n$. But this is clear from Lemma \ref{le:intlowerbound} and our assumptions on heights in $K_n$. 
\end{proof}
Needless to say that similar results hold for higher moments as well. We also note that there are many sequences satisfying the assumptions of Proposition \ref{prop:morethanPoisson}. For example, one has limiting results for Mahler measures such as (see \cite{Boydlimit}):   
$$\lim_{n\to \infty}H_\infty(\alpha_n)=1.3815\ldots  \text{, where }\alpha_n^n-\alpha_n+1=0,$$
so where $\alpha_n$ is a root of $f_n(x)=x^n-x+1$. 

\subsection{More general bodies}
\label{se:oddsends}
Although the bounds are more easily derived for indicator functions of balls, we can use spherical symmetrization to obtain results for more general bodies. The quantities appearing in the integral formula will be largest in the spherical case, so that the upper bounds on moments are valid more generally. The same methods as in Rogers' work \cite[Theorem 1,2]{R1956} carry through, so we simply restate: 
\begin{theorem}\label{th:genbody}
Let $g$ be a non-negative compactly supported Riemann integrable function on $K_{\mathbb{R}}^{t\times n}$  and let $g^*$ denote the function obtained by spherical symmetrization. Let $g(\Lambda)$ and $g^*(\Lambda)$ denote the corresponding lattice sum functions over non-trivial lattice points. Then the moments over the space of unimodular $\calO_K$- lattices or over the smaller sets satisfying mean value formulas as in Theorem \ref{th:weil} satisfy: 
\begin{align}
    \mathbb{E}[g(\Lambda)]&=  \mathbb{E}[g^*(\Lambda)]\\
    \mathbb{E}[g(\Lambda)^2]&\leq  \mathbb{E}[g^*(\Lambda)^2]\\
    \mathbb{E}[g(\Lambda)^3]&\leq  \mathbb{E}[g^*(\Lambda)^3]
\end{align}
and moreover if for each constant $c>0$ the set of points with $g(x)>c$ is convex we have that 
$$\mathbb{E}[g(\Lambda)^k]\leq  \mathbb{E}[g^*(\Lambda)^k]$$
   for all $k\geq 4$ as well. 
\end{theorem}
\begin{proof}
    Given the integral formula as in Theorems \ref{th:weil}, this reduces to integral inequalities in Euclidean space, and thus follows from the inequalities in \cite{Rogerstwoineq} in the same way as \cite[Theorems 1,2]{R1956}.
\end{proof}

This yields the following version of the main theorem.

\begin{theorem}\label{thm:genconvex}
     Let $\mathcal{S}$ denote any set of number fields satisfying Hypothesis \ref{hyp:Lehmer} and let $c_0, c_1$ denote the resulting uniform constants. 
     Fix a moment $n\geq 2$. 
     Let $g$ be the characteristic function of a bounded, convex set $S$ in $K_\mathbb{R}^t$ of volume $V$, with the origin as centre and assume that $S$ is fixed by the coordinate-wise action of a cyclic group $\mu_{N_K}\subseteq\mu_K$ of order $N_K$. There exist explicit constants $C_\mathcal{S},\varepsilon_\mathcal{S}>0$ uniform in $d,t$ such that the following holds: for $t_0(\mathcal{S},n)$ as defined in Theorem \ref{thm:maingeneralmomentsBogo} and for all $t > t_0(\mathcal{S},n)$, we have 
\begin{align}
       N_K^n\cdot  m_{n}(\tfrac{1}{N_K} V) &\leq \mathbb{E}[g(\Lambda)^n]
       \leq \omega_K^n \cdot m_n(\tfrac{1}{\omega_K} V)  +C_\mathcal{S}\cdot \omega_K^{\tfrac{n^2}{4}}(td)^{\tfrac{n-2}{2}}\cdot e^{-\varepsilon_\mathcal{S}d(t-t_0)}\cdot (V+1)^{n-1}\cdot Z(K,t,n).\\
    \end{align}

    Here $m_n$ is as defined in (\ref{eq:def_of_poisson}) and $Z(K,t,n),C_\mathcal{S},\varepsilon_\mathcal{S}$ are as in Theorem \ref{thm:maingeneralmomentsBogo}.
   
\end{theorem}
\begin{proof}
The upper bound follows from Theorem \ref{th:genbody} and from our results for the spherical case in Theorem \ref{thm:maingeneralmomentsBogo}. The lower bound follows by symmetry under $\mu_{N_K}$ in the same way as Lemma \ref{le:poisson_term} establishes the lower bound when $S$ is a ball and invariant under the whole $\mu_K$-action.    
\end{proof}
\begin{remark}
    The lower bound in Theorem \ref{thm:genconvex} can be tightened for general bounded convex bodies. For instance, for any chain of subgroups $\{\pm 1\}= G_1<\cdots <G_k=\mu_K$ and bounded convex set $S$, we may stratify $S=\sqcup_{i=1}^k S_i$ by setting 
    $$S_k= \bigcap_{g\in G_k}gS\text{ and }S_{i-1}=\bigcap_{g\in G_{i-1}}gS\setminus S_i$$
    for $1\leq i\leq k$. The lower bound can then be improved (in the notations of Theorem  \ref{thm:genconvex}) to 
    $$ \sum_{i=1}^k  (\card G_i)^n\cdot m_{n}(\tfrac{ \vol(S_i)}{\card G_i}) \leq \mathbb{E}[g(\Lambda)^n]$$
    by applying Theorem \ref{thm:genconvex} to the $G_i$-symmetrized and measurable sets $S_k\sqcup \cdots \sqcup S_i$ and inclusion-exclusion. 
\end{remark}

\newpage
\begin{appendix}
	\section{Proof of Lemma \ref{le:equivalence_of_height}}
  \label{ap:three_page_proof}

It is clear what $\det(D)\mathfrak{D}(D)$ on the right-hand side is measuring. Indeed, first note that 
\begin{equation}
\mathfrak{D}(D) = \card \frac{\mathcal{O}_K^{m}}{\{ v \in \mathcal{O}_K^{m} \mid D^{T}v \in \mathcal{O}_K^{n} \}} = \card
\frac{D^{T} \mathcal{O}_K^{m}}{ D^{T} \mathcal{O}_K^{m} \cap \mathcal{O}_K^{n} } 
\end{equation}
So we have that 
\begin{equation}
\mathfrak{D}(D) \det (D )=
\card \frac{D^{T} \mathcal{O}_K^{m}}{ D^{T} \mathcal{O}_K^{m} \cap \mathcal{O}_K^{n} } 
\det(D; M_{1 \times m }(\mathcal{O}_K))
\end{equation}
is simply the volume of a parallelepiped spanning a $\mathbb{Z}$-basis of the lattice $ \Lambda = D^{T} \mathcal{O}_K^{m} \cap \mathcal{O}_K^{n} = D^{T} K^{n} \cap \mathcal{O}_K^{m}$ (this equality holds since $D^{T}$ is column reduced and $m \le n$). The following tells us that the height $H(S)$ is also calculating this volume. 

\begin{proposition}
\label{pr:result_about_volumes}

Suppose that $w_1, \dots,w_m \in K^{n}$ are a set of $K$-linearly independent vectors spanning a subspace $S$. 
Let $\Lambda = \mathcal{O}_K^{n}\cap K$.
If we consider the lattice $ \Lambda' = \OK w_1 + \OK w_2 + \dots + \OK w_m$, 
then $[\Lambda:\Lambda' \cap \Lambda]<\infty$ and $[\Lambda':\Lambda' \cap \Lambda]<\infty$. So we have that 
\begin{equation}
	[\Lambda : \Lambda' \cap \Lambda] \det \Lambda = [\Lambda': \Lambda' \cap \Lambda]  {\det \Lambda'} .
\end{equation}

We claim: 
\begin{enumerate}
\item
	Set $N = \binom{n}{m}$ and $[x_1,\dots,x_{N}] = \iota(S)$ (as defined in \S \ref{ss:heights}).
Then
\begin{equation}
\det \Lambda' = \prod_{i=1}^{[K:\mathbb{Q}]} \sqrt{  \sum_{ j=1}^{N} |\sigma_i(x_{j})|^{2} }.
\end{equation}
\item
\begin{equation}
	\frac{[\Lambda: \Lambda' \cap \Lambda]}{[\Lambda': \Lambda' \cap \Lambda]} =  \N(\langle x_1,\dots,x_N \rangle ).
\label{eq:claim2}
\end{equation}
Here the right hand side denotes the norm of the fractional ideal $\OK x_1 + \dots + \OK x_N$.

\end{enumerate}
\end{proposition}

{\bf Proof of Claim 1:}

Let us evaluate $\det \Lambda'$. 
Fix a $\mathbb{Z}$-basis $a_1,a_2,\dots,a_r$ of $\OK$, where $r=[K:\mathbb{Q}]$. 
A $\mathbb{Z}$-basis of $\Lambda'$ is given by $a_1w_1,$$a_2w_1,\dots$, $a_rw_1, a_1w_2,a_2w_2,\dots$, $a_rw_2, \dots$, $a_1w_m,a_2w_m,\dots$, $a_rw_m$. We will calculate the volume of the parallelepiped spanned by these vectors with respect to the quadratic form in Equation \ref{eq:norm}. Observe that for $x,y \in K_\mathbb{R}$
\begin{equation}
\Tr(x\overline{y}) = \sum_{i=1}^{r} \sigma_{i}(x) \overline{\sigma_i(y) }.
\end{equation}

Let $w_i= (w_{i1},w_{i2},\dots,w_{in}) \in K^{n}$.
If we define the $rn \times rm$ matrix
\begin{align}
& A =  \begin{bmatrix}
\sigma_1(a_1w_{11})  & \dots & \sigma_1(a_rw_{11}) &  \dots &  \sigma_1(a_1 w_{m1} )& \dots &\sigma_1(a_rw_{m1})  \\
& \vdots & &   \vdots & &\vdots &  \\
\sigma_{r}(a_1w_{11})  & \dots & \sigma_{r}(a_rw_{11}) &  \dots &  \sigma_{r}(a_1 w_{m1} )& \dots &\sigma_{r}(a_rw_{m1})  \\
& \vdots & &   \vdots & &\vdots &  \\
\sigma_1(a_1w_{1n})  & \dots & \sigma_1(a_rw_{1n}) &  \dots &  \sigma_1(a_1 w_{mn} )& \dots &\sigma_1(a_rw_{mn})  \\
& \vdots & &   \vdots & &\vdots &  \\
\sigma_{r}(a_1w_{1n})  & \dots & \sigma_{r}(a_rw_{1n}) &  \dots &  \sigma_{r}(a_1 w_{mn} )& \dots &\sigma_{r}(a_rw_{mn})  \\
\end{bmatrix}
\end{align}
then it follows that 
\begin{equation}
\left( \det \Lambda'\right)^{2} = \Delta_{K}^{-2m}{ \det A^{*} A }.
\end{equation}

We can expand the right hand side using the Cauchy-Binet theorem, obtaining that 
\begin{equation}
\det A^{*} A  = \sum_{\substack{I \subseteq \{1\dots rn\} \\ \card I= rm}} |\det(A_{I})|^2, 
\label{eq:cauchy-binet}
\end{equation}
where $A_{I}$ is the $rm \times rm$ minor of $A$ with rows in $I$. 

Each row of $A$ is a complex embedding of the vector 
$$(a_i w_{jk})_{\substack{ 1 \le i \le r\\ 1 \le j \le m }}\text{ for some }k \in \{1,\dots,n\}.$$
We claim that the only $I$ for which $\det(A_I)^{2}$ could be non-zero are the ones where 
each embedding $\sigma_{i}$ appears exactly $m$ times applied to various $m$-subsets of these $n$ possible rows.
That is, $I \subseteq \{1\dots rn\}$ should be of the form
\begin{equation}
I = \bigsqcup_{ i\in 1}^{m} \bigcup_{k \in J_{i}} \{  kr-r+i\}, \ \text{ for some $J_1,J_2,\dots,J_r \subseteq \{1,\dots,n\}$, $\card J_1 = \card J_2 = \dots = \card J_r = m$}.
\label{eq:form_of_I}
\end{equation}

To observe this, note that if $m' > m$ then the following row-vectors are $K$-linearly dependent.
\begin{align}
	( w_{1k_1},\  w_{2k_1},\  \cdots\  & , w_{mk_1}\ ),\\
	(w_{1k_2}, \  w_{2k_2}, \   \cdots\  & , w_{mk_2}\ ), \\
	 \vdots\ \  &  \\
	(w_{1k_{m'}}, w_{2k_{m'}}, \cdots & , w_{mk_{m'}}).\\
\label{eq:ld_rows}
\end{align}

This implies that the following rows are $\mathbb{Q}$-linearly dependent
\begin{align}
	( a_1w_{1k_1} \ ,\dots ,a_rw_{1k_1}  ,\ &  \cdots, a_1w_{mk_1},\dots,a_r w_{mk_1}),\\
	(a_1w_{1k_2} \ ,\dots, a_r w_{1k_2} ,\  &  \cdots, a_1w_{mk_2},\dots,a_r w_{mk_2}),\\
	   & \ \ \vdots  \\
	(a_1w_{1k_{m'}},\dots,a_r w_{1k_{m'}},&   \cdots,a_1 w_{mk_{m'}} \dots,a_r w_{mk_{m'}}),\\
\label{eq:ld_rows_open}
\end{align}
and therefore if we apply invertible $\mathbb{Q}$-linear maps in each coordinate of these rows then they remain $\mathbb{Q}$-linearly dependent. 
Hence, each $\sigma_{i}$ will appear in no more than $m$ rows associated to each and as a result 
$I$ can only be of the form in Equation (\ref{eq:form_of_I}).

Suppose therefore that $I$ is of the form in Equation (\ref{eq:form_of_I}) where the $J_i = \{k_{i1},k_{i2},\dots,k_{im}\} \subseteq \{1,\dots,n\}$.
Then up to permutation of rows, the matrix $A_I$ is given by
\begin{align}
A_I = \begin{bmatrix}
\sigma_{1}(a_1)\sigma_1(w_{1k_{11}}) & \dots  & \sigma_{1}(a_r)\sigma_1(w_{1k_{11}}) &  \dots & \sigma_{1}( a_1)\sigma_1(w_{mk_{11}}) & \dots & \sigma_{1}(a_r)\sigma_1(w_{mk_{11}}) \\
\sigma_{1}(a_1)\sigma_1(w_{1k_{12}}) & \dots &  \sigma_{1}(a_r)\sigma_1(w_{1k_{12}}) &  \dots &  \sigma_{1}(a_1)\sigma_1(w_{mk_{12}}) & \dots & \sigma_1(a_r)\sigma_1(w_{mk_{12}})  \\
& & & \vdots &  & & \\
\sigma_{1}(a_1)\sigma_1(w_{1k_{1m}}) & \dots & \sigma_{1}(a_r)\sigma_1(w_{1k_{1m}} ) &  \dots &  \sigma_{1}(a_1)\sigma_1(w_{mk_{1m}}) & \dots & \sigma_{1}(a_r)\sigma_1(w_{mk_{1m}})\\
& & & \vdots &  & & \\
\sigma_{r}(a_1)\sigma_r(w_{1k_{r1}}) & \dots  & \sigma_{r}(a_r)\sigma_r(w_{1k_{r1}}) &  \dots & \sigma_{r}( a_r)\sigma_r(w_{mk_{r1}}) & \dots & \sigma_{r}(a_r)\sigma_r(w_{mk_{r1}}) \\
\sigma_{r}(a_1)\sigma_r(w_{1k_{r2}}) & \dots &  \sigma_{r}(a_r)\sigma_r(w_{1k_{r2}}) &  \dots &  \sigma_{r}(a_r)\sigma_r(w_{mk_{r2}}) & \dots & \sigma_r(a_r)\sigma_r(w_{mk_{r2}})  \\
& & & \vdots &  & & \\
\sigma_{r}(a_1)\sigma_r(w_{1k_{rm}}) & \dots & \sigma_{r}(a_r)\sigma_r(w_{1k_{rm}} ) &  \dots &  \sigma_{r}(a_r)\sigma_r( w_{mk_{rm}}) & \dots & \sigma_{r}(a_r)\sigma_r(w_{mk_{rm}})\\
\end{bmatrix}.
\label{eq:ld_cols}
\end{align}

Upon inspection, one can conclude that actually $A_I = W B$ where $W$ and $B$ are $rm \times rm$ matrices given by
\begin{align}
W = &
\begin{bmatrix}
\sigma_1(w_{1k_{11}}) & \dots & \sigma_1(w_{mk_{11}})       \\
& \vdots &  				\\
\sigma_1(w_{1k_{1m}})	& \dots & \sigma_{1}(w_{mk_{1m}}) \\
& & & \ddots \\
& & & &	   \sigma_r(w_{1k_{r1}}) & \dots & \sigma_r(w_{mk_{r1}})       \\
& & & &				 & \vdots &  				\\
& & & &	   \sigma_r(w_{1k_{rm}})	& \dots & \sigma_{r}(w_{mk_{rm}}) \\
\end{bmatrix},\\
B = & 
\begin{bmatrix}
\sigma_1(a_1) & \dots  & \sigma_1(a_r) \\
& & & \sigma_1(a_1) & \dots & \sigma_1(a_r) \\
& & & & & & \dots & & & \\\
& & & & & & & \sigma_1(a_1) & \dots & \sigma_1(a_r) \\
\sigma_2(a_1) & \dots  & \sigma_2(a_r) \\
& & & \sigma_2(a_1) & \dots & \sigma_2(a_r) \\
& & & & & & \dots & & & \\\
& & & & & & & \sigma_2(a_1) & \dots & \sigma_2(a_r) \\
& & & & & \vdots & & & \\
\sigma_r(a_1) & \dots  & \sigma_r(a_r) \\
& & & \sigma_r(a_1) & \dots & \sigma_r(a_r) \\
& & & & & & \dots & & & \\\
& & & & & & & \sigma_r(a_1) & \dots & \sigma_r(a_r) \\
\end{bmatrix}.
\label{eq:wandb}
\end{align}
It then follows that $|\det B |= \Delta_{K}^{m}$ and thus as $J_1,J_2,\dots,J_r$ go through all the possible $m$-subsets of $\{1,\dots,n\}$
\begin{equation}
\sum_{\substack{I \subseteq \{1,\dots,rn\} \\ \card I  = m}} |\det A_I|^2 = \Delta_{K}^{2m} \prod_{l=1}^{r}\left( \sum_{ \{k_1,\dots, k_{m}\} \subseteq \{1,\dots,n\}} \left|\det_{ 1 \le i,j \le m} \left[ \sigma_{l}(w_{i k_{j}} )\right] \right|^{2} \right).
\end{equation}
This settles the claim.

{\bf Proof of Claim 2:}

We are given $\{(w_{i1},w_{i2},\dots,w_{in})\}_{i=1}^{m} \in K^{n}$. Define $W \in M_{n \times m }(\OK)$ as 
\begin{equation}
W= 
\begin{bmatrix}
w_{11} & w_{21} & \dots & w_{m1} \\
w_{21} & w_{22} & \dots & w_{m2} \\
& & \vdots \\
w_{m1} & w_{m2} & \dots & w_{mn} \\
\end{bmatrix}.
\label{eq:brother}
\end{equation}
Then $\Lambda' = W \mathcal{O}_K^{m}$ and $\Lambda = WK^{m}\cap \mathcal{O}_K^{n}$. 

To prove this claim, it is sufficient to prove it for the case when $\{w_{ij}\} \subseteq \OK$. Indeed, let us multiply $W$ by an integer $\kappa \in \OK$ that can cancel all the denominators (i.e. $\kappa \cdot w_{ij} \in \OK$). Then note that $\kappa \Lambda' \subseteq \Lambda' \cap \Lambda$ and 
\begin{equation}
	[\Lambda' : \kappa \Lambda' ] = \N(\kappa)^{m},
\end{equation}
so we have 
\begin{equation}
	\frac{	[\Lambda :  \Lambda \cap \Lambda']}{[\Lambda' : \Lambda \cap \Lambda']} = 
	\frac{	[\Lambda :  \Lambda \cap \Lambda'][ \Lambda' \cap \Lambda : \kappa \Lambda' ]}{[\Lambda' : \Lambda \cap \Lambda'] [ \Lambda' \cap \Lambda : \kappa \Lambda' ]}
	= \frac{	[\Lambda :  \kappa \Lambda' ]}{[\Lambda' : \kappa \Lambda' ]} = [\Lambda: \kappa \Lambda' ] \N(\kappa)^{-m} .
\end{equation}
This establishing the identity $[\Lambda: \kappa \Lambda'] = \N(\langle \kappa^{m} x_1,\dots,\kappa^{m} x_N \rangle)$ would finish the proof.

Therefore, let us now assume without loss of generality that we have $\{\omega_{ij}\} \subseteq \OK$ and hence $\Lambda' \subseteq \Lambda$. We want to show that 
\begin{equation}
	\N(\langle x_1,\dots,x_N \rangle )= [\Lambda:\Lambda'] = [ W^{-1}\mathcal{O}_K^{n} : \mathcal{O}_K^{m} ],
\end{equation}
where $W^{-1} \mathcal{O}_K^n = \{ \alpha = (\alpha_1,\dots,\alpha_m) \in K^{m} \mid W \alpha \in \mathcal{O}_K^n\}$, which is an $\OK$-module in $K^{m}$. Let $W_{J}$ be the $m \times m$ minor of $W$ by selecting a subset of rows $J \subseteq \{1,\dots,m\}$ with $\card J = m$. Then by multiplying by adjoint matrices, it is clear that for $\alpha\in K^{m}$
\begin{equation}
W \cdot \alpha  \in \mathcal{O}_K^{n} \Rightarrow W_{J} \cdot \alpha \in \mathcal{O}_K^{m} \Rightarrow \det(W_J) \cdot \alpha  \in \mathcal{O}_K^m .
\end{equation}

Define $\mathcal{I}=\langle \rho \rangle    \subseteq \mathcal{O}_K$ to be the ideal generated by 
\begin{equation}
\rho = \prod_{J \in \binom{[n]}{m}} \det W_{J}.
\end{equation}
We see that if $\alpha \in \langle \tfrac{1}{\rho}  \rangle = \mathcal{I}^{-1}$, then we have
\begin{equation}
\mathcal{O}_K^{m} \subseteq W^{-1} \mathcal{O}_K^{n} \subseteq  ( \mathcal{I}^{-1} ) ^{m},
\end{equation}
since $\mathcal{I}^{-1}$ is the fractional ideal ``inverse'' of $\mathcal{I}$ defined as 
$\mathcal{I}^{-1} = \{ \kappa \in \OK \mid \kappa I \subseteq \OK\}$.
Note that $\mathcal{I}^{-1}$ is an $\OK$-module and so is $\mathcal{I}^{-1}/\OK$.
We are thus interested in simply calculating the number of solutions of 
\begin{equation}
W \cdot \alpha = 0 \pmod{\mathcal{O}_K}, \  \alpha \in (\mathcal{I}^{-1}/\OK)^{m}.
\end{equation}
In particular, we want to show that the number of solutions to this is equal, as in Equation (\ref{eq:claim2}), to $$\N(\langle \det W_{J} \rangle_{ \substack{ J \subseteq \{1,\dots,m\}\\ \card J = m}}).$$ 

This calculation can be done locally, with respect to each prime ideal $\mathcal{P}$ dividing $\mathcal{I}$. Since $\mathcal{I}$ is a principal ideal, multiplication by the generator $\rho$ gives us an isomorphism of $\OK$-modules as
\begin{equation}
\frac{\mathcal{I}^{-1}}{\OK} \simeq \frac{\mathcal{O}_K}{\mathcal{I}}.
\end{equation}
Factoring $\mathcal{I} = \mathcal{P}_1^{f_1}\mathcal{P}_2^{f_2} \dots \mathcal{P}_s^{f_s}$ and writing the sum of ideals generated by the $\det W_J$ as $\mathcal{P}_1^{e_1} \mathcal{P}_2^{e_2} \dots \mathcal{P}_s^{e_s}$, we have that $0 \le e_i \le f_i$ for each $i \in \{1,\dots,s\}$ and 
\begin{equation}
\frac{\mathcal{I}^{-1}}{\OK} \simeq \frac{\mathcal{O}_K}{\mathcal{I}}= \N(\mathcal{P}_1)^{e_1} \N(\mathcal{P}_2)^{e_2} \dots \N(\mathcal{P}_s)^{e_s}.
\end{equation}
Hence, the problem is reduced to showing that the number of solutions of the following is $\N(P)^{e_i}$ for each $i \in \{1,\dots,s\}$.
\begin{equation}
W \cdot \alpha = 0 \pmod{\mathcal{P}_i^{f_i}}, \  \alpha \in (\OK/\mathcal{P}_i^{f_i})^{m}.
\end{equation}

 This can be proved via induction as explained
in  \cite[Lemma 4.5]{S1967}. 

\begin{remark}
\label{re:finite_places_height}
Observe that for any $x_1,\dots,x_N \in K$, we get that the norm of the principal ideal generated by $x_1,\dots,x_N$ is 
\begin{equation}
	\N\left( \langle x_1, \dots, x_N \rangle \right)^{-1} =  \prod_{\substack{v \in M_K \\ v \nmid \infty}} \max_{1 \le i \le N} |x_i|_v
\end{equation}

\end{remark}

\section{Convex combinations lemma} 

We record here a general lemma that was used in some special instances in the paper. We hope that future literature around this topic could benefit from this idea.


\begin{lemma}
\label{le:convex_combinations}

Let $D \in M_{m \times n}(K)$ (not necessarily reduced) be a matrix given as 
\begin{equation}
	\label{eq:defi_of_D}
  D =  
  \begin{bmatrix}
	  \alpha_{11} & \alpha_{12} & \dots & \alpha_{1n} \\
	  \alpha_{21} & \alpha_{22} & \dots & \alpha_{2n} \\
		      & \vdots & \\
	  \alpha_{m1} & \alpha_{m2} & \dots & \alpha_{mn} \\
  \end{bmatrix}.
\end{equation}
Let $f:K_\mathbb{R}^{t} \rightarrow \mathbb{R}$ be the indicator function of a unit ball.
Then, we have that for any $c_1,\dots,c_n \in [0,1]$ such that $\sum c_i = 1$
\begin{equation}
\label{eq:integral_expression}
  \frac{1}{V(mtd)}
	\int_{K_\mathbb{R}^{ t \times m}} \prod_{j=1}^{n} f\Big(\sum_{i=1}^{m} \alpha_{ij} x_{i}\Big) dx~\le~
	\prod_{\sigma  : K \rightarrow \mathbb{C}} \Big( \sum_{J \in \binom{[n]}{m}} ({  \scriptstyle\prod_{j \in J} c_j })  |\sigma \left(\det (D_{J})\right)|^{2} \Big)^{-\frac{t}{2}}.
\end{equation}
Here the product on the right is over $d=r_1+2r_2$ embeddings of $K$ into $\mathbb{C}$.
\end{lemma}
\begin{proof}
Let $x = (x_1,\dots,x_m) \in (K_\mathbb{R}^{t})^{m}$. 
The integral is computing the volume of the set in $(K_\mathbb{R}^{t})^{m}$ satisfying 
\begin{align}
	\| \alpha_{11}x_1 + & \dots + \alpha_{1m} x_m \|^{2}\le 1,\\
  \| \alpha_{12}x_1 + & \dots + \alpha_{2m} x_m \|^{2}\le 1,\\
		      & ~ ~\vdots \\
  \| \alpha_{1n}x_1 + & \dots + \alpha_{mn} x_m \|^{2}\le 1 .
\end{align}
Adding all of these together with a weight of $c_i$ assigned to each respective condition, we get 
\begin{equation}
	\label{eq:quad_form_t_copies}
  \sum_{j=1}^{n} c_i \| \alpha_{1j} x_1 + \cdots + \alpha_{mj}x_m \|^{2} \le 1.
\end{equation}
This means that the set whose volume we are estimating is contained in the set of points satisfying inequality \eqref{eq:quad_form_t_copies}. The latter defines an ellipsoid whose volume is given by
\begin{equation}
	\frac{V(mtd)}{\vol\left(\tfrac{K_\mathbb{R}^{t \times m }}{ \mathcal{O}_{K}\sqrt{c_1} \omega_1 + \dots \OK \sqrt{c_n}\omega_n}\right)},
\end{equation}
where $\omega_1,\dots,\omega_n \in K_\mathbb{R}^{t}$ are the columns of $D$. 

Note that writing $K_\mathbb{R}^{t \times m} \simeq K_\mathbb{R}^{m \times t}$
the quadratic form defined by (\ref{eq:quad_form_t_copies}) is actually $t$ copies of the quadratic form in $K_{\mathbb{R}}^{m}$ defined by the same equation but with $x_1,\dots,x_m \in K_\mathbb{R}^{m}$ instead.
The result now follows from Lemma \ref{le:covolume}.
\end{proof}

\begin{lemma}
\label{le:covolume}
Let $D \in M_{m \times n}(K)$ be a full-rank matrix, for example, the one given in Equation (\ref{eq:defi_of_D}). 
Let $\Lambda \subseteq K_\mathbb{R}^{m}$ be the $\OK$-module generated by the columns of $D$. Then, the covolume of this lattice is 
\begin{equation}
	\prod_{\sigma  : K \rightarrow \mathbb{C}} \Big( \sum_{J \in \binom{[n]}{m}} |\sigma \left(\det (D_{J})\right)|^{2} \Big)^{\frac{1}{2}}.
\end{equation}
Here the product is $d=r_1+2r_2$ embeddings of $K \rightarrow \mathbb{C}$ and the inner sum is over all $m \times m$ minors of $D$.
\end{lemma}
\begin{proof}
	Follows from Proposition \ref{pr:result_about_volumes}, Part 1.
\end{proof}


\section{Rogers' integral formula} 
\label{ap:integralformula}

In this section, we record a proof of Rogers' integral formula. As mentioned before in Section 
\ref{se:average_over_lattice}, numerous starting points to obtain this formula can already be found in the literature. The purpose of this section is to make our paper self-contained.

The proof of Theorem \ref{th:weil} follows along the following steps:
\begin{itemize}
  \item Find an explicit Haar measure on $\SL_{t}(K_\mathbb{R}) / \SL_{t}(\OK)$.
  \item Find a ``coarse fundamenental domain'', also known as a Siegel domain for $\SL_{t}(K_\mathbb{R})/\SL_{t}(\OK)$. That is, find a domain $\mathfrak{S} \subseteq \SL_{t}(K_\mathbb{R})$ that is somewhat easy to integrate on and so that $\mathfrak{S} \cdot \SL_{t}(\OK) = \SL_{t}(K)$.
  \item Integrate the function $g \SL_{t}(\OK) \mapsto ( \sum_{v \in g \OK^{t}} f(v) )^{n}$ on this Siegel domain and show that the integral converges.
  \item Anticipate the final answer and conclude the integration formula by substituting suitable test functions.
\end{itemize}

We will directly state the results for the first two steps since these are fairly well-known. An interested reader can either refer to \cite{W58}, \cite{G21} or \cite{B19} for the details.

Let us first define the following notations. Recall that the norm map $\N: K_\mathbb{R} \rightarrow \mathbb{R}$ is the determinant of left multiplication.
\begin{definition} 
\label{de:kan}
We define the following.
\begin{align}
  \GL_k(K_{\mathbb{R}})  & = \{ g \in M_k(K_{\mathbb{R}}) \ | \ g \text{ is not a zero divisor }  \} , \\
  \mathcal{K} &  = \{ \kappa \in \SL_{t}(K_\mathbb{R}) \ | \ \kappa ^{*} \kappa = 1_{M_k(A)}\}, \\
  \mathcal{A} & = \{ a \in \SL_{t}(K_\mathbb{R}) \ | \ a \text{ is diagonal}, a_{ii} \text{ invertible}, \N(a_{ii}) >0 \},\\
  \mathcal{N} & =  \{ n \in \mathcal{G}(\mathbb{R}) \ | \ n \text{ is upper triangular with $1_{K}$ on the diagonal}\}.
\end{align}
\label{de:kan_defi}
\end{definition}

The three groups $\mathcal{K},\mathcal{A},\mathcal{N}$ defined above are Lie groups. In particular, $\mathcal{K}$ is a compact Lie group and $\mathcal{A}$ is an abelian group.

\begin{proposition}
\label{pr:haar_measure}
The map
\begin{align}
\mathcal{K} \times \mathcal{A} \times \mathcal{N}  & \rightarrow \SL_{t}(K_\mathbb{R})  \\
(\kappa, a,n) & \mapsto \kappa an
\end{align}
is a smooth surjective map. Furthermore, for any compactly supported continuous function $f : \SL_{t}(K_\mathbb{R}) \rightarrow \mathbb{R}$, the following defined a Haar measure on $\SL_{t}(K_\mathbb{R})$.
\begin{equation}
  f \mapsto \int_{\mathcal{K } \times \mathcal{A} \times \mathcal{N}} f( \kappa a n ) \left(\prod_{i<j} \frac{\N(a_{ ii })}{\N(a_{jj})}\right) d\kappa da dn.
\end{equation}
\end{proposition}

Now let us define a Siegel domain $\mathfrak{S} \subseteq \SL_{t}(K_\mathbb{R})$.

\begin{definition}
Let $K_\mathbb{R}^{(1)}$ be the set of unit norm elements of $K_\mathbb{R}$.
	Let $\omega_{1}\subseteq K_{\mathbb{R}}^{(1)}$ be a relatively compact set and let $c_1,c_2>0$. 
	Also, let $b_1, b_2, \dots , b_m$ be some elements of $\GL_t(K)$. 
	Recall the definition of $\mathcal{K},\mathcal{A},\mathcal{N}$ as defined in Definition \ref{de:kan_defi}.
	Then, we define
\begin{align}
  \underline{\mathcal{A}}^{\mathbb{R}} & = \{ a \in \GL_t(K_{\mathbb{R}})  \ | \ a'_{ij} = 0 \text{ for }i\neq j,a_{ij}' \in \mathbb{R}_{>0} \subset K_{\mathbb{R}} \}, \\
  \mathcal{A}_{\omega_1}^{(1)} & =  \{ a \in \mathcal{A}\ | \ a_{ii} \in \omega_{1}  \} ,\\
  \underline{\mathcal{A}}_{c_1}^{\mathbb{R}} & = \{ a' \in \underline{\mathcal{A}}^{\mathbb{R}} \ | \ a_{ii}' \in \mathbb{R}_{>0} \subset K_{\mathbb{R}}, a'_{ii} \le c_1 a'_{i+1,i+1} \},\\
  { \mathcal{A} }_{c_1}^{\mathbb{R}} & =  \{ a' \in \mathcal{A} \ | \ a'_{ii} \in \mathbb{R}_{>0} \subseteq K_{\mathbb{R}} ,  a'_{ii} \le c_1 a'_{i+1,i+1}  \} = \underline{\mathcal{A}}_{c_1}^{\mathbb{R}}  \cap \mathcal{A} ,\\
  \mathcal{N}_{c_2} & =  \{ n \in \SL_{t}(K_\mathbb{R}) \ | \ n \text{ is upper triangular with $1$ on diagonal }, \Tr(n_{ij}\overline{ n_{ij} })  < c_{2}  \},\\
  \fS^{1}& = \fS^{1}_{\omega_{1},c_1,c_2}  =  \mathcal{K} \mathcal{A}_{\omega_1}^{(1)} \underline{\mathcal{A}}_{c_1}^{\mathbb{R}} \mathcal{N}_{c_2},\\
  \fS& = \fS_{\omega_{1},c_1,c_2}  =  \left( \bigcup_{i=1}^{m} \fS^{1} b_i^{-1} \right) \cap \SL_{t}(K_\mathbb{R})  =  \bigcup_{i=1}^{m} \det(b_i)^{\frac{1}{dt}}  (\mathcal{K} \mathcal{A}^{(1)}_{\omega_1} \mathcal{A}_{c_1}^{\mathbb{R}} \mathcal{N}_{c_{2}}) b_i^{-1} .
\end{align}
Here, $d=[K:\mathbb{Q}]$ and $\N,\Tr: K_\mathbb{R} \rightarrow \mathbb{R}$ are the usual norm and trace maps. Also, $\overline{ (\ ) }:K_\mathbb{R} \rightarrow K_\mathbb{R}$ like in Equation \ref{eq:norm} is the coordinate-wise complex conjugation (or identity on real embeddings) using the identification $K_{\mathbb{R}} \simeq \mathbb{R}^{r_1} \times \mathbb{C}^{r_2}$. 
Note that $\det(b_i) \in (K \otimes \mathbb{R})^{*}$ so $\det(b_i)^{\frac{1}{dt}}$ is well-defined 
upto choosing some roots of unity in the complex embeddings of $K$.
\label{de:siegel_set}
\end{definition}

\begin{remark}\label{re:liesinQ}
	Note that $\{ b_i\}_{i=1}^{n}$ lie in $\GL_t(K)$. 
	This means that for some $N \in \mathbb{N}$, $N\cdot b_i \in M_t(\mathcal{O}_K)$. 
\end{remark}

The following result, which can be attributed to Weil \cite{W58}, holds:
\begin{theorem}
\label{th:bhc}
For some choice of $b_1,\dots,b_m \in \GL_{t}(K)$, some relatively compact set $\omega_{1} \subseteq K_{\mathbb{R}}^{(1)}$ and some $c_1,c_2 > 0$, the Siegel set $\fS$ defined in Definition \ref{de:siegel_set} satisfies the property that 
\begin{equation}
  \fS\cdot \SL_{t}(\OK)= \SL_{t}(K_\mathbb{R}).
\end{equation}
\end{theorem}

The amazing fact about $\fS$ is that integrating it along the Haar measure defined in Proposition \ref{pr:haar_measure} very easily verifies that $\SL_{t}(K_\mathbb{R})/\SL_{t}(\OK)$ has finite Haar measure. Showing that $\fS$ has finite volume is showing a special case of Borel and Harish-Chanda's breakthrough theorem \cite{BHC62} about bounded volumes of arithmetic quotients for this particular algebraic group.

This completes the first two steps of our proof outline at the start of this section. 
To simplify the notations a bit from here on out, we will use the notation of group schemes over $\mathbb{Q}$. 
Our main object of study is $\mathcal{G}$, which is the group $\Res^K_{\mathbb{Q}} \SL_t$ which is the rank-restriction from $K$ to $\mathbb{Q}$ of the group $\SL_t$. We will denote by $\mathcal{G}(\mathbb{R}) =\SL_{t}(K_\mathbb{R})$ the continuous group of real points of $\mathcal{G}$ and $\mathcal{G}(\mathbb{Q}) = \SL_{t}(K)$. The arithmetic subgroup $\Gamma = \SL_t(\OK)$ can be defined as the points of $\mathcal{G}(\mathbb{Q})$ that preserve the lattice $\OK^t$ under the natural left-action of $\mathcal{G}(\mathbb{R})$ on $K_\mathbb{R}^t$.

\begin{remark}
The role of $\SL_t(\OK)$ in the above definition of Siegel domain may be played by other arithmetic subgroups $\Gamma$ as well. A new Siegel domain mapping surjectively onto $\SL_t(K_\mathbb{R})/\Gamma$ may be found in each case. Note that the finite set $b_1,...,b_m$ above is sensitive to the choice of the arithmetic subgroup $\Gamma$.

\end{remark}

Moving on to the next step, we now use our Siegel domain $\fS$ to prove the following result.

\begin{proposition}
\label{pr:bounded_integral}
Suppose $t > n$.
Let $V_\mathbb{Q} = K^{t \times n}$ and let $\mathcal{G}(\mathbb{Q}) = \SL_{n}(K)$ be acting on the left. 
Let $\Lambda = {\OK}^{t \times n} \subseteq V_\mathbb{Q}$.
Thus $\Gamma = \SL_t(\mathcal{O}_K)$ is a discrete subgroup of $\mathcal{G}(\mathbb{R}) = \SL_{t}(K_\mathbb{R})$ that preserves the lattice $\Lambda$. Let $V_\mathbb{R} = V_\mathbb{Q} \otimes \mathbb{R}$.

Suppose that $f:V_\mathbb{R} \rightarrow \mathbb{R}$ is a bounded and compactly supported function that is integrable on every real subspace $W \subseteq V_{\mathbb{R}}$. Then, for any $0 < \varepsilon \le 1$, we have  that 
\begin{equation}
\int_{\mathcal{G}(\mathbb{R})/\Gamma} \left( \varepsilon^{dnt }\sum_{v \in g \Lambda} \left|f( \varepsilon v)\right| \right) dg  
\end{equation}
is uniformly bounded (independent of $\varepsilon$) from above by a function on $\mathcal{G}(\mathbb{R})/\Gamma$ whose integral is finite.
\end{proposition}
\begin{proof}

Let $ 0 < \varepsilon \le 1$ be a real number. Without loss of generality, we can replace $f$ by $|f|$ and therefore assume that $f$ is non-negative.
  
Recall the constructions of $\fS,\fS^{1}$ from Definition \ref{de:siegel_set}.
Let $R>0$ be such that $f$ is supported 
inside $B_{R}(0) \subset V_{\mathbb{R}}$, where $B_R(0)$ is a ball invariant under the action of the compact group $\mathcal{K}$ from Definition \ref{de:kan_defi}. We can create such a ball by averaging any quadratic form over $\mathcal{K}$.

Note that we have chosen the Haar measure on $\mathcal{G}(\mathbb{R})$ to be as given in Proposition \ref{pr:haar_measure}. This does not necessarily correspond to the probability measure on $\mathcal{G}(\mathbb{R})/\Gamma$. However, this will create a correction of at most a constant in our calculations, which does not affect the truthfulness of our proposition.

With this, we get that 
  \begin{align}
	  \int_{\mathcal{G}(\mathbb{R})/\Gamma}^{} \varepsilon^{dnt} \sum_{v \in \Lambda}{f}( \varepsilon g v ) dg 
	  & \ll \  
	  \varepsilon^{dnt} \int_{\fS}^{} 
	  \left( \sum_{v \in  \Lambda }^{} f( \varepsilon g v) \right) dg \\
	  & \ll   \varepsilon^{dnt} \int_{\fS}^{} \left(  (g \Lambda ) \cap B_{R/ \varepsilon }(0)\right) dg \\
    \le &    \varepsilon^{dnt} \sum_{i=1}^{m}\int_{\fS^{1}}  \#  \left(  (g b_{i}^{-1} \Lambda ) \cap B_{R/ \varepsilon }(0)\right)dg\\
    =    \sum_{i=1}^{m}\varepsilon^{dnt} \int_{\mathcal{N}_{c_2}} & \int_{\mathcal{A}_{\omega_1}^{(1)}} \int_{\mathcal{A}^{\mathbb{R}}_{c_1}} \int_{\mathcal{K}}  \# \left( \det(b_{i})^{\frac{1}{dt}}(\kappa a' a \eta) b_i^{-1}\Lambda \cap B_{R/\varepsilon}(0)\right)\ \prod_{i < j}^{} \left( \frac{a'_{ii}}{a'_{jj}} \right)^{d} d\kappa da' da d\eta .
  \end{align}

  We remind the reader that $d = [K:\mathbb{Q}]$.
  Since $B_{R/\varepsilon}(0)$ is chosen to be invariant under $K$, we know that for any $\kappa \in K$,
  \begin{align}
    \# \left( \det(b_i)^{\frac{1}{dt}} (\kappa a' a \eta) b_i^{-1}  \Lambda\cap B_{R/\varepsilon}(0) \right) 
    = \# \left( \det(b_i)^{\frac{1}{dt}}(a'a\eta) b^{-1}_i\Lambda \cap B_{R/\varepsilon}(0) \right).
  \end{align}

  Because of Remark \ref{re:liesinQ}, we know that there exists some $N \in \mathbb{N}$ such that $b_i^{-1} \Lambda \subseteq \frac{1}{N} \Lambda $ for every $1 \le i \le m$. Hence, we deduce that 
  \begin{align}
    \# \left( \det(b_{i})^{\frac{1}{dt}}(a'a\eta) b^{-1}_i\Lambda \cap B_{R/\varepsilon}(0) \right) 
    \ll
     &  \# \left( (a'a\eta) \Lambda \cap B_{R/\varepsilon}(0) \right).
  \end{align}
  Note that in this last step, we might have to move to a slightly bigger radius $R$ to accomodate for $\det(b_i)^{\frac{1}{dt}}$. This does not affect our conclusions. Now consider the set $$\mathcal{Y}= \{ a'an(a')^{-1} \ | \ {a' \in \mathcal{A}^{\mathbb{R}}_{c_1}, a \in \mathcal{A}^{(1)}_{\omega_1}, n \in \mathcal{N}_{c_2}} \} \subseteq \mathcal{G}(\mathbb{R}).$$ For $y \in \mathcal{Y}$, note that $y_{ij} = a'_{ii}a_{ii} n_{ij} (a'_{jj})^{-1} $. 
  Hence, $\Tr(\overline{ y_{ij} } y_{ij}) = \left(\frac{a'_{ii}}{a'_{jj}}\right)^{2}\Tr\left( \overline{ a_{ii} n_{ij} } a_{ii}n_{ij}\right)$.
  Here, $\left( \frac{a'_{ii}}{a'_{jj}} \right)$ is a positive real number bounded by $c_1^{j-i}$ because of the construction of $\mathcal{A}^{\mathbb{R}}_{c_1}$, and the other term is bounded because it continuously depends on $a_{ii} n_{ij}$ lying in a compact set. It follows that the set $\mathcal{Y}$ must lie inside a relatively compact set of $\mathcal{G}(\mathbb{R})$. 
  Furthermore, the set $\mathcal{Y}$ is only dependent on $c_1, c_2$ and $\omega_1$.

  Equipped with $\mathcal{Y} \subseteq \mathcal{G}(\mathbb{R})$, let $R' > 0$ be a radius such that 
  $\mathcal{Y}^{-1} B_{R}(0) \subseteq B_{R'}(0) \Rightarrow \mathcal{Y}^{-1}B_{R/\varepsilon}(0) \subseteq B_{R'/\varepsilon}(0)$. Then we write that 
  \begin{align}
    \# \left( (a'an) \Lambda \cap B_{R/\varepsilon}(0) \right) = & \# \left( (a'an(a')^{-1})a' \Lambda \cap B_{R/\varepsilon}(0) \right) \\
    \le & \# \left( a' \Lambda \cap \mathcal{Y}^{-1} B_{R/\varepsilon}(0) \right)  \\
    \le & \# \left( a' \Lambda \cap B_{R'/\varepsilon}(0) \right).
  \end{align}

 At this point, we invoke the identification $\Lambda =M_{t \times n}(\mathcal{O}_{K}) \subseteq M_{t \times n}(K) = V_\mathbb{Q}$. 
 After possibly replacing $R'$ with a bigger radius, we may assume that the norm on $V_\mathbb{R} \simeq {K_\mathbb{R}}^{t \times n}$ is the one given by 
 \begin{equation}
	 (x_1,\dots,x_{tn}) \mapsto \sum_{i,j} \Tr(\overline{ x_{ ij } }x_{ ij }).
 \end{equation}

 Then the value of the last expression 
 is equal to the number of integral solutions $(x_1, \dots, x_{tn}) \in M_{t \times n}(\mathcal{O})$ of the inequality
  \begin{align}
	  \sum_{i=1}^{t}  \sum_{j=1}^{n}  {a'_{ii}}^{2} \Tr(\overline{ x_{ij} } x_{ij}) \le \frac{ R'^{2}}{ \varepsilon^{2}}.
  \end{align}

  This is the number of points in a lattice intersecting with some ellipsoid. 
  By considering a suitable ``axis-parallel cuboid'' that contains this ellipsoid, an upper bound for the number of these lattice points in the ellipsoid is the following quantity.
  \begin{align}
    \prod_{i=1}^{t}\#\left\{  x \in \OK^{n} \ | \  \sum_{i=1}^{n}\Tr(\overline{ x_i  }x_i) \le \frac{ {R'}^{2} }{ {a'_{ii}}^{2} \varepsilon^{2} }  \right\}.
  \end{align}

  Each term in the product is the number of points in a ball of radius $R'/a_{ii}'\varepsilon$ in a ${dn}$-dimensional $\mathbb{R}$-vector space. Hence, there exist constants $B_1, B_2 > 0$ depending only on $\mathcal{O}_K, K $ such that 
  \begin{align}
    \#\left\{  x \in \mathcal{O}_K^{n} \ | \  \sum_{i=1}^{n}\Tr(\overline{ x_i } x_i) \le \frac{ {R'}^{2} }{ {a'_{ii}}^{2} \varepsilon^{2} }  \right\}
    \le B_1 + B_2 \left( \frac{R'}{a'_{ii} \varepsilon} \right)^{{dn}} ,
  \end{align}
  and therefore 
  \begin{align}
    & \int_{\mathcal{G}(\mathbb{R})/\Gamma} \varepsilon^{dnt} \sum_{v\in \Lambda}{f}( \varepsilon g v) dg  \\ 
    & \ll \sum_{i=1}^{m} \varepsilon^{dnt} \int_{\mathcal{N}_{c_{2}}} \int_{\mathcal{A}^{(1)}_{\omega_1}} \int_{\mathcal{A}^{\mathbb{R}}_{c_1}} \int_{\mathcal{K}}\prod_{i=1}^{t}  \left(  B_1 + B_2\left( \frac{R'}{a'_{ii} \varepsilon} \right)^{dn}    \right) \prod_{i < j}^{} \left( \frac{a'_{ii}}{ a'_{jj}} \right)^{d} d\kappa da' da d\eta \\ 
    &  \ll \int_{\mathcal{N}_{c_{2}}} \int_{\mathcal{A}^{(1)}_{\omega_1}} \int_{\mathcal{A}^{\mathbb{R}}_{c_1}} \int_{\mathcal{K}} 
    \left( \prod_{i=1}^{t} \left( B_1 \varepsilon^{{dn}} + B_2\left( \frac{R'}{a'_{ii} } \right)^{{dn}}   \right) \right) \prod_{i < j}^{} \left( \frac{a'_{ii}}{ a'_{jj}} \right)^{d} d\kappa da' da d\eta.
  \end{align}

  Now $\varepsilon \le 1 \Rightarrow  B_{1} \varepsilon^{dn} \le B_1$. Therefore, we can bound the integral above by 
  \begin{align}
    &  \int_{\mathcal{N}_{c_{2}}} \int_{\mathcal{A}^{(1)}_{\omega_1}} \int_{\mathcal{A}^{\mathbb{R}}_{c_1}} \int_{\mathcal{K}} 
    \left( \prod_{i=1}^{t} \left( B_1  + B_2\left( \frac{R'}{a'_{ii} } \right)^{{dn}}   \right) \right) 
    \prod_{i < j}^{} \left( \frac{a'_{ii}}{ a'_{jj}} \right)^{{d}} d\kappa da' da d\eta .
  \end{align}

  This last integral does not contain any appearance of $\varepsilon$. Note that for a decomposition of $g = \kappa a' a \eta$, the matrix $a'$ is unique. 
  Therefore, some appropriate scaling of the function $g \mapsto  \prod_{i=1}^{t}\left(B_1 + B_2(R' {a'_{ii}}^{-1})^{{dn}}\right)$ on a fundamental domain of $\mathcal{G}(\mathbb{R})/\Gamma$ contained in $\fS$
  is a dominating function of $g \Gamma \mapsto \varepsilon^{dk} \sum_{v \in g \Lambda}f({\varepsilon} v )$ if we prove that the integral above is convergent.

  The sets $\mathcal{K}, \mathcal{A}^{(1)}_{\omega_1}$ and $\mathcal{N}_{c_2}$ are compact and thus, $\int_{\mathcal{K}}{ d\kappa} \int_{\mathcal{N}_{c_2}}^{}dn $ and $\int_{\mathcal{A}^{(1)}_{\omega_{1}}}da$ are finite. Hence, we just need to show the finiteness of
  \begin{align}\label{eq:integral}
  \int_{\mathcal{A}^{\mathbb{R}}_{c_1}}  \left( \prod_{i=1}^{t} \left( B_1  + B_2\left( \frac{R'}{a'_{ii} } \right)^{dn}   \right) \right) \prod_{i < j}^{} \left( \frac{a'_{ii}}{ a'_{jj}} \right)^{d}  da' .
\end{align}
One can rewrite this as 
  \begin{align}\label{eq:integral}
  \int_{\substack{ \  \\ \  \\ a' = (a_1',\dots,a_t') \in \mathbb{R}_{>0}^{t}  \\ a_i' < c_1 a_{i+1}' \\ \prod_{i} a_{i'} = 1}}  \left( \prod_{i=1}^{t} \left( B_1  + B_2\left( \frac{R'}{a'_{i} } \right)^{dn}   \right) \right) \prod_{i < j}^{} \left( \frac{a'_{i}}{ a'_{j}} \right)^{d} \cdot  \prod_{i=1}^{t-1} \frac{da_{i}'}{a'_i} .
\end{align}
To show that this integral is finite is now a problem of multivariable calculus. We leave this for the reader to verify.
\end{proof}

\begin{corollary}
The same statement holds if we replace $\Lambda = \OK^{t \times n}$ by some finite index sublattice $\Lambda' \subseteq \frac{1}{N} \Lambda$ for some $N \in \mathbb{Z}_{\geq 1}$. 
\label{co:nonstandard_lattice}
\end{corollary}
\begin{proof}
We indeed observe that 
\begin{equation}
  \sum_{v \in \Lambda' }\varepsilon^{dnt} f(\varepsilon g v) \leq \sum_{v \in \Lambda} \varepsilon^{dnt} f\left(\frac{\varepsilon}{N} g v \right).
\end{equation}
The rest of the proof proceeds as before.
\end{proof}

With this in hand, we are almost ready to conclude Theorem \ref{th:weil}. The following proposition says that Theorem \ref{th:weil} is correct up to identifying some constants.

\begin{proposition}
	Consider the same notations as Proposition \ref{pr:bounded_integral}. Then, we have that for any function $f:V_\mathbb{R} \rightarrow \mathbb{R}$ that satisfies the hypothesis of \ref{pr:bounded_integral}, we get 
	\begin{equation}
	  \int_{\mathcal{G}(\mathbb{R})/\Gamma} \left( \sum_{v \in g \Lambda} f(v) \right) dg 
	  = 
	  f(0)+
	\sum_{m=1}^{{n}}\sum_{\substack{D \in M_{m \times n }(K) , \rank(D) = m \\ 
D \text{ is row reduced echelon}}} c_D  \cdot \int_{x \in K_\mathbb{R}^{t \times m }} f(x D ) dx,
			\label{eq:integral_formula_upto_constants}
	\end{equation}
	for some constants $c_D > 0$ such that the right hand side converges absolutely.
	\label{th:diamond_lemma}
\end{proposition}
\begin{proof}
Setting $\varepsilon = 1$ in Proposition \ref{pr:bounded_integral} tells us that the left hand side of Equation (\ref{eq:integral_formula_upto_constants}) is finite. 

We decompose $\Lambda$ as 
\begin{equation}
  \Lambda =  \bigsqcup_{ \Gamma \omega \in \Gamma \backslash \Lambda} \{ \gamma \omega \ | \ \gamma \in \Gamma\}.
\end{equation}

	So, we write that
	\begin{align}
	  \int_{\mathcal{G}(\mathbb{R}) / \Gamma} \left( \sum_{ v\in \Lambda} f(gv)  \right) dg
	  & = \int_{\mathcal{G}(\mathbb{R}) / \Gamma} 
	  \sum _{ \Gamma \omega \in \Gamma \backslash \Lambda}
	  \sum_{  \gamma \in \Gamma/ \Gamma_\omega } f(  g \gamma \omega ) dg\\
	  & =
	  \sum _{ \Gamma \omega \in \Gamma \backslash \Lambda}
	  \left( \int_{\mathcal{G}(\mathbb{R}) / \Gamma} 
	  \sum_{  \gamma \in \Gamma/ \Gamma_\omega } f( g \gamma \omega ) \right) dg.
	\end{align}

Denote $\mathcal{G}_{\omega} = \{g \in \mathcal{G} \mid g \omega = \omega \}$. This lets us define the real points $\mathcal{G}_{\omega}(\mathbb{R})$, the rational points $\mathcal{G}_{\omega}(\mathbb{Q})$ 
and the arithmetic points $\Gamma_{\omega} = \mathcal{G}_{\omega}(\mathbb{R}) \cap \Gamma $.
	We then observe that the following inclusion of groups holds.
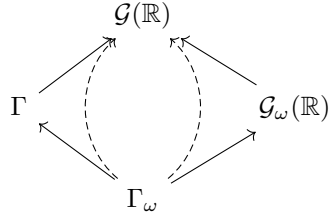
\begin{equation}
\begin{tikzcd}
                   & \mathcal{G}(\mathbb{R})                                                                            &                                       \\
\arrow[ru] \Gamma   & \                                                                                                  &  \mathcal{G}_\omega(\mathbb{R}) \arrow[lu] \\
                   & \Gamma_\omega  \arrow[lu] \arrow[ru] \arrow[uu, dashed, bend left=49] \arrow[uu, dashed, bend right=49] &                                      
\end{tikzcd}
\end{equation}
Notice that in the last expression, the integration is happening over $\mathcal{G}(\mathbb{R}) / \Gamma_\omega$ broken into two integrations along the left dashed path in the diagram. We want to instead break it down into the path on the right. 

Observe that the group $\mathcal{G}_{\omega}$ can be described in the following explicit way. Each $\omega \in \Lambda = \OK^{t \times n }$ is a matrix constituting of $n$ columns of vectors in $K^{t}$. The group $\mathcal{G}_{\omega}$ is then the pointwise stabilizer of the $K$-vector subspace of $K^{t}$ spanned by the columns of $\omega$. 
Therefore, we know that $\mathcal{G}_{\omega}$ is $g_0$-conjugate, 
for some  $g_0 \in \mathcal{G}_{\omega}(\mathbb{Q})$, to the stabilizer group of $K^{k} \times \{0\}^{t-k} \subseteq K^{t} $ for some $k \leq n$. This is a maximal parabolic $\mathbb{Q}$-subgroup of $\mathcal{G}(\mathbb{Q})$
and in particular this means that $\mathcal{G}_{\omega}(\mathbb{R})$ is a unimodular group.

Since $\mathcal{G}_{\omega}$ is a unimodular algebraic group, the homogeneous space $\mathcal{G}_{\omega}(\mathbb{R})/\Gamma_\omega$ has a well-defined Haar measure on which we can unfold our integral as follows:
	\begin{align}
	  \int_{\mathcal{G}(\mathbb{R}) / \Gamma} \left( \sum_{ v\in \Lambda} f(gv)  \right) dg
	  & =
	  \sum _{ \Gamma \omega \in \Gamma \backslash \Lambda}
	   \int_{ \gamma_1 \in \mathcal{G}(\mathbb{R}) / \mathcal{G}_\omega(\mathbb{R})}  
	   \left( 
	  \int_{  \gamma_2 \in \mathcal{G}_\omega(\mathbb{R}) / \Gamma_\omega }f( \gamma_1 \gamma_2 \omega )  d\gamma_2  \right) d \gamma_1.
	\end{align}
Now we know that $\gamma_2 \omega = \omega$. This gives us the following.
\begin{equation}
	\label{eq:diamand_winding}
	  \int_{\mathcal{G}(\mathbb{R}) / \Gamma} \left( \sum_{ v\in \Lambda} f(gv)  \right) dg
	=  \sum_{  \Gamma \omega  \in \Gamma \backslash \Lambda } \vol\left( \mathcal{G}_{\omega}(\mathbb{R})/ \Gamma_\omega \right) \int_{ \mathcal{G}(\mathbb{R}) / \mathcal{G}_{\omega}(\mathbb{R})   } f(g \omega) dg.
\end{equation}
We now revisit the explicit description of $\mathcal{G}_{\omega}$. As argued before, $\mathcal{G}_{\omega} \hookrightarrow \mathcal{G}$ is the pointwise stabilizer of the $K$-subspace of $W_{\mathbb{Q}} \subseteq K^{t}$ spanned by the columns of $\omega$.
This tells us that $\vol( \mathcal{G}_{\omega}/\Gamma_{\omega} )$ only depends on $W_{\mathbb{Q}}$. Furthermore, $\mathcal{G}(\mathbb{Q}) \omega  \simeq \mathcal{G}(\mathbb{Q})/\mathcal{G}_{\omega}(\mathbb{Q})$ can be identified with the set of all ordered ${K}$-bases of $W_{\mathbb{Q}}$ having a particular rank factorization. Let us make this a bit more precise.

According to the rank-factorization of $\omega \in M_{t \times n} (K)$, we can uniquely decompose 
\begin{equation}
  \omega = C \cdot D
\end{equation}
where $C \in M_{t \times k }(K)$ 
and $D \in M_{k \times n}(K)$ and $k=\rank(\omega)$ 
and $D$ is in a row-reduced echelon form. Then, the orbit $\mathcal{G}(\mathbb{Q}) \omega$ is the set of all 
the $\omega' \in M_{t \times n}(K)$ with $\rank \omega' = k$ and having the same $D$ in the rank factorization of $\omega' = C' \cdot D$.
Therefore, we get that $\mathcal{G}(\mathbb{Q}) \omega \subseteq M_{t \times k}(K) \cdot D$ forms a dense Zariski-open subset and hence $\mathcal{G}(\mathbb{R})\omega \subseteq M_{t \times k}(K_\mathbb{R}) \cdot D$ is a dense open subset in the usual topology. This bijection extends to a homeomorphism between $\mathcal{G}(\mathbb{R})/\mathcal{G}_{\omega}(\mathbb{R})$ and $\mathcal{G}\omega$ since it is a map from a quotient of locally compact groups to a locally compact space. Hence, continuous functions on $\mathcal{G}(\mathbb{R}) \omega$ are the same as continuous functions on $\mathcal{G}(\mathbb{R})/\mathcal{G}_{\omega}(\mathbb{R})$.
Finally, the integral
\begin{equation}
	f \mapsto \int_{M_{t \times k}(K_\mathbb{R})} f(x D) dx
\end{equation}
defines a left $\mathcal{G}(\mathbb{R})$-invariant measure on $\mathcal{G}(\mathbb{R}) \omega$. Therefore, up to a constant, this is equal to the integral $\int_{\mathcal{G}(\mathbb{R})/\mathcal{G}_{\omega}(\mathbb{R})} f(g\omega)dg$. Combining all the $\Gamma \omega \in \Gamma\backslash \Lambda$ with the same second term $D$ in the rank-factorization, we obtain that the right-hand side of Equation (\ref{eq:diamand_winding}) is exactly as we desire.
\end{proof}

\begin{remark}
As mentioned in Remark \ref{re:module_lattices}, the same methods and strategy can be used to show Rogers' integral formula for the larger space of module lattices. Proposition \ref{pr:bounded_integral}, Proposition \ref{th:diamond_lemma} and the proof below go through without any obstructions for the space $\SL_t(K_\mathbb{R})/\Gamma$, where $\Gamma$ is an arithmetic subgroup that stabilizes a maximal rank $\OK$-module
$\Lambda  \subseteq K \otimes \mathbb{R}^t$. Corollary \ref{co:nonstandard_lattice} is in this case applied to such a lattice stabilized by $\Gamma$. The space of module lattices, as discussed in \cite{DK22}, is a finite disjoint union of such quotients $\SL_t(K_\mathbb{R})/\Gamma$.
\end{remark}

We are now ready to finish our proof.

\begin{proof} {\bf (of Theorem \ref{th:weil})}

All that remains is to determine the constants $c_{D}$ in 
Proposition \ref{th:diamond_lemma}. Fix a row-reduced echelon matrix $D_0 \in M_{m_0 \times n}(K)$ of rank $m_0$. For $m_0 = 0$, there is nothing to prove so we assume $m_0 \geq 1$.

We shall determine $c_{D_0}$ by replacing $f$ by
\begin{equation}
  x \mapsto f(\varepsilon x) \ind_{M_{t \times m_0}(K_\mathbb{R}) D_0}(x) ,
\end{equation}
where $\ind_{M_{t \times m_0}(K_\mathbb{R})D_0}$ is the indicator function of the subspace $M_{t \times m_0}(K_\mathbb{R}) D_0 \subseteq M_{t \times n}(K_\mathbb{R})$. 

Equation (\ref{eq:integral_formula_upto_constants}) then yields, writing $\Lambda=M_{t \times n}(\OK)$: 
\begin{equation}
  \int_{\mathcal{G}(\mathbb{R})/\Gamma} \left( \sum_{v \in g \Lambda} f(\varepsilon v) \ind_{M_{t \times m_0}(K_\mathbb{R}) D_0}(x) \right) dg 
  = 
  f(0)+
\sum_{m=1}^{m_0}
\sum_{\substack{D \in M_{m \times n }(K), 
\rank(D) = m \\ 
D \text{ is row reduced echelon} \\ M_{t \times m}(K_\mathbb{R}) D \subseteq M_{t \times m_0}(K_\mathbb{R})D_0}} c_D  \cdot \int_{x \in K_\mathbb{R}^{t \times m }} f( \varepsilon x D ) dx.
		\label{eq:integral_formula_upto_constants_with_varepsilon}
\end{equation}
Indeed, row reduced echelon matrices $D$ contribute to the right hand side only if the intersection of subspaces $M_{t \times m}(K_\mathbb{R})D \cap M_{t \times m_0}(K_\mathbb{R}) D_{0}$ is not a subset of Lebesgue measure zero of $M_{t \times m}(K_\mathbb{R}) D$.\par
We now rewrite the left hand side of \eqref{eq:integral_formula_upto_constants_with_varepsilon} as follows: 
define the function $f_0(x)= f(xD_0)$ on $M_{t \times m_0}(K_\mathbb{R})$ as well as the lattice 
\begin{equation}
  \Lambda_0 = 
  \{C \in M_{t \times m_0}(\OK) \mid C \cdot D_0 \in M_{t \times n}(\OK)\} 
\end{equation}
in the Euclidean space $M_{t \times m_0}(K_\mathbb{R})$. Observe that, since $D_0$ is in reduced form, we have
\begin{equation}
  \Lambda \cap M_{t \times m_0}(K_\mathbb{R})D_0 = 
  \{C \in M_{t \times m_0}(K) \mid C \cdot D_0 \in M_{t \times n}(\OK)\} \cdot D_0=\Lambda_0\cdot D_0.
\end{equation}
We may therefore rewrite $ \sum_{v \in g \Lambda} f(\varepsilon v) \ind_{M_{t \times m_0}(K_\mathbb{R}) D_0}(x)$ as the lattice sum function 
$\sum_{v \in \Lambda_0} f_0 ( \varepsilon g v )$. 
The lattice $\Lambda_0$ is fixed by $\SL_{t}(\OK)$, so that Proposition \ref{pr:bounded_integral} and Corollary \ref{co:nonstandard_lattice}
with $n$ replaced by $m_0$
imply that the function 
\begin{equation}
g \SL_{t}(\OK) \mapsto \varepsilon^{m_0td}\sum_{v \in \Lambda_0} f_0 ( \varepsilon g v )
\end{equation}
is uniformly dominated by an integrable function on $\mathcal{G}(\mathbb{R})/\Gamma$. 
Multiplying Equation (\ref{eq:integral_formula_upto_constants_with_varepsilon}) by 
$\varepsilon^{m_0 td}$ on both sides, we obtain
\begin{align}
  \int_{\mathcal{G}(\mathbb{R})/\Gamma}\varepsilon^{m_0 td} \left( \sum_{v \in \Lambda_0} f_0 ( \varepsilon g v ) \right) dg 
  = & \varepsilon^{m_0td}f(0) 
   \\   + 
\sum_{m=1}^{m_0}
  \varepsilon^{( m_0-m )td}& 
\sum_{\substack{D \in M_{m \times n }(K), 
\rank(D) = m \\ 
D \text{ is row reduced echelon} \\ M_{t \times m}(K_\mathbb{R}) D \subseteq M_{t \times m_0}(K_\mathbb{R})D_0}} c_D    \cdot \int_{x \in K_\mathbb{R}^{t \times m }} f( x D ) dx. 
\end{align}
As $\varepsilon \rightarrow 0$, by dominated convergence the left hand side converges to the average over $\SL_{t}(K_\mathbb{R})/\SL_{t}(\OK)$ of 
\begin{equation}
	\lim_{\varepsilon\to 0}\varepsilon^{m_0 td} \sum_{v \in g\Lambda_0} f_0 ( \varepsilon v )=\vol(g\Lambda_{0})^{-1}
   \int_{M_{t \times m_0}(K_\mathbb{R}) } f(xD_{0})dx.
\end{equation}
On the right hand side, as $\varepsilon \rightarrow 0$ the surviving terms come from row reduced echelon matrices $D$ of rank $m_0$ satisfying $M_{t \times m}(K_\mathbb{R}) D \subseteq M_{t \times m_0}(K_\mathbb{R})D_0$. We deduce that $D$ must equal $D_0$ and the right hand side thus converges to 
\begin{equation}
  c_{ D_{0} } \int_{M_{t \times m_0}(K_\mathbb{R}) } f(xD_{0})dx.
\end{equation}
Therefore, comparing we obtain that 
\begin{equation}
  c_{ D_{0} }= \int_{\mathcal{G}(\mathbb{R})/\Gamma}\vol(g\Lambda_{0})^{-1}dg=[M_{t \times m_0}(\OK): \Lambda_{0}]^{-1}.
\end{equation}
The result follows since the index of $\Lambda_0$ in $M_{t \times m_0}(\OK)$ is exactly $\mathfrak{D}(D_0)^{t}$.

\end{proof}
\end{appendix}

\bibliographystyle{unsrt}
\bibliography{authfile}

  \bigskip
  \footnotesize

  N.~Gargava, \textsc{École Polytechnique Fédérale de Lausanne,
    Vaud, Switzerland}\par\nopagebreak
  \textit{E-mail address}:  \texttt{nihar.gargava@epfl.ch}

  \medskip

   V.~Serban, \textsc{New College of Florida,
    FL, USA}\par\nopagebreak
  \textit{E-mail address}:  \texttt{vserban@ncf.edu}

  \medskip

  M.~Viazovska, \textsc{École Polytechnique Fédérale de Lausanne,
    Vaud, Switzerland}\par\nopagebreak
                                                                                                     \textit{E-mail address}:  \texttt{maryna.viazovska@epfl.ch}

\end{document}